\documentclass[12pt,twoside]{amsart}
\usepackage[margin=3cm]{geometry}
\usepackage[colorlinks=false]{hyperref}
\usepackage[english]{babel}
\usepackage{graphicx,titling}
\usepackage{float}
\usepackage{amsmath,amsfonts,amssymb,amsthm}
\usepackage{lipsum}
\usepackage[T1]{fontenc}
\usepackage{fourier}
\usepackage{color}
\usepackage[latin1]{inputenc}
\usepackage{esint}
\usepackage{caption}
\usepackage{ccicons}

\makeatletter
\def\blfootnote{\xdef\@thefnmark{}\@footnotetext}
\makeatother

\newcommand\ccnote{
	\blfootnote{\copyright\,\, Ciprian Demeter, and Hong Wang}
	\blfootnote{\ccLogo\, \ccAttribution\,\, Licensed under a \href{https://creativecommons.org/licenses/by/4.0/}{Creative Commons Attribution License (CC-BY)}.}
}

\usepackage[export]{adjustbox}
\numberwithin{equation}{section}
\usepackage{setspace}
\renewcommand{\le}{\leqslant}
\renewcommand{\leq}{\leqslant}
\renewcommand{\ge}{\geqslant}
\renewcommand{\geq}{\geqslant}
\renewcommand{\mathbb}{\varmathbb}
\usepackage{fancyhdr}
\newtheorem{theorem}{Theorem}[section]

\newtheorem{defn}[theorem]{Definition}

\newtheorem{remark}[theorem]{Remark}
\newtheorem{lemma}[theorem]{Lemma}
\newtheorem{prop}[theorem]{Proposition}
\newtheorem{cor}[theorem]{Corollary}

\address{Ciprian Demeter, Indiana University, Bloomington, Department of Mathematics,  831 East 3rd St., Bloomington IN 47405}
\email{demeterc@iu.edu}
\medskip
\address{Hong Wang, Courant institute of mathematical sciences, New York University} 
\email{hw3639@nyu.edu}
\medskip

\newcommand{\R}{\mathbb{R}}

\newcommand{\Z}{\mathbb{Z}}

\newcommand{\cS}{\mathcal{S}}

\newcommand{\cP}{\mathcal{P}}
\newcommand{\cT}{\mathcal{T}}
\newcommand{\cA}{\mathcal{A}}
\newcommand{\cD}{\mathcal{D}}
\newcommand{\cL}{\mathcal{L}}

\newcommand{\bT}{\mathbf{T}}

\newcommand{\bp}{\mathbf{p}}

\newcommand{\cU}{\mathcal{U}}

\newcommand{\cB}{\mathcal{B}}

\newcommand{\br}{\mathbf{r}}

\newcommand{\cI}{\mathcal{I}}

\newcommand{\cQ}{\mathcal{Q}}

\newtheorem*{ack*}{Acknowledgment}

\begin{document}
	
\thispagestyle{empty}	

\begin{minipage}{0.28\textwidth}
	\begin{figure}[H]
		\includegraphics[width=2.5cm,height=2.5cm,left]{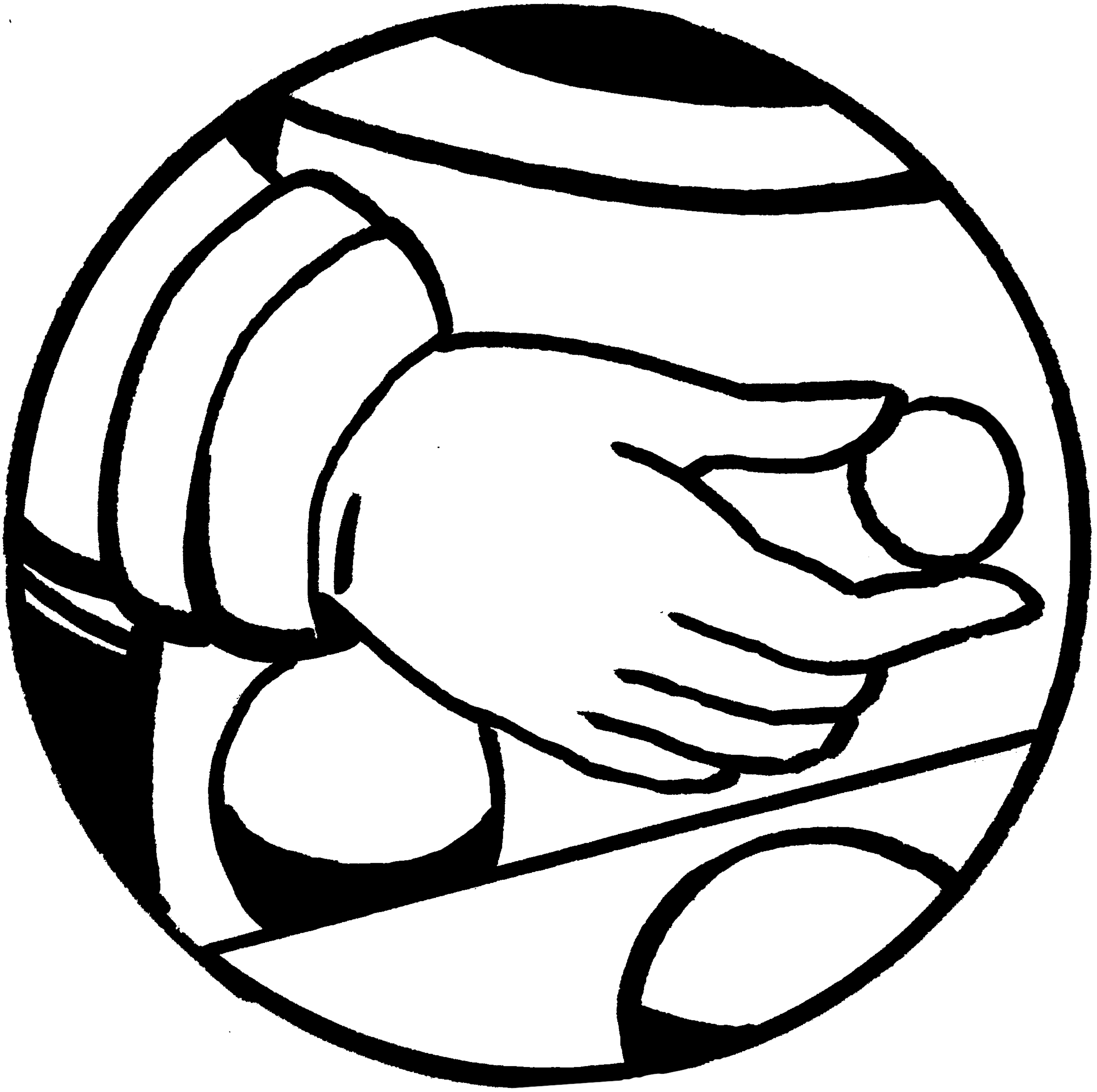}
	\end{figure}
\end{minipage}
\begin{minipage}{0.7\textwidth} 
	\begin{flushright}
		Ars Inveniendi Analytica (2025), Paper No. 1, 46 pp.
		\\
		DOI 10.15781/mt54-gc31
		\\
		ISSN: 2769-8505
	\end{flushright}
\end{minipage}

\ccnote

\vspace{1cm}

\begin{center}
	\begin{huge}
		\textit{Szemer\'edi-Trotter bounds for tubes and applications}

	\end{huge}
\end{center}

\vspace{1cm}


\begin{minipage}[t]{.28\textwidth}
	\begin{center}
		{\large{\bf{Ciprian Demeter}}} \\
		\vskip0.15cm
		\footnotesize{Indiana University, Bloomington}
	\end{center}
\end{minipage}
\hfill
\noindent
\begin{minipage}[t]{.28\textwidth}
	\begin{center}
		{\large{\bf{Hong Wang}}} \\
		\vskip0.15cm
		\footnotesize{Courant Institute, New York University}
	\end{center}
\end{minipage}

\vspace{1cm}


\begin{center}
	\noindent \em{Communicated by Larry Guth}
\end{center}
\vspace{1cm}


\noindent \textbf{Abstract.} \textit{We prove sharp estimates for incidences involving planar tubes that satisfy packing conditions. We apply them to improve the estimates for the Fourier transform of fractal measures supported on planar curves.	}
\vskip0.3cm

\vspace{0.5cm}


	\section{Introduction}
	
	Given a collection of $\delta\times 1$ tubes in $[0,1]^2$ we will denote by $\cP_r(\cT)$ ($\cP_{\gtrsim r}(\cT)$) the collection of all $\delta$-squares in  $\cD_\delta=\{[n\delta,(n+1)\delta]\times[m\delta,(m+1)\delta]:\;0\le n,m\le \delta^{-1}-1\}$ that intersect $\sim r$ ($\gtrsim r$) of the tubes.
	
	Here is our main result (see also its more general version, Theorem \ref{thm: mainbettt}). We refer the unfamiliar reader to Section \ref{sec2} for terminology.

	\begin{theorem}\label{thm: main} Let $s\in (0, 1/2]$.
		Suppose $\Lambda\subset \mathbb{S}^1$ is a $(\delta, s)$-set consisting of $\delta$-intervals (arcs) with cardinality $|\Lambda|\sim \delta^{-s}$. Suppose that for each $\delta$-arc $\theta\in \Lambda$ there is a $(\delta, 1-s)$-set $\cT_{\theta}$ of $\delta$-tubes in the direction (normal to)  $\theta$, with $|\cT_{\theta}|\sim \delta^{-1+s}$.  Let $\cT=\cup_{\theta\in \Lambda} \cT_{\theta}$.
		
		Then for each  $\upsilon>0$ there is $C_\upsilon>0$ such that for each $1\le r\lesssim \delta^{-s}$ we have
		\begin{equation}
			\label{Mmaineq}
			|\cP_r(\cT)|\le C_\upsilon\delta^{-\upsilon}\frac{|\cT|^2}{r^3}.
		\end{equation}
	\end{theorem}Theorem  \ref{thm: main} can be seen as a tube analogue of the celebrated Szemer\'edi--Trotter Theorem \cite{SzTr}, which asserts that given a family $\cL$ of lines in the plane, the number $|\cP_r(\cL)|$ of points that intersect $\sim r$ of these lines is  $O(\frac{|\cL|^2}{r^3}+\frac{|\cL|}{r})$. Note that in our case $\frac{|\cT|}{r}\lesssim \frac{|\cT|^2}{r^3}$, since $s\le 1/2$, $r\lesssim \delta^{-s}$ and $|\cT|\sim \delta^{-1}$.
	\smallskip
	
	Let us  draw a few comparisons with results in the literature.
	Inequality \eqref{Mmaineq} was proved in \cite{GSW} for well spaced collections consisting of $\sim \delta^{-1}$ many tubes. In particular, Theorem 1.1 in \cite{GSW} with $W=\delta^{-1/2}$ proves our Theorem \ref{thm: main} when $s=\frac12$, in the  case when both $\Lambda$ and each $\cT_\theta$ are well spaced.

	We remark that Theorem \ref{thm: main}  implies (in fact it is equivalent to) the incidence bound
	\begin{equation}
		\label{j rejijio tjuioj}
		\cI(\cP,\cT)=|\{(p,T)\in\cP\times\cT:\;p\cap T\not=\emptyset\}|\lessapprox (|\cP||\cT|)^{2/3},	\end{equation}
	for  arbitrary collections $\cP\subset\cD_\delta$. Without any requirement on $\cT$, the best possible upper bound is $\cI(\cP,\cT)\lesssim \delta^{-1/3}(|\cP||\cT|)^{2/3}$. See \cite{FOP}, Theorem 4.3 in \cite{FuRe}, or Exercise 7.5 in \cite{Dembook} for slightly different arguments.
	
	Inequality \eqref{j rejijio tjuioj} is known to hold when $\cT$ is a $(\delta,\alpha)$-set with cardinality $\sim \delta^{-\alpha}$ and $\cP$ is a $(\delta,\beta)$-set with cardinality $\sim \delta^{-\beta}$, with $\alpha+\beta=3$. This is a particular case of Theorem 4.3 in \cite{FuRe}. When $\alpha=1$, this would require $\cP$ to have essentially maximal cardinality $\delta^{-2}$.
	
	It is also worth comparing our result with Theorem 5.2 in \cite{FuRe}. A particular case of it asserts that if $\cT$ is a $(\delta,1)$-set with cardinality $\sim \delta^{-1}$ and if $\cP_T$ is a $(\delta,1-s,K)$-Katz-Tao set of $\delta$-squares intersecting $T$, then, writing $\cP=\cup_{T\in\cT}\cP_T$,
	\begin{equation}
		\label{djefurfureg}
		\sum_{T\in\cT}|\cP_T|\lessapprox (\delta^{-1}K|\cP||\cT|)^{1/2}.
	\end{equation}
	It is easy to see that if $\cT$ is as in Theorem \ref{thm: main}, then $\cP_T:=\cP_r(\cT)\cap T$ is a $(\delta,1-s,\frac{\log \delta^{-1}}{r\delta^s})$-Katz-Tao set. Indeed, fix a ball $B_\rho$ of radius $\rho$, and estimate $|\cP_T\cap B_\rho|$ by double counting incidences between the squares in $\cP_T$ and their associated bushes
	$$r|\cP_T\cap B_\rho|\lesssim \sum_{\delta\lesssim \eta\lesssim 1:\;dyadic}(\eta/\delta)^s\frac1\eta(\rho\eta/\delta)^{1-s}.$$
	This is because
	
	1. there are $\lesssim (\eta/\delta)^s$ directions of tubes $T'\in\cT$ intersecting $T$ at angle $\sim \eta$
	
	2. each such $T'$ intersects $T$ along $\sim \eta^{-1}$ many $\delta$-squares
	
	3. for each direction there are $\lesssim (\rho\eta/\delta)^{1-s}$ tubes in this direction that intersect $B_\rho$.
	
	Thus  \eqref{djefurfureg} gives
	$$r|\cP_r(\cT)|\lessapprox (\delta^{-1}(r\delta^s)^{-1}|\cP_r(\cT)||\cT|)^{1/2},$$
	which may be written as
	$$|\cP_r(\cT)|\lessapprox \delta^{-s}\frac{|\cT|^2}{r^3}.$$
	This however falls short of proving \eqref{Mmaineq} by the factor $\delta^{-s}$.
	\medskip

	Theorem \ref{thm: main} is sharp in a number of ways.
	
	First, inequality \eqref{Mmaineq} is false when $s>\frac12$, even when $\frac{|\cT|^2}{r^3}$ is replaced with $\frac{|\cT|^2}{r^3}+\frac{|\cT|}{r}$. Indeed, start with a $(\delta,1-s)$-set $\cP\subset \cD_\delta$ of squares intersecting $[0,1]\times \{0\}$, with cardinality $\delta^{s-1}$. Consider also an AD-regular set $\Lambda$ as in Theorem \ref{thm: main}. By this we mean that each $\rho$-arc centered at some $\theta\in\Lambda$ intersects $\sim (\rho/\delta)^s$ arcs in $\Lambda$.  Finally, for each $p\in \cP$, let $\cT_p$  be the collection of $\delta$-tubes passing through $p$, with one tube for each direction in $\Lambda$. The collection $\cT=\cup_{p\in \cP}\cT_p$ is as in Theorem  \ref{thm: main}. An easy computation shows that for each $p\in\cP$ and each $1\le r\lesssim \delta^{-s}$ we have $|\cP_r(\cT_p)|\sim \delta^{-1-s}r^{-\frac{s+1}{s}}$. Thus  $|\cP_r(\cT)|\gtrsim \delta^{-2}r^{-\frac{s+1}{s}}$, and this is much larger than $\delta^{-2}r^{-3}+\delta^{-1}r^{-1}$ when $s>1/2$ and $r\ll \delta^{-s}$. Note also that $\frac{s+1}{s}=3$ when $s=\frac12$.
	\smallskip

	Second, given \eqref{Mmaineq}, the index $1-s$ for the dimension of $|\cT_\theta|$ cannot be replaced with  any other $t$. Indeed, if $t>1-s$ then we have the area estimate $\sum_{T\in\cT}|T|=\delta^{-1-s-t}\gg \delta^{-2}$. A random choice of $\cT$ subject to the constraints in our theorem would lead to the (essentially uniform) scenario where a large fraction of $\cD_\delta$ is in $\cP_r(\cT)$, with $r\approx \delta^{-s}$. It is easy to see that \eqref{Mmaineq} fails for such a configuration. On the other hand,  \eqref{Mmaineq} shows that the number of incidences between tubes in $\cT$ and squares in $\cD_\delta$ is dominated by
	$$\sum_{r\ge 1}\frac{|\cT|^2}{r^2}\lessapprox \delta^{-2(t+s)}.$$
	Since there are $\sim \delta^{-1}|\cT|$ incidences, we find that $t$ must be $\ge 1-s$.
	\smallskip
	
	Third, it was pointed out to us by Joshua Zahl that Theorem \ref{thm: main} is false if $s=1/2$ when the non-concentration requirement on $\Lambda$ is removed, even assuming $\cT$ is a $(\delta,1)$-set with cardinality $\sim \delta^{-1}$. Let us consider the following train track-type of example. The set $\cT$ will consist of $\delta^{-1/2}$ many bushes. Each bush contains $\sim \delta^{-1/2}$ many tubes that intersect an axes parallel rectangle $U$ of height $\delta$ and horizontal length $\delta^{1/2}$. Consecutive rectangles $U$ have  vertical separation $ \delta^{1/2}$ between them. It is easy to see that this $\cT$ is a $(\delta,1)$-set. Each of the $\delta$-squares lying inside one  $U$ is in $\cP_r(\cT)$, with $r\sim \delta^{-1/2}$.
	Since there are $\delta^{-1}$ such squares, this violates \eqref{Mmaineq}. However, $\Lambda$ fails to be a $(\delta,\frac12)$-set, as it  consists of $\delta^{-1/2}$ consecutive $\delta$-intervals.
	\smallskip
	
	Theorem \ref{thm: main} implies the following more general version. This implication is discussed in Section \ref{sec4}, see Proposition \ref{cor: mainKT}. 
	\begin{theorem}\label{thm: mainbettt} Let $s\in (0, 1/2]$.
		Suppose $\Lambda\subset \mathbb{S}^1$ is a $(\delta, s,K_1)$-Katz-Tao set consisting of $\delta$-intervals $\theta$. Suppose that for each  $\theta\in \Lambda$ there is a $(\delta, 1-s,K_2)$-Katz-Tao-set $\cT_{\theta}$ of $\delta$-tubes in the direction (normal to)  $\theta$.  Let $\cT=\cup_{\theta\in \Lambda} \cT_{\theta}$.
		
		Then for each  $\upsilon>0$ there is $C_\upsilon>0$ such that for each $1\le r\lesssim \delta^{-s}$ we have
		\begin{equation}
			\label{Mmaineq}
			|\cP_r(\cT)|\le C_\upsilon\delta^{-1-\upsilon}(K_1K_2)^2\frac{|\cT|}{r^3}.
		\end{equation}
	\end{theorem}

	As an application of Theorem \ref{thm: mainbettt} we  prove the following result.
	\begin{theorem}
		\label{t1}
		Let $\Gamma$ be the graph of a $C^3$ function $\gamma:[-1,1]\to\R$ satisfying the nonzero curvature condition $\min_{-1\le x\le 1}|\gamma''(x)|>0$.
		Let $0<s\le \frac12$. Let $\mu$ be a Borel measure supported on $\Gamma$
		satisfying the Frostman condition
		\begin{equation}
			\label{e4}
			\mu(B(y,r))\lesssim r^s
		\end{equation}
		for each $y\in \R^2$ and each $r>0$. Then  for each ball $B_R\subset\R^2$ of radius $R\ge 1$ and each $\epsilon>0$ we have
		\begin{equation}
			\label{e1}
			\|\widehat{\mu}\|^6_{L^6(B_R)}\lesssim_\epsilon R^{2-2s-\frac{s}{4}+\epsilon}.
		\end{equation}	
	\end{theorem}
	It was conjectured in \cite{O} that the exponent of $R$ should be $2-3s$.
	In the same paper, the exponent $2-2s$ is proved as a consequence of a sharp $L^4$ estimate.
	Also, an epsilon improvement over  $2-2s$ was achieved when $\Gamma$ is the parabola. A similar improvement was recovered by a different argument in \cite{DD}, one that works for all $\Gamma$.
	
	Basic constructions discussed in \cite{O} show that the exponent $2-3s$ is best possible in \eqref{e1} for each $s\le \frac12$, while the exponent $1-s$ is best possible in the range $\frac12\le s\le 1$. The recent paper \cite{O2} proves the expected result with essentially sharp exponent $1-s+\epsilon$ in the range $s\ge \frac23$, when $\Gamma$ is the parabola. The case $s\le \frac12$ is expected to be more difficult, as it entails full square root cancellation.

	For context, Theorem \ref{t1} is of similar strength with the following statement about the $\delta$-energy of a finite $\delta$-separated set of points  $S\subset \Gamma$. This is also a consequence of our argument in the last section.
	\begin{theorem}
		\label{wejfiuriofuriurioeug}	
		Let $s\le \frac12$. Assume that for each $y\in\R^2$ and each $\delta\le r\lesssim 1$ we have  $$|S\cap B(y,r)|\lesssim (r/\delta)^s.$$
		Let $$\mathbb{E}_{3,\delta}(S)=|\{(s_1,\ldots,s_6)\in S^6:\;|s_1+s_2+s_3-s_4-s_5-s_6|\lesssim \delta\}|.$$
		Then $$\mathbb{E}_{3,\delta}(S)\lesssim_\epsilon(\delta^{-s})^{\frac{7}{2}+\frac{1}{4}+\epsilon}.$$
	\end{theorem}
	The exponent $\frac72+\frac14$ is slightly worse than the one from the estimate
	$$\mathbb{E}_{3}(S)=|\{(s_1,\ldots,s_6)\in S^6:\;s_1+s_2+s_3=s_4+s_5+s_6\}|\lesssim |S|^{\frac72},$$
	proved in \cite{BB} for arbitrary finite subsets of the circle, and then in \cite{BD} for arbitrary finite subsets of the parabola. In both cases, the sharp exponent is conjectured to be $3$.
	
	It is possible that a refinement of our argument would lead to the better exponents $\frac72+\epsilon$ in Theorem \ref{wejfiuriofuriurioeug} and $2-2s-\frac{s}{2}+\epsilon$ in \eqref{e1}. This is certainly true in the AD-regular case, see Remark \ref{jhhdhfuihugururgurguiy}.
	\smallskip
	
	One of the main tools we will use in the proof of Theorem \ref{thm: main} is the following recent result of Kevin Ren and the second author, that solves the Furstenberg set Conjecture.
	\begin{theorem}[\cite{RW}]
		\label{Fur}	
		Let $0<s\le 1$ and $0\le t\le 2$.  Then for each $\eta,\eta_1>0$ there exists $\delta(\eta,\eta_1,s,t)>0$ and $\epsilon(\eta,\eta_1,s,t)>0$ with $\lim_{\eta\to 0\atop{\eta_1\to 0}}\epsilon(\eta,\eta_1,s,t)=0$,  such that the following holds.
		
		Consider a pair $(\cP,\cT)$ consisting of $\delta$-squares and $\delta$-tubes in $[0,1]^2$ with $\delta<\delta(\eta,\eta_1,s,t)$ such that
		\\
		\\
		(a) $\cP$ is a $(\delta,t,\delta^{-\eta})$-set
		\\
		\\
		(b) for each $p\in\cP$ there is a $(\delta,s,\delta^{-\eta_1})$-set $\cT(p)\subset\cT$ of tubes intersecting $p$, with cardinality $\sim M$.	
		
		Then
		\begin{equation}
			\label{irufurfur9fu9}
			|\cT|\gtrsim \delta^{-\min\{t,\frac{s+t}{2},1\}+\epsilon(\eta,\eta_1,s,t)}M.
		\end{equation}
	\end{theorem}
	The equivalent dual formulation (when $\eta_1=\eta$) of this result assumes $\cT$ is a $(\delta,t,\delta^{-\eta})$-set and each $T\in\cT$ intersects all squares in some $(\delta,s,\delta^{-\eta})$-set $\cP_T\subset \cP$ with cardinality $M$. Note that this forces $M\gtrsim \delta^{-s+\eta}$. Then
	$$
	|\cP|\gtrsim \delta^{-\min\{t,\frac{s+t}{2},1\}+\epsilon(\eta,s,t)}M.$$	
	Our Theorem \ref{thm: main} easily implies (a more general form of) this estimate in the case $t=1$ and $\frac12\le s\le 1$, when $\cT$ has the special structure in Theorem \ref{thm: main} (with $s$ replaced by $1-s$).
	\begin{cor}Let $\eta>0$ and let $\frac12\le s\le 1$. The following is true for  $\delta<\delta(\eta)$.
		
		Assume  $\cT=\cup_{\theta\in\Lambda}\cT_\theta$, where $\Lambda$ is a $(\delta,1-s)$-set with size $\sim \delta^{s-1}$, and each $\cT_\theta$ is a $(\delta,s)$-set of tubes in the direction $\theta$, with size $\sim \delta^{-s}$. Consider also a set $\cP$ such that each $T\in\cT$ intersects  $\sim M$ squares $\cP_T\subset\cP$. Note that we do not make any non-concentration assumption on $\cP_T$.
		
		Then $$|\cP|\gtrsim \delta^{-\frac{1}{2}+\eta}M^{3/2}.$$
	\end{cor}
	\begin{proof}
		We double count the incidences.
		We have, with $L\sim (\delta M)^{-\frac{1}{2}}\delta^{-\eta}$
		\begin{align*}
			\delta^{-1}M\le \cI(\cP,\cT)&\lesssim \sum_{r\text{: dyadic}}r|\cP_r(\cT)\cap \cP|\\&=\sum_{r\text{: dyadic }\ge L}r|\cP_r(\cT)|+\sum_{r\text{: dyadic }<L}r|\cP|\\&\lesssim \delta^{-\eta}\sum_{r\text{: dyadic }\ge L}r\frac{\delta^{-2}}{r^3}+L|\cP|\\&\lesssim \delta^{-\eta-2}L^{-2}+L|\cP|.
		\end{align*}
		We have used Theorem \ref{thm: main} for the first term. Our choice of $L$ (and small $\delta$) implies that the second term must dominate, leading to
		$$|\cP|\gtrsim \delta^{-1}M/L\gtrsim \delta^{-\frac{1}{2}+\eta}M^{3/2}.$$
		
		When $\cP_T$ is also a $(\delta,s_0,\delta^{-\eta})$-set for some $s_0\in (0,1)$, this lower bound is in fact better than the one in the Furstenberg set problem
		$$|\cP|\gtrsim \delta^{-\frac{s_0+1}2+\epsilon}M.$$	
		This is because $M$ could potentially be much larger than $\delta^{-s_0}$.
		
	\end{proof}
	\medskip
	
	\textbf{Strategy:} In Section \ref{sec2} we recall the essential concepts and basic tools we will employ throughout the paper. Most of them were known in the form we present them. One example is the multi-scale decomposition in Lemma \ref{Lipdec}, that produces both good and bad intervals. However, we also prove a new version of it, Lemma  \ref{lem: goodinterval2},  that delivers only good intervals. This will be used in the form of Theorem \ref{justonegood}.
	
	Section \ref{sec3} presents a few probabilistic constructions that are  used in Section \ref{sec4} to relate Theorem \ref{thm: main} to more general versions of it. They will in turn be used in the main argument in Section \ref{sec7}.
	
	In Section \ref{sec5}, the main theorem is analyzed in the simpler context when the tubes intersecting the $r$-rich squares have certain regularity. It is this part only that makes use of the estimate \eqref{irufurfur9fu9} on the size of Furstenberg sets. Even in this special setting, Theorem \ref{thm: delta-s} only delivers a dichotomy.
	In one case, it gives (a superficially stronger form of)  the  desired Szemer\'edi--Trotter-type bound. The second possible scenario is the high concentration of the rich squares at a certain scale $\Delta$. In Section \ref{sec6} we use Fourier analysis to show how this scenario leads to an abundance of tubes intersecting squares  of a certain scale larger than $\delta$.
	
	To make Theorem \ref{thm: delta-s} applicable, in Section \ref{sec:new} we use Lemma  \ref{lem: goodinterval2} and a combinatorial argument to create rich subcollections of two-ends tubes.
	
	All these ingredients will be fed into the induction on scales argument presented in Section \ref{sec7}. Instead of using the full  multi-scale decomposition, this latter argument relies on a two-scale increment provided by Theorem  \ref{justonegood}. It is in this part of the argument that the novel decomposition in Theorem  \ref{justonegood} is proving to be crucial. Our main assumption $s\le \frac12$ will only be used in Step 3 of the argument in Section \ref{sec7}.
	
	In Section \ref{sec9} we combine the incidence bound in Theorem \ref{thm: main} with decoupling to prove Theorem \ref{t1}.
	
	Most of the arguments are extremely delicate, as we need to keep track of various epsilons and scales, and to sharply quantify losses in a way that makes induction on scales viable.
	\\
	\\
	\textbf{Notation:}
	For positive  quantities $A,B$, typically depending on the scale $\delta$, we write either $A\lesssim B$ or $A=O(B)$ if there is a universal constant $C$ independent of the scale $\delta$ such that $A\le CB$. The notation $A\sim B$ means $A\lesssim B\lesssim A$.

	We write $o_\epsilon(1)$ for a not necessarily positive quantity (e.g. $\epsilon^{1/2}$, $-100\epsilon$) that goes to zero, as $\epsilon$ goes to zero. In particular $O(\epsilon)$ implies $o_\epsilon(1)$, but not vice-versa. Whenever we use $o_\epsilon(1)$, it will mean that the sign of this quantity is irrelevant.

	We write $A\sim B\delta^{o_\epsilon(1)}$ if $$A/B,B/A\lesssim \delta^{o_\epsilon(1)},$$
	with the implicit constants depending on $\epsilon$, but not on $\delta$.
	
	The symbol $\lessapprox$ will be reserved only to denote arbitrarily small $\delta^{-\upsilon}$ losses, such as  $\log(1/\delta)^{O(1)}$ losses.

	\begin{ack*}
		We are grateful to Tuomas Orponen and Joshua Zahl   for helpful conversations. We also thank the referee, whose careful reading and comments have led to improving the presentation. The first author is  partially supported by the NSF grants DMS-2055156 and  DMS-2349828. The second author is supported by NSF CAREER DMS-2238818
		and NSF DMS-2055544
	\end{ack*}
	\section{Preliminaries}
	\label{sec2}
	The reader may assume that all scales $\delta$, $\rho$, etc are dyadic, meaning that they are in $2^{-\mathbb{N}}$. Smaller $\delta$-squares are either inside or disjoint from a given larger $\rho$-square.
	We start by reviewing standard concepts and terminology.

	For a set $A$ in a metric space, we denote by  $|A|_\rho$ the smallest number of balls of radius $\rho$ needed to cover $A$. Typically for us, $A$ will be a disjoint union of  either $\delta$-intervals (in $\R$ or $\mathbb{S}^1$), $\delta$-squares, or $\delta$-tubes, for some $\delta\le \rho$. There is a map sending points to lines
	$$(a,b)\mapsto\{(x,y):\;y=ax+b\}$$
	that generates a metric on the space of lines. $\delta$-tubes are (length one segments of) images of $\delta$-squares under this map. We denote by $\cD^{dual}_\delta$ the collection of $\delta$-tubes $T$.
	
	Due to the loose distinction between relations such as $p\cap T\not=\emptyset$ and $p\subset CT$ (with $C=O(1)$), we will think of $\delta$-tubes as $O(\delta)$-neighborhoods (in the Euclidean metric) of unit line segments.
	
	If $\cT$ are $\delta$-tubes, the quantity $|\cT|_\rho$ is comparable with the cardinality of the smallest family of $\rho$-tubes that covers all $T\in\cT$.
	
	In all these instances, if $\rho=\delta$, then $|A|_\delta$ coincides with the cardinality $|A|$ of $A$, understood as a collection of intervals, squares, tubes. We will interchange the use of $|A|_\delta$ and $|A|$ throughout the arguments.

	\begin{defn}[$(\delta, s,C)$-set ]
		For  $s\in (0, d]$, a non-empty set $A\subset \R^d$ with diameter $\lesssim 1$ is called a $(\delta, s,C)$-set if
		\[
		|A\cap B(x,r)|_{\delta}\le C r^s |A|_{\delta}, \quad \forall x\in \R^d, r\in [\delta, 1].
		\]
	\end{defn}
	Testing the inequality with $r\sim 1$ shows that $C\gtrsim 1$. When $C\sim 1$, we simply call it a $(\delta,s)$-set.
	
	\begin{defn}[$(\delta,s, K)$-Katz-Tao set]
		For  $s\in (0, d]$, a non-empty  set $A\subset \R^d$ with diameter $\lesssim 1$ is called a $(\delta, s, K)$-Katz-Tao set if
		\[
		|A\cap B(x,r)|_{\delta}\le K(r/\delta)^s, \quad \forall x\in \R^d, r\in [\delta, 1].
		\]\end{defn}
	Testing the inequality with $r\sim\delta$
	shows that $K\gtrsim 1$.
	It will always be implicitly assumed that $K=O(\delta^{s-d})$, as each set is trivially Katz-Tao with respect to any larger $K$.

	Let $\delta\le \rho'\le \rho$.
	Given $\cP\subset \cD_\delta$ we write $\cP_\rho=\{\bp\in\cD_\rho:\;\bp\text{ contains some }p\in\cP\}$. Note that $\cP_\delta=\cP$. Given $\bp\in\cD_{\rho}$ we write
	$$\cP_{\rho'}[\bp]=\{\bp'\in\cP_{\rho'}:\;\bp'\subset \bp\}$$
	and simply
	$$\cP[\bp]=\{p\in\cP:\;p\subset \bp\}.$$
	Given  $\cT\subset\cD_\delta^{dual}$ we write $\cT^\rho=\{\bT\in\cD_\rho^{dual}:\;\bT\text{ contains some }T\in\cT\}$. Given $\bT\in\cD_{\rho}^{dual}$ we write
	$$\cT^{\rho'}[\bT]=\{\bT'\in\cT^{\rho'}:\;\bT'\subset \bT\}$$
	and simply
	$$\cT[\bT]=\{T\in\cT:\;T\subset \bT\}.$$
	
	One of the typical defects of $\delta$-sets is that they may have different concentrations inside disjoint balls of a given radius. Another one is that, while it is not hard for a set $S$ to be a $(\delta,s, O(1))$-set for $s$ close enough to zero, the size of $S$ may be much larger than $\delta^{-s}$. The next definition and Lemma  \ref{largeunifset} repair the first defect, and take a step towards fixing the second one.

	\begin{defn}[Uniform sets]
		\label{defunif}
		Given $\epsilon>0$, let $T_\epsilon$ satisfy $T_\epsilon^{-1}\log(2T_\epsilon)=\epsilon$. Given $0<\delta\le 2^{-T_\epsilon}$, pick the largest natural number $m$ such that $\frac1m\log_2\delta^{-1}\ge T_\epsilon$ and let $T=T(\delta,\epsilon)=\frac1m\log_2\delta^{-1}$. Note that $T_\epsilon\le T\le 2T_{\epsilon}$, and $\delta=2^{-mT}$.

		A set
		$\cP\subset \cD_\delta$ is called $\epsilon$-uniform if for each  $\rho=2^{-jT}$, $0\le j\le m$  and each $\bp_1,\bp_2\in\cP_\rho$  we have that
		$$|\cP[\bp_1]|=|\cP[\bp_2]|.$$
		A set $\cT\subset \cD_\delta^{dual}$ is called $\epsilon$-uniform if  for each  $\rho=2^{-jT}$, $0\le j\le m$ and each $\bT_1,\bT_2\in\cT^\rho$  we have that
		$$|\cT[\bT_1]|= |\cT[\bT_2]|.$$\end{defn}
	Given the correspondence between squares and tubes, the two definitions are in sync with each other. We will typically formulate auxiliary results for squares, but they hold equally well for tubes. A similar definition also holds for intervals on the real line, or on $\mathbb{S}^1$.
	
	From now on, it will always be implicitly assume that $\delta$ is small enough, depending on the context.
	\begin{lemma}
		\label{goodallscales}	
		Given $\epsilon>0$, there is $C(\epsilon)>0$ such that the following holds. 	
		For each $\epsilon$-uniform set $\cP\subset \cD_\delta$, each  $\delta\le\rho\le 1$, and for each $\bp_1,\bp_2\in\cP_\rho$  we have that
		$$C(\epsilon)^{-1}|\cP[\bp_1]|\le |\cP[\bp_2]|\le C(\epsilon)|\cP[\bp_1]|.$$
	\end{lemma}
	\begin{proof}
		For each $2^{-jT}<\rho<2^{-(j-1)T}$, we use the fact that each $\rho$-square contains at most $2^{2T}$ many $2^{-jT}$-squares. Thus the inequalities are satisfied with the constant $2^{2T}$. We take
		$C(\epsilon)=2^{4T_\epsilon}$, since this only depends on $\epsilon$.

	\end{proof}
	
	\begin{remark}
		\label{approxunifsets}	
		Lemma \eqref{goodallscales} shows that each $\epsilon$-uniform set is also an approximate $\epsilon$-uniform set. By this we mean that the equality in Definition \ref{defunif} is replaced with a double inequality involving a constant $C(\epsilon)$. We will apply this observation  with $\upsilon=\upsilon(\epsilon)$. What matters is that $C(\epsilon)=C(\epsilon^{-1}(\upsilon))= O_\upsilon(1)$, meaning that the constant only depends on $\upsilon$. All results and arguments continue to hold for such an approximate $\upsilon(\epsilon)$-uniform set, and the implicit constants will carry further dependence on $\upsilon(\epsilon)$.
	\end{remark}

	\begin{lemma}
		\label{tradim}	
		If $\cP$  is an $\epsilon$-uniform $(\delta,s, C_1)$-set and $\delta\le \rho$, then $\cP_\rho$ is a $(\rho,s, C(\epsilon)C_1)$-set. Moreover, its cardinality satisfies $|\cP_\rho|\gtrsim C_1^{-1}\rho^{-s}$.
	\end{lemma}
	\begin{proof}
		Let $\rho\le \rho'\le1$ and let $\bp'\in\cP_{\rho'}$. Let $n'$ be the average value of $|\cP[\bp]|$ for $\bp\in \cP_\rho[\bp']$. Then $n'|\cP_\rho[\bp']|\le C_1|\cP|(\rho')^s$. Let $n$ be the average value of $|\cP[\bp]|$ for $\bp\in \cP_\rho$. Then $n|\cP_\rho|= |\cP|$. Uniformity implies that $n\le C(\epsilon)n'$. The combination of these shows that $|\cP_\rho[\bp']|\le C(\epsilon)C_1|\cP_\rho|(\rho')^s$.
		
		The inequality $|\cP_\rho|\gtrsim C_1^{-1}\rho^{-s}$ follows since $n\le C_1|\cP|\rho^s.$
		
	\end{proof}

	We also recall Lemma 2.15 in \cite{OS2}.
	\begin{lemma}
		\label{largeunifset}
		Let $\epsilon>0$ and $\delta<\delta(\epsilon)$. Then each  set  $\cP\subset \cD_\delta$  has an $\epsilon$-uniform subset containing at least a $\sim \delta^\epsilon$-fraction of the original set.
	\end{lemma}
	
	\begin{lemma}
		\label{fromgentoexact}	
		Each $\epsilon$-uniform $(\delta,s,K)$-Katz-Tao set $S\subset \cD_\delta$ can be partitioned into $K\delta^{o_\epsilon(1)}$ many sets $S_i$ which are $(\delta,s,O_\epsilon(1))$-Katz-Tao sets. The quantities $o_\epsilon(1),O_\epsilon(1)$ are independent of $K$.
	\end{lemma}
	\begin{proof}
		Recall that $\delta=2^{-Tm}$, with $T$ depending only on $\epsilon$.

		Start testing from the bottom of the tree of dyadic scales. Let $\rho_1$ be the smallest scale of the form $2^{-Tl}$ such that the number $M_1$ of squares in $S$ inside a nonempty $\rho_1$-square $Q_1$ (this number is the same for all such $Q_1$) satisfies $M_1>  (\rho_1/\delta)^{s}$. If no such scale exists, the desired conclusion is immediate.  Create $\le 2M_1/(\rho_1/\delta)^{s}$ collections, each with $\le (\rho_1/\delta)^{s}$ squares inside each nonempty $Q_1$. 
		
		We repeat the process. Let $\rho_2>\rho_1$ be the smallest scale (again, assuming it exists) of the form $2^{-Tl}$ such that the number $M_2$ of squares $Q_1\in S_{\rho_1}[Q_2]$ inside a nonempty $\rho_2$-square $Q_2$ satisfies $M_2>  (\rho_2/\rho_1)^{s}$. Create $\le 2M_2/(\rho_2/\rho_1)^{s}$ collections of squares $Q_1$, each with $\le (\rho_2/\rho_1)^{s}$ such squares inside each nonempty $Q_2$.
		
		The selection process stops when there are no eligible scales left. The algorithm produces  scales $\rho_1,\ldots,\rho_L$, numbers $M_1,\ldots,M_L$ and sets $S_1,\ldots, S_I$. Each set $S_i$ satisfies for each dyadic $r$-square $Q$ with  $r\ge \delta$
		$$|S_i\cap Q|\le 2^{2T} (r/\delta)^s.$$
		Note that $2^{2T}$ is $O_\epsilon(1)$.

		We next prove that  $I\le K\delta^{-\frac1T}$. The $i^{th}$ iteration of the selection process increases the number of sets by a multiplicative factor of at most $2M_i/(\rho_i/\rho_{i-1})^s$, with the convention that $\rho_0=\delta$. Thus, we have
		$$I\le 2^{m}\frac{M_1\cdot\ldots\cdot M_L}{(\rho_L/\delta)^s}.$$
		Note that $M_1\cdot\ldots\cdot M_L$ is the number of $\delta$-squares inside a nonempty $\rho_L$-square. By hypothesis, this is $\le K(\rho_L/\delta)^s$. Recall that $2^m=\delta^{-\frac1T}$ is $\delta^{o_\epsilon(1)}$ in our notation.

	\end{proof}
	
	\begin{lemma}
		\label{useful}	
		Let $\epsilon>0$. Each $(\delta,s,K)$-Katz-Tao set $S$ can be partitioned into $O(\delta^{-\epsilon}\log(1/\delta))$ many $\epsilon$-uniform $(\delta,s,K)$-Katz-Tao sets.
	\end{lemma}
	\begin{proof}
		Each subset of a $(\delta,s,K)$-Katz-Tao set inherits this property. Consider an $\epsilon$-uniform subset $S_1$ of $S$, with size $\gtrsim  \delta^{\epsilon}|S|$, cf. Lemma \ref{largeunifset}. Then consider an $\epsilon$-uniform subset $S_2$ of $S\setminus S_1$, with size $\gtrsim \delta^{\epsilon}|S\setminus S_1|$. Repeat the process. Note that it takes at most  $n=O(\delta^{-\epsilon})$ steps to achieve half size $|S\setminus (S_1\cup\ldots\cup S_n)|<|S|/2$. We repeat this until the leftover has size $O(\delta^{-\epsilon})$. Each of the remaining  $s\in S$ is a one-element uniform set.
		
	\end{proof}

	We recall the following multi-scale decomposition, whose formulation has evolved in recent years, see e.g.  \cite{Sh},  \cite{SW} and \cite{OS2}.
	\begin{lemma}
		\label{Lipdec}	
		For every $\epsilon>0$ there is $\tau=\tau(\epsilon)>0$ such that the following holds: for each  non-decreasing $C$-Lipschitz  function $f:[0,1]\to [0,\infty)$ there is a partition
		$$0=a_1<a_2<\ldots<a_{J+1}=1$$
		and a sequence $$ \sigma_1<\sigma_2<\ldots<\sigma_J\le C$$ such that
		
		(i) for each $j$ we have $\tau\le a_{j+1}-a_j$	
		
		(ii) for each $1\le j_1<j_0\le J$ we have $$\sum_{j=j_1}^{j_0}\sigma_j(a_{j+1}-a_{j})\ge f(a_{j_0+1})-f(a_{j_1})-\epsilon$$
		
		(iii) for each $j$ we have
		$$f(x)\ge f(a_j)+\sigma_j(x-a_j),\;a_j\le x\le a_{j+1}.$$
	\end{lemma}
	\begin{proof}
		Items (i) and (iii) as well as (ii) for $j_1=1$, $j_0=J$ follow by combining Lemma 5.20 and 5.21 in \cite{SW}. See also Lemma 2.10 in \cite{OS2}. The fact that (ii) holds for all $j_0,j_1$ follows from this and  summation of (iii) (with $x=a_{j+1}$) over $j$.

	\end{proof}
	
	In previous literature Lemma \ref{Lipdec} was typically used to classify the intervals  $[a_j,a_{j+1}]$ as either ``good" or ``bad".
	We will prove a more robust version of the previous lemma, which guarantees all intervals are good. The key difference is that between (ii) in Lemma \ref{Lipdec} and the corresponding \eqref{eq: good2''} in Lemma \ref{lem: goodinterval2}, the latter being what defines an interval as being good. This difference is very subtle, and can only be fully appreciated in the argument from Section \ref{sec7}. More precisely, the $(\rho_0/\delta)^{o_\epsilon(1)}$-loss in \eqref{jewie9r93rr8948t958t9} is finely tuned into the induction on scales. On the other hand, (ii) in Lemma \ref{Lipdec} would only produce the very harmful $(1/\delta)^{o_\epsilon(1)}$-loss. We mention that this refined version proved in the next lemma is only needed once in our paper, in the afore-mentioned argument from Section \ref{sec7}. That proof requires the first interval to be good. For the other two applications in Sections \ref{sec5} and \ref{sec:new},
	Lemma \ref{Lipdec} would suffice.
	\begin{lemma}\label{lem: goodinterval2}
		For each small enough $\epsilon>0$
		we let 	
		\begin{equation}
			\label{fpoifruig9i096iy9i56'}
			\epsilon_0=\epsilon^{2\epsilon^{-1}}.
		\end{equation}	Then for each  non-decreasing $1$-Lipschitz  function $f:[0,1]\to [0,1]$  there exists a partition
		$$0=A_1<A_2<\ldots<A_{L+1}=1$$
		and a sequence
		$$0\le t_1<t_2<\ldots<t_L\le 1$$
		such that
		and  for each $1\le l\le L$
		\begin{equation}
			\label{sdjffudniugurtgurt8'}
			A_{l+1}-A_l\ge \epsilon_0\epsilon^{-1},
		\end{equation}
		
		\begin{equation}\label{eq: good'}
			f(x)\geq f(A_l)+t_l (x-A_l) -\epsilon( A_{l+1}-A_l),\; A_l\leq x\leq A_{l+1},
		\end{equation}
		\begin{equation}\label{eq: good2''}
			f(A_{l+1})\leq f(A_l)+ (t_l +3\epsilon) (A_{l+1}-A_l),
		\end{equation}

		\begin{equation}
			\label{sdjffudniugurtgurt8jdfuygtugit'}
			t_1\le f(1)-f(0)+\epsilon.
		\end{equation}

	\end{lemma}
	\begin{proof}
		Apply  Lemma \ref{Lipdec} to $f$ with $\epsilon$ replaced by $\epsilon_0$. The lemma produces the numbers $0\le a_j\le 1$ and $0\le \sigma_j\le 1$. Only (ii) and (iii) in the lemma  will be used.
		
		Write $s_k=\epsilon k$, $0\le k\le \epsilon^{-1}-1$.  Let $I_{k}$ be the union of the intervals $[a_j,a_{j+1}]$ with $\sigma_j\in [s_k, s_{k+1})$  (if $k=\epsilon^{-1}-1$, we include the right endpoint, allowing for $\sigma_j\in [s_k, s_{k+1}]$). Since $\sigma_j<\sigma_{j+1}$ for all $1\leq j\leq J-1$, $I_{k}$ is a union of consecutive intervals, and so it is either a nonempty interval, or the empty set. Also,  $\{I_{k}\}_{k=0}^{\epsilon^{-1}-1}$ form a partition of $[0,1]$, and $I_{{k+1}}$ follows right after $I_{k}$.

		Let $k^{(1)}$ be the index that corresponds to  the largest value  $|I_{k}|$. Write $I_{k^{(1)}}=[a,b]$.  If there are multiple choices, we choose an arbitrary one.  We have $|I_{k^{(1)}}|\geq \epsilon$. In particular, due to \eqref{fpoifruig9i096iy9i56'}	 we have
		\begin{equation}
			\label{kjehfcherufhurceiyui}
			|I_{k^{(1)}}|\ge \epsilon_0\epsilon^{-1}.
		\end{equation}

		Let $I_{k_l}=[a_{ll}, a_{lr}]$ be an interval on the left hand side of $I_{k^{(1)}}$ such that $|I_{k_l}|\ge  \epsilon^{2}|I_{k^{(1)}}|$ and  $|I_k|<\epsilon^{2}|I_{k^{(1)}}|$ for each $k_l<k<k^{(1)}$. Such an interval may not exist, in which case we let $a_{ll}=a_{lr}=a$.  But if it does, it is unique.

		Similarly, we let $I_{k_r}=[a_{rl}, a_{rr}]$ be an interval on the right hand side of $I_{k^{(1)}}$ such that $|I_{k_r}|\ge \epsilon^2 |I_{k^{(1)}}|$ and $|I_k|<\epsilon^{2}|I_{k^{(1)}}|$ for each $k^{(1)}<k<k_r$. If such an interval does not exist, we let $a_{rl}=a_{rr}=b$.
		
		We choose $[a_{lr},a_{rl}]$ to be the first interval  in our partition of $[0,1]$, and we call it $[A_{l_1},A_{l_1+1}]$. We assign it the value $t_{l_1}=s_{k^{(1)}}$. Note that \eqref{sdjffudniugurtgurt8'} is satisfied due to \eqref{kjehfcherufhurceiyui}.

		We claim that
		\begin{equation}
			\label{roi9ir9gi9gi09ig}
			f(x)\geq f(a_{lr}) + s_{k^{(1)}}(x-a_{lr}) -\epsilon(a_{rl}-a_{lr}),  \;\;a_{lr}\leq x \leq a_{rl}\end{equation}  and
		\begin{equation}
			\label{jicfhirficpreiofwi}
			f(a_{rl})\leq f(a_{lr})+ (s_{k^{(1)}} +3\epsilon) (a_{rl}-a_{lr}).
		\end{equation}
		
		We first prove \eqref{roi9ir9gi9gi09ig}.
		Recall that $[a_{lr},a]$ is the union of (at most $\epsilon^{-1}$) intervals $I_k$, each with length $<\epsilon^2|I_{k^{(1)}}|$. Thus
		\begin{equation}
			\label{dkforeifoirpipig}
			a-a_{lr}\leq \epsilon|I_{k^{(1)}}|\leq \epsilon (a_{rl}-a_{lr}).
		\end{equation}
		Similarly,
		\begin{equation}
			\label{dkforeifoirpipigdoi}
			a_{rl}-b\leq \epsilon (a_{rl}-a_{lr}).
		\end{equation}
		Recall that $s_{k^{(1)}}\leq 1$. Thus,    when $x\in (a_{lr}, a]$,
		by monotonicity and \eqref{dkforeifoirpipig}
		$$f(x) \geq f(a_{lr}) \geq f(a_{lr}) +s_{k^{(1)}}(x-a_{lr}) -\epsilon (a_{rl} -a_{lr}).$$
		When $x\in [a,a_{rl}]$, then due to Lemma \ref{Lipdec} (iii) and \eqref{dkforeifoirpipig}
		\begin{align*}
			f(x)&\geq f(a) +s_{k^{(1)}} (x-a) \\&\geq f(a_{lr}) + s_{k^{(1)}} (x-a)+s_{k^{(1)}}(a-a_{lr}) -\epsilon(a_{rl}-a_{lr})\\&=f(a_{lr}) + s_{k^{(1)}} (x-a_{lr}) -\epsilon(a_{rl}-a_{lr}).
		\end{align*}
		Next we verify \eqref{jicfhirficpreiofwi}. By first using the 1-Lipschitz property, then Lemma \ref{Lipdec}	(ii) we find
		\[
		f(a_{rl}) \leq f(b)+ ( a_{rl}-b)  \leq \epsilon_0 + f(a_{lr})+  (s_{k^{(1)}}+\epsilon)(b-a_{lr}) + (a_{rl}-b).
		\]
		Since (recall \eqref{dkforeifoirpipigdoi}) $a_{rl}-b\leq \epsilon(a_{rl}-a_{lr})$ and (recall \eqref{kjehfcherufhurceiyui}) $\epsilon_0 \leq \epsilon(b-a)\le \epsilon(a_{rl}-a_{lr})$, we have
		\[
		f(a_{rl}) \leq  f(a_{lr}) + (s_{k^{(1)}}+3\epsilon)(a_{rl}-a_{lr}).
		\]
		
		If $a_{ll}>0$, we iterate the same process to merge $I_{k^{(2)}}:=I_{k_l}$ with intervals $I_k$ on the left hand side of $I_{k_l}$ that satisfy $|I_k|<\epsilon^2|I_{k_l}|$. Similarly, if $a_{rr}<1$, we merge $I_{k^{(3)}}:=I_{k_r}$
		with intervals on the right hand side of $I_{k_r}$ that satisfy $|I_k|<\epsilon^2|I_{k_r}|$. This will produce the next two intervals $[A_{l_2},A_{l_2+1}],\;[A_{l_3},A_{l_3+1}]$ in our partition, and the corresponding values $t_{l_2},t_{l_3}$. Recall that we have control over the drop in size
		$$|I_{k^{(2)}}|,\;|I_{k^{(3)}}|\ge \epsilon^2|I_{k^{(1)}}|\ge \epsilon^3.$$
		
		We  iterate this process until we exhaust all intervals. We create the last interval of the partition by using $I_{k^{(L)}}$. Note that $|I_{k^{(L)}}|\ge |I_{k^{(1)}}|\epsilon^{2(L-1)}\ge \epsilon^{2L-1}$. Since there are only $\epsilon^{-1}$ intervals $I_k$ to start with, we have that $L\le \epsilon^{-1}$. When combined with \eqref{fpoifruig9i096iy9i56'} this gives
		$$|I_{k^{(L)}}|\ge \epsilon_0\epsilon^{-1}.$$
		As we have seen earlier, it is this inequality that allows for the verification of \eqref{jicfhirficpreiofwi} for the last selected  interval $[A_{l_{L}},A_{l_{L+1}}]$.	Our choice for small $\epsilon_0$ guarantees that this inequality is preserved throughout the whole selection process. The key is that we knew a priori the  upper bound $\epsilon^{-1}$ on the number of steps. Note also that \eqref{sdjffudniugurtgurt8'} is immediate and
		\eqref{eq: good'} follows as in the first step of the iteration.
		
		Finally, the fact that $t_l$ is increasing follows from the fact that $\sigma_j$ is increasing, while \eqref{sdjffudniugurtgurt8jdfuygtugit'} follows by summing \eqref{eq: good'}	  with $x=A_{l+1}$ and using the fact that $t_1$ is the smallest among all $t_l$.
		
	\end{proof}

	\begin{defn}
		We define the branching function $\beta:[0,m]\to[0,m]$ of an $\epsilon$-uniform set  $\cP$ of $\delta$-intervals with parameters $T,m$  to be
		$$\beta(j)=\frac{\log_2|\cP|_{2^{-jT}}}{T}, \;0\le j\le m$$
		for integers, and then we interpolate linearly. A similar definition holds for tubes and squares.
	\end{defn}
	This function is $1$-Lipschitz.

	The next theorem is one of the main novelties in our paper. It is a direct consequence of  Lemma \ref{lem: goodinterval2} applied to the branching function. Predecessors of this result have established themselves as chief tools in fractal analysis. They essentially produce a sequence of scales
	$\delta=\delta^{A_{L+1}}>\delta^{A_{L}}>\ldots>\delta^{A_{1}}=1$ such that, apart from negligible $\epsilon$-losses,  the dyadic tree of each uniform set $S$ consists of $O(1)$ many  layers, each of which is a (rescaled) $(\delta^{A_{l+1}-A_l},t_l,O(1))$-set with cardinality $\sim \delta^{(A_l-A_{l+1})t_l}$. This is almost as good as $S$ being a $(\delta,t,O(1))$-set with size $\sim \delta^{-t}$. Such a set would only have one layer, $L=1$, and  $t_1=t$.

	\begin{theorem}[Multi-scale decomposition with  good intervals]
		\label{justonegood}	
		
		For each small enough $\epsilon>0$
		we let
		\begin{equation}\label{riogitgoitophiophi}
			\epsilon_0=\epsilon^{2\epsilon^{-1}}.
		\end{equation}
		Then for each $\epsilon_0$-uniform set $S$ of $\delta$-intervals (or equivalently, $\delta$-tubes intersecting a fixed $\delta$-square)
		there exists a partition
		$$0=A_1<A_2<\ldots<A_{L+1}=1$$
		and a sequence
		$$0\le t_1<t_2<\ldots<t_L\le 1$$
		such that
		
		(i) $A_{l+1}-A_l\ge \psi(\epsilon):=\epsilon_0\epsilon^{-1}$,
		
		(ii) $\log_{1/\delta}(\frac{|S|_{\delta^{A_{l+1}}}}{|S|_{\delta^{A_{l}}}})\le (t_l+3\epsilon)(A_{l+1}-A_l)$,
		
		(iii) for each  $I\in S_{\delta^{A_l}}$ the $\times \delta^{-A_l}$ rescaled copy of $S_{\delta^{A_{l+1}}}[I]$
		is a $(\delta^{A_{l+1}-A_l},t_l,O(\delta^{(A_l-A_{l+1})\varphi(\epsilon)}))$-set with size in the interval $$[ O(\delta^{(A_l-A_{l+1})(t_l-\varphi(\epsilon))}), O(\delta^{(A_l-A_{l+1})(t_l+\varphi(\epsilon))})]$$
		
		(iv)  $t_1\le \log_{1/\delta}|S|+\epsilon$.

		The nonnegative function $\varphi(\epsilon)$ satisfies $\lim_{\epsilon\to0}\varphi(\epsilon)=0$. Moreover, $L$ is independent of $\delta$.
	\end{theorem}
	While we may take $\varphi(\epsilon)=3\epsilon$, the exact value of $\varphi(\epsilon)$ will be irrelevant. The same applies to $\psi$, whose only requirement is to be $>0$. Calling these quantities $\varphi(\epsilon)$ and $\psi(\epsilon)$ will help us distinguish their role in our later arguments. See for example \eqref{inforr} and \eqref{inforrtrei}.

	There are similar versions of Lemma \ref{lem: goodinterval2} and of Theorem \ref{justonegood} for the branching functions of higher dimensional collections, such as squares and tubes. 	
	
	\section{Probabilistic arguments}
	\label{sec3}
	Katz-Tao sets may sometimes have too small of a size, that is not reflective of their dimension. In this section we use random rigid motions to correct this deficit.
	
	\begin{lemma}
		\label{KT1}	
		Let $S$ be an  $\epsilon$-uniform $(\delta,s,K)$-Katz-Tao set of intervals in $\cD_\delta$. Let $\upsilon>0$. Then, for each $\delta<\delta(\epsilon,\upsilon)$ there is $T\subset [0,1]\cap\delta\Z$ with size $N\sim \frac{\delta^{-s}K}{|S|_\delta}$ such that the set
		$$A:=\cup_{t\in T}(t+S)$$
		(a) has size $K\delta^{-s}\gtrsim |A|_\delta\gtrsim \delta^{\upsilon}K\delta^{-s}$
		\\
		\\
		(b)  is a $(\delta,s,O_\epsilon(K\delta^{-\upsilon}))$-Katz-Tao set
		\\
		\\
		(c) each $\delta$-interval is in at most $\delta^{-\upsilon}$ of the translates $t+S$, $t\in T$.
	\end{lemma}
	\begin{proof}
		As a result of Stirling's formula, for each $1\le M\le N$
		$${N\choose M}\lesssim (\frac{Ne}{M})^{M+\frac12}.$$
		We select a random set $T$ of $N\sim \frac{\delta^{-s}K}{|S|_\delta}$ distinct translations. First, let us gauge the size of the resulting set $A$. Fix a $\delta$-interval $I_\delta =[n\delta,(n+1)\delta]$. The probability that a random translate $t+S$ will contain $I_\delta$ is at most $\delta|S|_\delta$. The probability that at least $\delta^{-\upsilon}$ translates will contain $I_{\delta}$ is at most
		$$\sum_{M\ge \delta^{-\upsilon}}{N\choose M}(\delta|S|_\delta)^{M}\lesssim N^{3/2}\max_{M\ge \delta^{-\upsilon}}(\frac{Ne\delta|S|_\delta}{M})^{M}.$$
		Since $K\lesssim \delta^{s-1}$ and  $N\le \delta^{-1}$, this number is $O(\delta^{-\frac32+\upsilon\delta^{-\upsilon}})$, and thus $O(\delta^{100})$, assuming $\delta$ is small enough.
		
		It follows that the probability that at least $\delta^{-\upsilon}$ of the $N$ chosen translates will contain some $I_\delta$ is $O(\delta^{99})$. Equivalently, with probability $1-O(\delta^{99})$, the set $T$ will satisfy
		\begin{equation}
			\label{djrufurgugui}
			\|\sum_{t\in T}1_{t+S}\|_\infty\le \delta^{-\upsilon}.\end{equation}
		Then (a) follows by combining this with  $$\delta N|S|_\delta=\|\sum_{t\in T}1_{t+S}\|_1\le \delta|A|_\delta\|\sum_{t\in T}1_{t+S}\|_\infty.$$
		Note also that (c) follows from \eqref{djrufurgugui}.
		For (b), it suffices to consider $I_\rho$ with length $\rho=2^{-jT}\ge \delta$. The set $S$ may be covered with $|S|_\rho$  $\rho$-intervals, each containing $|S|_\delta/|S|_\rho$ intervals $I_\delta$. We have two possibilities.
		\\
		\\
		1. If $|S|_\rho\rho^s\delta^{-\upsilon}\ge 1$, then any choice of $N$ translates will contribute at most
		$$\le N\frac{|S|_\delta}{|S|_\rho}\lesssim K\delta^{-\upsilon}(\frac{\rho}{\delta})^s$$
		$\delta$-intervals to $I_\rho$. Thus, the non-concentration condition is verified in this case, with probability 1.
		\\
		\\
		2. Assume  $|S|_\rho\rho^s\delta^{-\upsilon}\le 1$.
		The probability that $t+S$ intersects $I_\rho$ is at most $\rho|S|_\rho$.  If it does intersect $I_\rho$, it contributes at most $|S|_\delta/|S|_\rho$ $\delta$-intervals.
		We have
		$${\mathbb{P}}(I_\rho\text{ contains }\ge \delta^{-\upsilon}K(\rho/\delta)^s\text{ intervals from }A)\le$$$$ {\mathbb{P}}(I_\rho\cap(t+S)\not=\emptyset\text{ for }M\ge\delta^{-\upsilon}K(\rho/\delta)^s|S|_\rho/|S|_\delta\text{ translates }t).$$
		Our non-concentration hypothesis together with the main assumption in this case implies that
		$$N\ge\delta^{-\upsilon}K(\rho/\delta)^s|S|_\rho/|S|_\delta\ge \delta^{-\upsilon}.$$
		The same computation shows $M\ge \delta^{-\upsilon}$. The fact that $N$ is at least as large as the smallest admissible $N$ shows that the above probability is nonzero, so it demands an estimate.
		As before, this  probability is
		\begin{align*}
			&\lesssim N^{3/2}\max_{M\ge \delta^{-\upsilon}K(\rho/\delta)^s|S|_\rho/|S|_\delta}(\frac{N}{M}\rho e|S|_\rho)^{M}\\&\lesssim N^{3/2}\max_{M\ge \delta^{-\upsilon}}(\frac{N}{M}\rho e|S|_\rho)^{M} \\&\lesssim \delta^{-3/2}(e\rho^{1-s}\delta^\upsilon)^{\delta^{-\upsilon}}\le \delta
			^{100},
		\end{align*}
		when $\delta$ is small enough.
		
		We only need to test (b) for (at most $O(\delta^{-1})$) intervals $I_\rho$, with $\rho$ of the form $2^{-jT}$. For all other scales, the non-concentration inequality will be forced to hold with $O_\epsilon(1)$  losses.
		
		We conclude that (a) and (b) can be simultaneously achieved, with probability very close to 1.
		
	\end{proof}
	\begin{remark}
		\label{logsareneeded}	
		Since $(\log \delta^{-1})^{\log\delta^{-1}}\gtrsim \delta^{-100}$, the factor $\delta^{-\upsilon}$ may be replaced with $\log\delta^{-1}$ in the previous lemma. This remains true for the applications of this result in this section and the following one. 	
	\end{remark}
	\begin{cor}
		\label{KT2}	
		Suppose $\Lambda\subset \mathbb{S}^1$ is a $(\delta, s, K_1)$-Katz-Tao $\epsilon$-uniform set such that  for each $\delta$-arc $\theta\in \Lambda$, there is a $(\delta, 1-s, K_2)$-Katz-Tao $\epsilon$-uniform set $\cT_{\theta}$ of $\delta$-tubes in the direction $\theta$. Assume all $|\cT_\theta|_{\delta}$ are comparable.  Let $\cT=\cup_{\theta\in\Lambda}\cT_\theta$.
		
		Let $\upsilon>0$. Then there is a collection of $N\sim\frac{\delta^{-1}K_1K_2}{|\cT|_\delta}$ rigid motions $\cA_i$ such that
		$$\cT_{new}=\cup_{i=1}^N\cA_i(\cT)$$
		is a collection of tubes $($in the slightly larger set $[-2,2]^2)$ having the following structure:  there is  a $(\delta,s, O(K_1\delta^{-\upsilon}))$-Katz-Tao set $\Lambda_{new}\subset\mathbb{S}^1$ with cardinality $\delta^\upsilon K_1\delta^{-s}\lesssim |\Lambda_{new}|\lesssim K_1\delta^{-s}$ and for each $\delta$-arc $\theta\in\Lambda_{new}$  there is a $(\delta,1-s, O(K_2\delta^{-\upsilon}))$-Katz-Tao set $\cT_{\theta,new}$ of tubes with cardinality $\delta^\upsilon K_2\delta^{s-1}\lesssim |\cT_{\theta, new}|\lesssim K_2\delta^{s-1}$ such that
		$$\cT_{new}=\bigcup_{\theta\in\Lambda_{new}}\cT_{\theta,new}.$$
		Moreover, each $T\in\cT_{new}$ appears in at most $O(\delta^{-\upsilon})$ many collections $\cA_i(\cT)$. All implicit constants are allowed to depend on $\upsilon,\epsilon$.
	\end{cor}
	\begin{proof}
		This is a two-step procedure. First, we use (a close variant of) Lemma \ref{KT1}	to find a collection $\mathcal{R}$ of $N_1\sim \frac{\delta^{-s}K_1}{|\Lambda|_\delta}$ rotations, such that
		$$\Lambda_{new}=\bigcup_{r\in{\mathcal {R}}}r(\Lambda)$$
		satisfies the desired properties. For $\theta'\in\Lambda_{new}$, $\theta'=r(\theta)$ with $\theta\in\Lambda$ and $r\in \mathcal{R}$, we define $\cT_{\theta'}$ to be $r(\cT_\theta)$.
		
		Second, we use again Lemma \ref{KT1} to find
		$N_2\sim \frac{\delta^{s-1}K_2}{|\cT_{\theta}|_\delta}$ translates $T$ such that for each $\theta'\in\Lambda_{new}$, the set
		$$\cT_{\theta',new}=\cup_{t\in T}(t+\cT_{\theta'})$$
		satisfies the desired properties. The fact that we can realize this with the same $T$ for each $\theta'\in\Lambda_{new}$ is due the high probabilities we have computed in the proof of the lemma.
		
		The rigid motions $\cA_i$ will be all possible combinations of rotations $r\in\mathcal{R}$ followed by translations $t\in T$. Note that $N\sim N_1N_2$.
		
	\end{proof}
	We are now able to formulate a sharper version of Theorem \ref{Fur}, when it comes to the dependence of $\epsilon$ on $\eta$. While this will not be needed in this paper, it may be useful in future applications.
	
	\begin{theorem}
		\label{Fursharper}	
		Let $0<s\le 1$ and $0\le t\le 2$.  Then for each $\upsilon_1,\eta_1>0$ there is $\epsilon_1(\upsilon_1,\eta_1,s,t)$
		with $\lim_{\upsilon_1,\eta_1\to0}\epsilon_1(\upsilon_1,\eta_1,s,t)=0$,
		such that the following holds for each $\eta>0$ and each (small enough) $\delta$.
		
		Consider a pair $(\cP,\cT)$ consisting of $\delta$-squares and $\delta$-tubes in $[0,1]^2$  such that
		\\
		\\
		(a) $\cP$ is a $(\delta,t,\delta^{-\eta})$-set
		\\
		\\
		(b) for each $p\in\cP$ there is a $(\delta,s,\delta^{-\eta_1})$-set $\cT(p)\subset\cT$ of tubes intersecting $p$, with cardinality $\sim M$.	
		
		Then for each $\upsilon_1>0$
		$$|\cT|\gtrsim \delta^{-\min\{t,\frac{s+t}{2},1\}+\eta+\epsilon_1(\upsilon_1,\eta_1,s,t)}M.$$
	\end{theorem}
	\begin{proof}	
		Let first  $\cP_{unif}$ be a $\upsilon_1$-uniform subset of $\cP$ with $|\cP_{unif}|\gtrsim \delta^{-\upsilon_1}|\cP|$.	
		We apply (the two dimensional version of) Lemma \ref{KT1} to $\cP_{unif}$ with $K=|\cP_{unif}|\delta^t\delta^{-\eta-\upsilon_1}$. Using a collection $T$ of $N\sim \delta^{-\eta-\upsilon_1}$ translates, we get a $(\delta,t,O(K\delta^{-\upsilon_1}))$-Katz-Tao set $\cP_{new}=\cup_{t\in T}(t+\cP_{unif})$ with $$K\delta^{-t+\upsilon_1}\lesssim |\cP_{new}|\lesssim K\delta^{-t}.$$
		Note that these turn  $\cP_{new}$ into a $(\delta,t, \delta^{-2\upsilon_1})$-set. Note also that $|\cP_{new}|\lesssim |\cP|\delta^{-\eta-\upsilon_1}$.
		
		For each $p\in\cP_{new}$ lying in some translate $t'+\cP_{unif}$, $p=t'+p'$, define $\cT(p)=t'+\cT(p')$. If there is more than one such pair $(p',t')$, just pick any of them.
		
		We apply Theorem \ref{Fur} to this configuration to get
		$$|\bigcup_{p\in\cP_{new}}\cT(p)|\gtrsim \delta^{-\min\{t,\frac{s+t}{2},1\}+\epsilon(2\upsilon_1,\eta_1,s,t)}M.$$
		Note also that
		\begin{align*}
			|\bigcup_{p\in\cP_{new}}\cT(p)|&\le |\bigcup_{p'\in\cP_{unif}}\bigcup_{t'\in T}(t'+\cT(p'))|\\&=|\bigcup_{t'\in T}(t'+\bigcup_{p'\in\cP_{unif}}\cT(p'))|\\&\le |T||\bigcup_{p'\in\cP_{unif}}\cT(p')|.
		\end{align*}
		Combining these we conclude that
		$$|\cT|\ge |\bigcup_{p'\in\cP_{unif}}\cT(p')|\gtrsim \delta^{-\min\{t,\frac{s+t}{2},1\}+\eta+\upsilon_1+\epsilon(2\upsilon_1,\eta_1,s,t)}M.$$
		Our result is thus verified with  $\epsilon_1(\upsilon_1,\eta_1,s,t)=\epsilon(2\upsilon_1,\eta_1,s,t)+\upsilon_1$.
	\end{proof}
	\section{Consequences of Theorem \ref{thm: main}}
	\label{sec4}
	
	\begin{defn}\label{KK}
		Given $K_1,K_2\ge 1$, we call a  collection $\cT$ of $\delta$-tubes a $(K_1,K_2)$-set if its direction set $\Lambda$ is a $(\delta,s,K_1)$-Katz-Tao set and each $\cT_\theta$ is a $(\delta,1-s,K_2)$-Katz-Tao set.
	\end{defn}
	
	Let us write $\cS\cT(\delta,K_1,K_2)$ for the smallest constant that satisfies the inequality
	\begin{equation}
		\label{defst}
		|\cP_r(\cT)|\le \cS\cT(\delta,K_1,K_2)  \frac{\delta^{-2}}{r^3}
	\end{equation}
	for all $r\ge 1$ and each $(K_1,K_2)$-set $\cT$. Theorem \ref{thm: main} will follow if we prove the superficially more general inequality $\cS\cT(\delta,K_1,K_2)\lessapprox 1$, for each $K_1,K_2=O(1)$.
	\medskip
	
	In this section we relate $\cS\cT(\delta,K_1,K_2)$ to $\cS\cT(\delta,O(1),O(1))$, for arbitrary $K_1,K_2$. In particular, we show how Theorem \ref{thm: mainbettt} follows from Theorem \ref{thm: main}. The inequalities we derive will be used at larger scales, in the final argument in Section \ref{sec7}. In this section we will not impose the restrictive assumption $s\le 1/2$, as we only compare, rather than explicitly estimate various upper bounds.
	\medskip
	
	We will  repeatedly use the following lemma.
	\begin{lemma}
		\label{diffcoleff}	
		Consider pairwise disjoint collections of tubes $\cT_1,\ldots,\cT_N$ and write $\cT=\cT_1\cup\ldots\cup \cT_N$. Then for each $1\le r\lesssim \delta^{-1}$	
		$$
		|\cP_r(\cT)|\lesssim  \log (1/\delta)^2 \max_{1\le M\lesssim \log(1/\delta)N}M^{-1} \sum_{i=1}^N |\cP_{r/M}(\cT_i)|.$$
	\end{lemma}	
	\begin{proof}
		By dyadic  pigeonholing, there exists a  dyadic number $M\lesssim \log(1/\delta)N$ and a  subset $\cP\subset \cP_r(\cT)$ such that $|\cP|\gtrsim (\log(1/\delta))^{-1} |\cP_r(\cT)|$ and for each $p\in \cP$ there exist $\gtrsim M/\log(1/\delta)$ sets $\cT_i$ such that $p\in \cP_{r/M}(\cT_i)$. Let us explain why. First, for each $p\in \cP_r(\cT)$ we have
		$$r\sim \sum_{M\text{ dyadic}\atop_{M\le r}}\frac{r}{M}|\{i\le N:\;p\in\cP_{r/M}(\cT_i)\}|.$$
		This allows us to select some dyadic $M_p$ such that  $|\{i:\;p\in\cP_{r/M_p}(\cT_i)\}|\gtrsim M_p/\log(1/\delta)$.	 Since there are $O(\log(1/\delta))$ possible values for $M_p$, one of them, let us call it $M$, will be shared by at least a $1/\log(1/\delta)$-fraction of $\cP_r(\cT)$, that we call $\cP$.	
		
		Apply double counting to conclude that
		$$
		|\cP_r(\cT)|\lesssim  |\cP|\log(1/\delta)\lesssim \log (1/\delta)^2 M^{-1} \sum_{i=1}^N |\cP_{r/M}(\cT_i)|.$$
	\end{proof}
	Here is the main result of this section. When $K_1,K_2=O_\epsilon(1)$, the following inequality improves the  trivial upper bound $\frac{\delta^{-2}}{r^3}$ (a mere consequence of the definition of $\cS\cT(\delta,O_\epsilon(1),O_\epsilon(1))$),  replacing it with the stronger upper bound $\delta^{-1}\frac{|\cT|}{r^3}$.
	\begin{prop}\label{cor: mainKT} Assume $\cT$ is a $(K_1,K_2)$-set.  Then for each $r\ge 1$ and $\epsilon>0$
		\begin{equation}
			\label{kjijfiriurtgiurtigut}	
			|\cP_r(\cT)|\lesssim_\epsilon \cS\cT(\delta,O_\epsilon(1),O_\epsilon(1))(K_1K_2)^2 \delta^{-1+o_\epsilon(1)} \frac{|\cT|}{r^3}.
		\end{equation}
	\end{prop}
	\begin{proof}
		We start by pointing out the following trade-off. As $\epsilon\to 0$, the exponent $o_\epsilon(1)$ will go to zero, while the constants $O_\epsilon(1)$ will go to infinity. 
		\\
		\\
		Step 1. We first prove the superficially weaker inequality (as our hypothesis forces $|\cT|\le K_1K_2\delta^{-1}$)
		\begin{equation}
			\label{kjijfiriurtgiurtigut1}	
			|\cP_r(\cT)|\lesssim_\epsilon \cS\cT(\delta,O_\epsilon(1),O_\epsilon(1))\delta^{o_\epsilon(1)} (K_1K_2)^3\frac{\delta^{-2}}{r^3}.
		\end{equation}
		To see this, we use Lemma \ref{fromgentoexact} to partition $\cT$ into $N=O_\epsilon(\delta^{o_\epsilon(1)}K_1K_2)$ many $(O_\epsilon(1),O_\epsilon(1))$-sets $\cT_i$. Then we apply Lemma \ref{diffcoleff}, noting that (by the definition of $\cS\cT(\delta,O_\epsilon(1),O_\epsilon(1))$)
		$$|\cP_{r/M}(\cT_i)|\le \cS\cT(\delta,O_\epsilon(1),O_\epsilon(1))\frac{\delta^{-2}}{(r/M)^3}.$$
		It follows that 
		\begin{align*}
			|\cP_r(\cT)|&\lesssim  \log (1/\delta)^2 \max_{1\le M\lesssim \log(1/\delta)N}M^{-1} \sum_{i=1}^N |\cP_{r/M}(\cT_i)|\\&\lesssim \log (1/\delta)^4\cS\cT(\delta,O_\epsilon(1),O_\epsilon(1))N^3\frac{\delta^{-2}}{r^3}\\&\lesssim_\epsilon \cS\cT(\delta,O_\epsilon(1),O_\epsilon(1))\delta^{o_\epsilon(1)} (K_1K_2)^3\frac{\delta^{-2}}{r^3}.
		\end{align*}
		\\
		\\
		Step 2. We next prove \eqref{kjijfiriurtgiurtigut}	 in the case $K_1,K_2=O_\epsilon(1)$. 
		Let us partition $\cT$ into sets $\cT^l$, each of which satisfies the requirements in Corollary \ref{KT2}. We achieve this as follows.
		
		Apply Lemma \ref{useful} to write each $\cT_\theta$ as a union of $O_\epsilon(\delta^{o_\epsilon(1)})$ many $\epsilon$-uniform   sets $\cT_{\theta,i}$. By allowing some of these sets to be empty, we split $\cT$ into $O(\delta^{o_\epsilon(1)})$
		many sets $\cT_j$, such that each $\cT_j$ contains at most one $\cT_{\theta,i}$ for each $\theta$ (by that we mean, it either contains all tubes from some $\cT_{\theta,i}$ and no tubes from any other $\cT_{\theta,i'}$, or it contains no tubes in the direction $\theta$).
		Moreover, we require that $|\cT_{\theta,i}|$ have similar size for all $\theta$ contributing to $\cT_j$. This is achieved by selecting sets $\cT_{\theta,i}$ with size inside a fixed dyadic level set.
		
		For each $\cT_j$ we call $\Lambda_j\subset \Lambda$ the corresponding set of contributing directions. We apply  Lemma \ref{useful} again to partition each $\Lambda_j$ into $O(\delta^{o_\epsilon(1)})$ many $\epsilon$-uniform    sets $\Lambda_{j,k}$. Overall, we have a partition of $\cT$ into $O(\delta^{o_\epsilon(1)})$ many  sets, and we call them $\cT^l$. The size of these sets will vary, but this will do no harm to the argument. 
		
		Due to Lemma \ref{diffcoleff}, it suffices to prove that for each $l$ and $r\ge 1$
		\[
		|\cP_r(\cT^l)|\lesssim_\epsilon \cS\cT(\delta,O_\epsilon(1),O_\epsilon(1))\delta^{-1+o_\epsilon(1)} \frac{|\cT^l|}{r^3}.
		\]
		We apply Corollary \ref{KT2} to each $\cT^l$. This result provides $N\sim _\epsilon\frac{\delta^{-1} }{|\cT^l|}$ rigid motions $\{\cA_i\}_{i=1}^N$ and a $(O_\epsilon(\delta^{o_\epsilon(1)}),O_\epsilon(\delta^{o_\epsilon(1)}))$-set $\cT^l_{new}$.

		Note  that the sets $\cP_i=\cP_{r}(\cA_i(\cT^l))$ have the same cardinality as $\cP_{r}(\cT^l)$, as they are rigid motions of the latter. We use the fact that
		$$N|\cP_{r}(\cT^l)|=\sum_{i=1}^{N}|\cP_i|=\sum_{1\le M\le N}M\#\{q:\;q \text{ is in }M\text{ sets }\cP_i\}.$$
		Since Corollary \ref{KT2} guarantees that each tube in $\cT^l_{new}$ appears in at most $O(\delta^{o_\epsilon(1)})$ of the sets $\cA_i(\cT^l)$, it follows that each $q$ that is in $M$ of the sets $\cP_i$ is in fact in $\cP_{L}(\cT^l_{new})$, for some $L\sim Mr\delta^{o_\epsilon(1)}$.
		Using \eqref{kjijfiriurtgiurtigut1} we find that
		$$|\cP_{L}(\cT^l_{new})|\lesssim_\epsilon \cS\cT(\delta,O_\epsilon(1),O_\epsilon(1))\delta^{o_\epsilon(1)} \frac{\delta^{-2}}{L^3},$$
		and thus
		$$\#\{q:\;q \text{ is in }M\text{ sets }\cP_i\}\lesssim_\epsilon \cS\cT(\delta,O_\epsilon(1),O_\epsilon(1))\delta^{o_\epsilon(1)} \frac{\delta^{-2}}{(Mr)^3}.$$
		Putting things together, it follows that
		$$|\cP_r(\cT^l)|\lesssim_\epsilon \cS\cT(\delta,O_\epsilon(1),O_\epsilon(1)) \delta^{-1+o_\epsilon(1)}\frac{|\cT^l|}{r^3}.$$
		\\
		\\
		Step 3. We are left with verifying \eqref{kjijfiriurtgiurtigut}for arbitrary $K_1,K_2\ge 1$.	The argument is similar to the one in Step 1, with one key difference. 
		
		First, we use Lemma \ref{fromgentoexact} to partition $\cT$ into $N=O_\epsilon(\delta^{o_\epsilon(1)}K_1K_2)$ many $(O_\epsilon(1),O_\epsilon(1))$-sets $\cT_k$. To find an upper bound for $|\cP_{r/M}(\cT_k)|$ we do not use the definition of $\cS\cT(\delta,O_\epsilon(1),O_\epsilon(1))$ as in Step 1, but rather the improved bound from Step 2
		$$|\cP_{r/M}(\cT_k)|\lesssim_\epsilon \cS\cT(\delta,O_\epsilon(1),O_\epsilon(1)) \delta^{-1+o_\epsilon(1)} \frac{|\cT_k|}{(r/M)^3}.$$
		Lemma \ref{diffcoleff} finishes the argument as follows
		\begin{align*}
			|\cP_r(\cT)|&\lesssim  \log (1/\delta)^2 \max_{1\le M\lesssim \log(1/\delta)N}M^{-1} \sum_{k=1}^N |\cP_{r/M}(\cT_k)|\\&\lesssim_\epsilon \cS\cT(\delta,O_\epsilon(1),O_\epsilon(1)) \delta^{-1+o_\epsilon(1)} \frac{\sum_k|\cT_k|}{r^3}N^2\\&\lesssim_\epsilon \cS\cT(\delta,O_\epsilon(1),O_\epsilon(1))(K_1K_2)^2 \delta^{-1+o_\epsilon(1)} \frac{|\cT|}{r^3}.
		\end{align*}
	\end{proof}
	We record the following immediate corollary. Note that the quantity $\delta|\cT|^3/r^3$ is larger than $|\cT|^2/r^3$, since the hypothesis forces $|\cT|\gtrsim \delta^{-1}$.	
	\begin{cor}\label{cor: main} Let $s\in (0, 1)$.
		Suppose $\Lambda\subset \mathbb{S}^1$ is a $(\delta, s, C_1)$-set such that for each $\delta$-arc $\theta\in \Lambda$, there is a  $(\delta, 1-s,C_2)$-set $\cT_{\theta}$ of $\delta$-tubes in the direction $\theta$. We assume $|\cT_\theta|$ are comparable. Let $\cT=\cup_{\theta\in \Lambda} \cT_{\theta}$.
		Then for each $r\ge 1$ and $\epsilon>0$
		\[
		|\cP_r(\cT)|\lesssim_\epsilon  \cS\cT(\delta,O_\epsilon(1),O_\epsilon(1))(C_1C_2)^2 \delta\frac{|\cT|^3}{r^3}.
		\]
	\end{cor}
	\begin{proof}
		It suffices to note that $\cT$ is a $(K_1,K_2)$-set  with $K_1=C_1|\Lambda|\delta^{s}$ and $K_2=C_2|\cT_\theta|\delta^{1-s}$, and that $K_1K_2\sim \delta C_1C_2|\cT|$.
		
	\end{proof}

	We close this section with one last result of the same type.
	
	\begin{prop}
		\label{actthis}	
		Suppose $\Lambda\subset \mathbb{S}^1$ is an $\eta$-uniform $(\delta, s,1)$-set with $|\Lambda|\sim \delta^{-s}$ and suppose that for each $\delta$-arc $\theta\in \Lambda$ there is a $(\delta, 1-s,1)$-set $\cT_{\theta}$ of $\delta$-tubes in direction $\theta$ with $|\cT_{\theta}|\sim \delta^{-1+s}$.  Let $\cT=\cup_{\theta\in \Lambda} \cT_{\theta}$, so $|\cT|\sim \delta^{-1}$. Assume $\cT$ is $\eta$-uniform.
		
		Let $\rho_0\ge \delta$, and consider the collection $\cT^{\rho_0}$ of $\rho_0$-thickenings of the tubes in  $\cT$.
		
		Then for each $\br\ge 1$ and $\epsilon>0$
		\[
		|\cP_\br(\cT^{\rho_0})|\lesssim_{\epsilon,\eta} \cS\cT(\rho_0,O_\epsilon(1),O_\epsilon(1))  \rho_0^{1+o_\epsilon(1)}\frac{|\cT^{\rho_0}|^3}{\br^3}.
		\]
	\end{prop}
	\begin{proof}
		It suffices to  prove that the collection $\cT^{\rho_0}$ satisfies the assumptions of Corollary  \ref{cor: main} with $\delta$ replaced by $\rho_0$ and $C_1,C_2=O_\eta(1)$.
		
		First, by Lemma \ref{tradim} the set $\Lambda_{\rho_0}$ is a $(\rho_0,s, O_\eta(1))$-set.

		Second, we claim that for any $\rho_0$-arc $\tau\in\Lambda_{\rho_0}$, $\cT^{\rho_0}_{\tau}:=\{T_{\rho_0}\in \cT^{\rho_0}: \text{dir}(T_{\rho_0}) =\tau\}$ is a $(\rho_0, 1-s,O_\eta(1))$-set.
		Indeed,  since $\cT$ is $\eta$ uniform, the number $N=\{T\in\cT:\;T\subset T_{\rho_0}\}$ is  roughly the same (modulo a multiplicative factor of $O_\eta(1)$) for each $T_{\rho_0}\in \cT^{\rho_0}$. There are
		$\sim \#\{\theta\subset \tau\}\delta^{s-1}$
		tubes $T\in\cT$ with direction  $\theta\subset\tau$.
		Each of these $T$ lies inside some $T_{\rho_0}\in \cT_\tau^{\rho_0}$.
		Thus, there are $\sim \frac1N \#\{\theta\subset \tau\}\delta^{s-1}$ fat tubes in $\cT_\tau^{\rho_0}$.

		Similarly, for each $\delta\le \Delta\le 1$ there are $\lesssim \frac1N \#\{\theta\subset \tau\}(\delta/\Delta)^{s-1}$ fat tubes $T_{\rho_0}\in \cT_\tau^{\rho_0}$ inside each $\Delta\times 1$-box in the direction $\tau$. Combining these it follows that  $\cT_{\tau}^{\rho_0}$ is a $(\rho_0, 1-s,O_\eta(1))$-set.
		
		Third, since $\Lambda$ is uniform,  the number $\frac1N \#\{\theta\subset \tau\}\delta^{s-1}$ is roughly the same for each $\tau\in\Lambda_{\rho_0}$.
	\end{proof}

	\section{Incidence results via Furstenberg set estimates}
	\label{sec5}
	In many of the forthcoming arguments we will consider collections  of tubes $\cT$, squares $\cP$, and subsets $\cT_p\subset\cT$ with the property that $p\cap T\not=\emptyset$ for each $T\in\cT_p$. Typically, $\cT_p$ may not contain all $T\in\cT$ intersecting $p$. The incidences between $\cP$ and $\cT$ (relative to the family of sets $\cT_p$) are defined as
	$$\cI(\cP,\cT)=\{(p,T)\in\cP\times\cT:\;T\in\cT_p\}.$$
	We will often use double counting
	$$|\cI(\cP,\cT)|=\sum_{p\in\cP}|\cT_p|=\sum_{T\in\cT}|\{p:\;T\in\cT_p\}|.$$

	The key concept used in this section is that of {\em two-ends}. This is a rather weak form of non-concentration. Here is how it plays into our argument. The main result in this section, Theorem \ref{thm: delta-s}, will use the  estimate \eqref{irufurfur9fu9}. The minimum there contains three terms. Working with two-ends tubes will narrow it down to two terms. Thus, Theorem \ref{thm: delta-s} will deliver a dichotomy, rather than a trichotomy.	
	\begin{defn}[Two-ends]
		\label{def: two-ends}
		Let $T\in\cD_{\delta}^{dual}$ and let $\cP_T\subset \cD_\delta$ be a nonempty collection of $\delta$-squares intersecting $T$.
		We say that $T$ is $\epsilon$-two-ends (with respect to $\cP_T$) if for each $\rho\in (\delta, 1)$ and each ball $B_{\rho}$,
		\[
		| \cP_T\cap B_{\rho}|_{\delta} \lesssim \rho^{\epsilon}\delta^{-5\epsilon^3} | \cP_T|_{\delta}.
		\]
		
	\end{defn}

	Note that being $\epsilon$-two-ends implies that no $\delta^{\epsilon}$-segment of $T$ contains the whole $\cP_T$. It also implies that $|\cP_T|\gtrsim \delta^{-\epsilon}$. The choice of the exponent $5\epsilon^3$ is finely tuned for the application in Section \ref{sec:new}.

	The two-ends condition will be used in our proof that  Theorem \ref{thm: delta-s2}$\implies$Theorem \ref{thm: delta-s}, via the following simple observation. 	
	\begin{lemma}
		\label{covnumb}	
		Let $\cP$ be a collection of $\delta$-squares.
		Assume $\cT_p$ is a $(\delta,s)$-set of tubes intersecting a fixed $\delta$-square $p$.
		Let $\cP_T=T\cap \cP$, the squares intersecting $T$. Assume there is $\cT_p'\subset \cT_p$ with  $|\cT_p'|\gtrsim \delta^{O(\epsilon)}|\cT_p|$ such that each $T\in \cT_p'$ is  $\epsilon$-two-ends.

		Then for each $\rho\gtrsim  \delta$ we have $$|\cP|_\rho\gtrsim \rho^{-s}\delta^{O(\epsilon)}.$$
	\end{lemma}
	\begin{proof}Note that $|\cT_p'|\gtrsim \delta^{-s+O(\epsilon)}$.
		Consider a collection of $\gtrsim \rho^{-s}\delta^{O(\epsilon)}$ directions that are $\delta^{-\epsilon}\rho$-separated, such that there is  $T\in\cT'_p$  pointing in each of these directions. For each such tube, consider its segment that is $\gtrsim \delta^\epsilon$-away from $p$.  Our hypothesis forces each segment to contain at least one square in $\cP$. The $\rho$-neighborhoods of these segments are pairwise disjoint.
	\end{proof}

	Here is our main result in this section.
	\begin{theorem}\label{thm: delta-s}	
		Let $\epsilon,\eta>0$.	
		Let $\cP_{\delta, s}$ be a set of  $\delta$-squares $p$. For each $p\in\cP_{\delta,s}$, let  $\cT_p$ be a $(\delta, s,\delta^{o_\epsilon(1)})$-set of tubes intersecting $p$, with  $|\cT_p|\sim r$.
		
		Let $\bar{\cT}_p\subset \cT_p$ and write $\cP_{T}=\{p:\;T\in\bar{\cT}_p\}$, $\bar{\cT}=\cup_{p}\bar{\cT}_p$. Assume
		\begin{equation}
			\label{ ifj iojiojio}
			T\in\bar{\cT} \implies T\text{ is } \epsilon\text{-two-ends with respect to }\cP_T,\end{equation}
		\begin{equation}
			\label{o.k}
			\sum_{p}|\bar{\cT}_p|\gtrsim \delta^{\eta}\sum_p|\cT_p|.
		\end{equation}
		Then at least one of the following  happens (these will be referred later as Item (1) and Item (2)):
		\\
		\\
		\textbf{Item 1.}
		\[
		|\cP_{\delta, s}|\lesssim   \delta^{s+o_\epsilon(1)+o_\eta(1)} \frac{|\bar{\cT}|^2}{r^2}.
		\]
		\textbf{Item 2.} There exists a scale $\Delta\geq \delta^{1-\sqrt{\epsilon}}$ and a family of $\Delta$-squares $B_{\Delta}\subset (\cP_{\delta, s})_\Delta$ such that
		these squares cover at least a $\delta^{O(\epsilon+\eta)}$-fraction  of $\cP_{\delta,s}$ and
		$$|\cP_{\delta,s} \cap B_{\Delta}|\gtrsim (\frac{\Delta}{\delta})^{2-s+\epsilon^{1/4}}.$$
	\end{theorem}
	Let us say a few words about the relevance of this result in the grand scheme of proving  Theorem \ref{thm: main}. Note that our hypothesis forces $r\lessapprox \delta^{-s}$. So if Item (1) holds then we get the upper bound $|\cT|^2/r^3$ for $|\cP_{\delta,s}|$. This matches the one in Theorem \ref{thm: main}. However, for the typical square  $p\in \cP_r(\cT)$ in Theorem \ref{thm: main}, the collection $\cT_p$ of tubes intersecting $p$ will not be a $(\delta,s, C)$-set with small $C$, unless $r\approx \delta^{-s}$. We will have to carefully manipulate the collection $\cT$ to make Theorem \ref{thm: delta-s} applicable.
	
	The application of Theorem \ref{thm: delta-s} will generate a dichotomy in our final argument based on induction on scales. If Item (1) holds, it will lead to a more direct closing of the induction. If Item (2) holds, we will use an elegant application of the {\em high-low method}.
	\smallskip

	We deduce Theorem \ref{thm: delta-s} from the following related result.

	\begin{theorem}\label{thm: delta-s2}
		Let $\cT$ be a set of $\delta$-tubes in the plane. Let $\cP$ denote an $\epsilon$-uniform set of $\delta$-squares $p$ such that for each $\rho\gtrsim \delta$ we have
		\begin{equation}
			\label{covnumblowb}
			|\cP|_\rho\gtrsim \rho^{-s}\delta^{o_\epsilon(1)}.
		\end{equation}
		Let   $\cT_p\subset \cT$ be a $(\delta, s, \delta^{o_\epsilon(1)})$-set of tubes passing through $p$ with $|\cT_p|\sim r\delta^{o_\epsilon(1)}$.
		We assume each $\cT_p$ is $\epsilon$-uniform, with the same branching function (i.e. for each $\rho=2^{-jT}$ and each  $\bT\in \cT_p^{\rho}$, $|\cT_p[\bT]|$ has a fixed value $N_\rho$, independent of $p\in\cP$).

		Then at least one of the following happens:
		\begin{enumerate}
			\item \[
			|\cP|\lesssim   \delta^{s+o_\epsilon(1)} \frac{|\cT|^2}{r^2}.
			\]
			\item there exists $\Delta\gtrsim \delta^{1-\sqrt{\epsilon}}$   such that for each $B_{\Delta}\in \cP_\Delta$
			$$|\cP \cap B_{\Delta}|\gtrsim (\frac{\Delta}{\delta})^{2-s+\epsilon^{1/4}}.$$
		\end{enumerate}

	\end{theorem}
	For later use in the final induction on scales argument, we quantify explicitly the lower bounds for $\Delta$ and $|\cP \cap B_{\Delta}|$. The other $o_\epsilon(1)$ terms are irrelevant.

	\begin{proof}( Theorem \ref{thm: delta-s2}$\implies$Theorem \ref{thm: delta-s} )

		There are two steps.
		\\
		\\
		1. Uniformization: We reduce the theorem to proving it for $\epsilon$-uniform collections of tubes and squares.
		\\
		\\
		Write $P=|\cP_{\delta,s}|$. We call a square $p$ incident to a tube $T\in\bar{\cT}$ if $T\in\bar{\cT}_p$, or equivalently, if $p\in\cP_T$.
		Note the lower bound for incidences
		$$|\cI(\cP_{\delta,s},\bar{\cT})|\gtrsim  \delta^\eta Pr.$$
		In the following we refine first $\bar{\cT}$, then $\cP_{\delta,s}$, by essentially preserving the incidences. It will be important that the refinement of $\cP_{\delta,s}$ remains comparable in size.
		
		Select a  number $t$ and a subset $\cT'\subset\bar{\cT}$ such that each $T\in\cT'$ is incident to $\sim t$ squares in  $\cP_{\delta,s}$ (thus $|\cP_T|\sim t$) and moreover, the number of incidences is essentially preserved
		$$|\cT'|t\sim \cI(\cP_{\delta,s},\cT')\gtrsim \delta^{\eta}(\log(1/\delta))^{-1}Pr.$$
		Note that the size of $\cT'$ might be much smaller than the size of $\bar{\cT}$, but this will only help us with our final estimate.

		Each $p$ was initially incident to $\sim r$ tubes. But since we trimmed the collection of tubes, this is no longer the case. However, the following satisfactory substitute holds.
		
		Note that the  subset $\cP'\subset \cP_{\delta,s}$ of those $p$ that are incident to $\gtrsim \delta^{\eta}(\log(1/\delta))^{-1} r$ tubes in $\cT'$ satisfies $|\cP'|\gtrsim P\delta^{\eta}(\log(1/\delta))^{-1}$. This happens since $$\delta^{\eta}(\log(1/\delta))^{-1}Pr\lesssim   \cI(\cP_{\delta,s},\cT')=\cI(\cP',\cT')+\cI(\cP_{\delta,s}\setminus \cP',\cT'),$$
		$$\cI(\cP_{\delta,s}\setminus \cP',\cT')\ll\delta^{\eta}(\log(1/\delta))^{-1} Pr,$$
		and thus
		$$|\cP'|r\gtrsim  \cI(\cP',\cT')\gtrsim Pr\delta^{\eta}(\log(1/\delta))^{-1}.$$
		For each $p\in\cP'$ we write $\cT_{p}'$ for the collection of tubes in $\cT'$ that $p$ is incident to. Then $|\cT_p'|\gtrsim \delta^{\eta}(\log(1/\delta))^{-1} r$.
		
		Find a subset $\cP'' \subset \cP'$ such that $|\cP''|\gtrsim P\delta^{O(\epsilon^3)+\eta}$ and for each $p\in \cP''$, there is  (cf. Lemma \ref{largeunifset}) an $\epsilon^3$-uniform set $\cT_{p,unif}\subset \cT_p'$ with $|\cT_{p,unif}|\gtrsim \delta^{\epsilon^3} |\cT_p'|$ and  with  branching function independent of $p$.
		Here is why this is possible. Each $\cT_{p,unif}$ is associated with a string $(N_1,\ldots,N_m)$, with $N_j:=|\cT_{p,unif}|_{2^{-jT}}$ satisfying  $N_{j+1}/N_j=2^{c_j}$ for some nonnegative integer $c_j\lesssim T$.  There are $(O(T))^m\lesssim \delta^{-\epsilon^3}$ such strings. Pigeonhole to pick one that represents many $p$. This gives a large subset $\cP''$.
		
		Finally, find an $\epsilon^3$-uniform subset $\cP_{unif}\subset \cP''$ with size $\gtrsim P\delta^{O(\epsilon^3)+\eta}$. Note that
		$$\cI(\cP_{unif},\cT')\gtrsim \delta^{O(\epsilon^3)+\eta}|\cP_{unif}|r\gtrsim \delta^{O(\epsilon^3)+\eta}Pr.$$

		Note that for each $p\in\cP_{unif} $ the collection $\cT_{p,unif}$ is a $(\delta,s, \delta^{o_\epsilon(1)+o_\eta(1)})$-set with cardinality $\sim r\delta^{o_\epsilon(1)+o_\eta(1)}$.
		\\
		\\
		2. Covering number estimates. We prove that for each $\rho\gtrsim \delta$ we have
		\begin{equation}
			\label{kldcrt0y6o00uo706}
			|\cP_{unif}|_\rho\gtrsim \rho^{-s}\delta^{o_\epsilon(1)+\eta}.
		\end{equation}
		\\
		\\
		Recall that each $T\in \cT'$ is incident to  $\sim t$
		squares in $\cP_{\delta,r}$, so $$Pr\gtrsim \cI(\cP_{\delta,r},\cT')\sim |\cT'|t.$$
		Recall also that
		$$\cI(\cP_{unif},\cT')\gtrsim \delta^{O(\epsilon^3)+\eta}Pr.$$
		Thus,
		$$ \cI(\cP_{unif},\cT')\gtrsim |\cT'|t\delta^{O(\epsilon^3)+\eta}.$$
		It follows that, for $C_1$ large enough,  the number of tubes in the collection  $\cT''\subset\cT'$
		that are incident to  $\gtrsim t\delta^{C_1\epsilon^3}$ squares in $\cP_{unif}$ is $\gtrsim|\cT'|\delta^{O(\epsilon^3)+\eta}$. For later use we record that
		\begin{equation}
			\label{juryhfurfuruyu}
			\cI(\cP_{unif},\cT'')\gtrsim |\cT'|t\delta^{O(\epsilon^3)+\eta}Pr\gtrsim  \delta^{O(\epsilon^3)+2\eta}Pr.\end{equation}
		
		Note that for each $T\in \cT''$, the squares in  $\cP_{unif}$ that are incident to $T$ cannot all be contained in a $\delta^\epsilon$-segment of $T$. Indeed, $T$ being two-ends with respect to $\cP_T$ guarantees that such a segment only contains $\lesssim \delta^{\epsilon^2-5\epsilon^3}t$  squares even from $\cP_T$. And if $\epsilon$ is small enough we have $\epsilon^2-5\epsilon^3>C_1\epsilon^3$.
		
		Using \eqref{juryhfurfuruyu} we get
		$$\sum_{p\in\cP_{unif}}|\{T\in \cT'':\;p\cap T\not=\emptyset\}|\ge \cI(\cP_{unif},\cT'')\gtrsim \delta^{O(\epsilon^3)+2\eta}Pr.$$
		It follows that there is $p$ such that
		$$|\{T\in \cT'':\;p\cap T\not=\emptyset\}|\gtrsim \delta^{O(\epsilon^3)+2\eta}r.$$
		The desired \eqref{kldcrt0y6o00uo706} now follows from Lemma \ref{covnumb}.
		\\
		\\
		3. Conclusion:  We write $\cP=\cP_{unif}$.  Remark \ref{approxunifsets} implies that $\cP$ is approximate $\epsilon$-uniform, satisfies \eqref{kldcrt0y6o00uo706} and
		\begin{equation}
			\label{cooomparison}
			|\cP_{\delta,s}|\lesssim \delta^{-o_\epsilon(1)-\eta}|\cP|.
		\end{equation}
		Each $\cT_{unif, p}$ is approximate  $\epsilon$-uniform and a $(\delta,s,\delta^{o_\epsilon(1)+o_\eta(1)})$-set. The value of $r$ has decreased by at most a factor of $\delta^{o_\epsilon(1)+\eta}$.
		We now apply Theorem \ref{thm: delta-s2} to the collections $\cP$ and $\cT'$. The $\delta^{\eta}$-losses in constants will only incur $\delta^{O(\eta)}$-losses in Item (1).
		
		If Item (1) in that theorem holds, then when combined with \eqref{cooomparison} it immediately implies (1) in
		Theorem \ref{thm: delta-s}.
		
		Also, since $\cP\subset \cP_{\delta,p}$, if   Item (2) in Theorem \ref{thm: delta-s2} holds, then Item  (2) in
		Theorem \ref{thm: delta-s} must also hold.

	\end{proof}

	\begin{proof}(of Theorem \ref{thm: delta-s2} )
		Recall the function $\psi$ from Theorem \ref{justonegood}.
		We let $\upsilon=\upsilon(\epsilon)\ge \epsilon$ be such that $\psi(\upsilon)\ge \sqrt{\epsilon}$ and $\lim_{\epsilon\to0}\upsilon(\epsilon)=0.$ Each $\epsilon$-uniform set will also be (approximate) $\upsilon$-uniform, see Remark \ref{approxunifsets}. It will be useful to keep in mind that being $o_\upsilon(1)$ is equivalent with being $o_\epsilon(1)$.

		Apply (the two-dimensional version of) Theorem \ref{justonegood} to $\cP$ and $\upsilon$ (in place of $\epsilon$). Then, using \eqref{covnumblowb},  rearranging Theorem \ref{justonegood} (ii) and since $A_1=0$ we get
		$$A_2(t_1-t_2)+A_3(t_2-t_3)+\ldots+A_{l}(t_{l-1}-t_{l})+t_{l}A_{l+1}\ge sA_{l+1}+o_\upsilon(1),\;1\le l\le L.$$
		Since $t_l<t_{l+1}$ this gives
		$$t_{l}A_{l+1}\ge sA_{l+1}+c_\upsilon,$$
		with $c_\upsilon=o_\upsilon(1)$.
		
		A closer inspection of the proof of Theorem \ref{justonegood} reveals that $c_\upsilon$ is negative, and this requires some attention.
		We pick $l_0$ such that on the one hand $A_{l_0}=o_\upsilon(1)$, and on the other hand $c_\upsilon/A_{l+1}=o_\upsilon(1)$ for each $l\ge l_0$. All intervals $[A_l,A_{l+1}]$ with $l<l_0$ will be moved into the ``bad" category. Thus,  the combined total length of all the bad intervals is still $o_\upsilon(1)$. All intervals $[A_l,A_{l+1}]$ with $l\ge l_0$ will be called ``good".

		Our choice forces that for each $l\ge l_{0}$,
		\begin{equation}
			\label{ghdfgsfhjsgf}
			t_{l}\ge s+o_\upsilon(1).
		\end{equation}
		For each good interval we will get essentially sharp estimates. For the bad ones trivial estimates will suffice, since their lengths add up to a negligible amount $o_\upsilon(1)$.
		
		\vspace{5pt}
		
		Case 1. When $t_L\ge 2-s+\epsilon^{1/4}+3\epsilon$, this gives Item $(2)$ with $\Delta=\delta^{A_L}$. Indeed, $$\Delta/\delta=(1/\delta)^{A_{L+1}-A_{L}}\ge (1/\delta)^{\psi(\upsilon)}\ge \delta^{-\sqrt{\epsilon}},$$
		and, by Theorem \ref{justonegood} (iii) with $I$ replaced by $B_\Delta$,
		$$|B_\Delta\cap\cP|\gtrsim (\frac1\delta)^{(A_{L+1}-A_L)(t_L-3\epsilon)}=(\Delta/\delta)^{t_L-3\epsilon}\ge (\Delta/\delta)^{2-s+\epsilon^{1/4}}.$$
		If we are in Case 1, this is the end of the whole  argument.
		
		\vspace{5pt}
		Case 2. When
		\begin{equation}
			\label{lalaland}
			t_L\leq 2-s+\epsilon^{1/4}+3\epsilon,
		\end{equation}
		the only implication we use is that
		\begin{equation}
			\label{doi}
			t_l\leq 2-s+o_\upsilon(1)
		\end{equation}
		for each $1\le l\le L$.
		
		We write $\bp_l$ for a generic square in $\cP_{\delta^{A_l}}$. In particular, $\bp_{L+1}$ is an alternative notation for an arbitrary square  $p\in \cP$.
		\smallskip
		
		The following argument is applicable to each $l$.
		Let us explain the main step for  $[A_L,A_{L+1}]$.
		Consider $\bp_L\in \cP_{\delta^{A_L}}$. For $p\in \cP\cap\bp_L$, call $\cU_p$ the collection of $\delta\times \delta^{A_L}$-tube segments $U=T\cap \bp_L$ with $T\in\cT_p$. Let also $\cU[\bp_L]$ denote the collection of all these distinct segments for $p\in \bp_L\cap \cP$.

		By Lemma \ref{tradim}, $\cU_p$ is a (rescaled) $(\delta^{A_{L+1}-A_L},s,\delta^{-(A_{L+1}-A_L)\eta_1} )$-set, with
		$\eta_1=o_\epsilon(1)=o_\upsilon(1).$ The set $\bp_L\cap\cP$ is a rescaled $(\delta^{A_{L+1}-A_L},t_L,\delta^{-(A_{L+1}-A_L)\eta})$-set, with $\eta=o_\upsilon(1)$.

		\vspace{5pt}
		
		Using  \eqref{ghdfgsfhjsgf} and \eqref{lalaland} we find
		$$\min(\frac{s+t_L}{2}, t_L, 1)=\frac{s+t_L}{2}+o_\upsilon(1).$$To uniformize existing notation with the forthcoming one, we denote $\cU[\bp_L]$ by $\cT_{\bp_L}$ and $|\cT_p|_{\delta^{A_{L+1}-A_L}}$ by $M_{\bp_L}$.
		\smallskip
		
		By the Furstenberg set estimate Theorem \ref{Fur} applied to (the rescaled version of) $(\bp_L\cap \cP, \cT_{\bp_L})$ we get
		\[
		|\cT_{\bp_L}| \gtrsim  \delta^{-(A_{L+1}-A_L)(\frac{t_L+s+o_\upsilon(1))}{2} +\epsilon(o_\upsilon(1),o_\upsilon(1),s,t_L))}M_{\bp_L}.
		\]
		Let $\bar{\epsilon}(\eta,\eta_1,s)=\max_{s\le t\le 2-s}\epsilon(\eta,\eta_1,s,t)$. We have $\lim_{\eta,\eta_1\to 0}\bar{\epsilon}(\eta,\eta_1,s)=0$, so in particular, $\bar{\epsilon}(o_\upsilon(1),o_\upsilon(1),s)=o_\upsilon(1)$.  Given \eqref{ghdfgsfhjsgf} and \eqref{doi}, we have $\epsilon(\eta,\eta_1,s,t_l)\le \bar{\epsilon}(\eta,\eta_1,s)$   for each $l\ge l_0$.
		Recalling also that by (iii) of Theorem \ref{justonegood} we have  $|\bp_L\cap \cP|_{\delta}\sim  \delta^{-(A_{L+1}-A_L)(t_L+o_\upsilon(1))}$, we conclude that
		\[
		\#(\bp_{L+1}\subset\bp_L)=|\bp_L\cap \cP| \lesssim \delta^{(A_{L+1}-A_L)(s+o_\upsilon(1))} (\frac{|\cT_{\bp_L}|}{ M_{\bp_L}})^2.
		\]

		We repeat this argument for the next scale.
		We use \cite[Proposition 4.1]{OS}, to cover the $\delta$-tubes through each $\bp_L\subset \bp_{L-1}$ with $M_{\bp_{L-1}}$ many $\delta^{A_L}\times \delta^{A_{L-1}}$-tube segments lying inside $\bp_{L-1}$. We write $\cT_{\bp_{L-1}}$ for the collection of all these tube segments.
		Applying the argument from the first step to $\bp_{L-1}\in \cP_{\delta^{A_{L-1}}}$ and $\cT_{\bp_{L-1}}$, we are led to
		\[
		\#(\bp_{L}\subset\bp_{L-1}) \lesssim \delta^{(A_L-A_{L-1})(s+o_\upsilon(1))}(\frac{|\cT_{\bp_{L-1}}|}{ M_{\bp_{L-1}}})^2 .
		\]
		We repeat this step, with the last inequality in the iteration being (recall that $\bp_1=[0,1]^2$)
		\[
		\#(\bp_{2}\subset\bp_{1}) \lesssim \delta^{(A_2-A_{1})(s+o_\upsilon(1))}(\frac{|\cT_{\bp_{1}}|}{ M_{\bp_{1}}})^2 .
		\]
		Proposition 5.2 \cite{OS} explains how to efficiently use Proposition 4.1 from the same paper, in a way that the product of these estimates is kept under control. The following is recorded in Proposition 3.8 of \cite{RW}
		$$\prod_{l=1}^{L}\max_{\bp_l}\frac{|\cT_{\bp_{l}}|}{ M_{\bp_{l}}} \lessapprox\frac{|\cT|}{r},$$
		with $\lessapprox$ hiding multiple logarithmic losses.

		We have described what happens if all intervals are good. For the bad intervals, a trivial inequality will hold, with the bound $\delta^{(A_l-A_{l-1})s} $ replaced by the crude estimate $\delta^{-O(A_l-A_{l-1})} $,
		\[
		\#(\bp_{l}\subset\bp_{l-1}) \lesssim \delta^{-O(A_l-A_{l-1})}(\frac{|\cT_{\bp_{l-1}}|}{ M_{\bp_{l-1}}})^2 .
		\]
		The contribution from the bad intervals is thus $\delta^{o_\upsilon(1)}$.
		Combining all these inequalities we get
		\[
		|\cP|=\Pi_{l=2}^{L+1} \#(\bp_{l}\subset\bp_{l-1}) \lessapprox \delta^{s+o_\upsilon(1)} \frac{|\cT|^2}{r^2}.
		\]
		
		\vspace{ 5pt}
		
		In addition to factors of the form  $\delta^{(A_{l+1}-A_l)o_\upsilon(1)}$ incurred at each step $l$, that contribute a total  $\delta^{o_\upsilon(1)}$-loss, there are implicit constants hidden in the notation $\lesssim$,  that are $O(1)$. Recall that $L$ is independent of $\delta$, and only depends on $\upsilon$, so this contribution is $O(1)^{L}=O_\upsilon(1)=O_\epsilon(1)$. Finally, recall that $o_\upsilon(1)=o_\epsilon(1)$.
		
	\end{proof}

	\section{The high-low method}
	\label{sec6}
	The following result traces its origins to \cite{GSW}. We will use it in the next section  to perform induction on scales, in the case when Item (2) in Theorem \ref{thm: delta-s} holds. 	
	
	\begin{lemma}[The high-low method]\label{lem: high-low}
		Let $\beta>0$.
		Let $\cT$ be a collection of $\delta$-tubes and let $r_0\ge 1$. Suppose $|\cP_{r_0}(\cT)|\ge \delta^{-2-\beta}/r_0$. Then there exists a scale $1\ge \tilde{\delta}\ge \delta^{1-\frac\beta{2}}$ such that for at least half of the $\delta$-squares $q$ (with center $c_q$) in  $\cP_{r_0}(\cT)$,
		the number of tubes in $\cT$ that intersect the ball $ B(c_q,\tilde{\delta}\delta^{-\upsilon})$ is $\gtrsim_\upsilon r_0\tilde{\delta}/\delta$, for each $\upsilon>0$.
	\end{lemma}
	\begin{proof}Our argument follows \cite{GSW}. We fix a  positive Schwartz function $\phi:\R^2\to[0,\infty)$ with $\phi\ge 1_{B(0,1)}$ and with Fourier transform supported inside $B(0,1)$.
		Let $f=\sum_{T\in \cT} \phi_T$, where $\phi_T$ is the affine rescaling of $\phi$ that satisfies $\phi_T\ge 1_T$ and has Fourier support inside the  $1\times \delta^{-1}$ rectangle centered at the origin,  dual to $T$.

		If $|\cT|\gtrsim \delta^{-1}r_0$, we are done by taking $\tilde{\delta}=1$.  So we may assume that  $|\cT|\ll \delta^{-1}r_0$. Assume $\psi$ is smooth and $ 1_{B(0,\delta^{-1+\beta/2})}\le \psi\le 1_{B(0,2\delta^{-1+\beta/2})}$.
		Decompose $f=f^h+f^l$ into the high frequency part $f^h$ and low frequency part $f^l$ where $f^l:= (\widehat{f}\psi )^{\vee}$.
		For each $x$,  at least one of the inequalities $|f(x)|\le 2|f^h(x)|$ and  $|f(x)|\le 2|f^l(x)|$ will hold.
		\smallskip
		
		Assume the former holds for a subset of $\cP_{r_0}(\cT)$ with at least half its measure. We show below that this forces that
		\begin{equation}
			\label{imp}
			r_0^2\delta^2|\cP_{r_0}(\cT)| \lesssim  \int |f^h|^2 \lesssim \delta^{1-\beta}  |\cT| .
		\end{equation}
		However,  this cannot hold because $|\cP_{r_0}(\cT)|\geq \delta^{-2-\beta}/r_0$ and $|\cT|\ll \delta^{-1}r_0$.
		
		To see the second inequality in \eqref{imp}, partition $\mathbb{S}^1$ into $\delta^{1-\frac\beta2}$-arcs. For each such arc, consider a tiling of the plane with thin $ \delta^{1-\frac{\beta}{2}}\times 1$ rectangles $R$, with the long side pointing in the direction normal to the arc.
		
		Group the tubes in $\cT$ according to which $R$ they belong to. Write
		$$\phi_T^l=(\widehat{\phi_T} \psi)^{\vee}$$
		and
		$f_R^h=\sum_{T\subset R}(\phi_{T}-\phi_T^l)=f_R-f_{R}^l$. It is easy to see that the functions $f_R^h$ form an   almost orthogonal family, due to the combination of spatial and frequency localization.
		
		Assume  $R$ contains $M_R$ tubes. Note that $M_R\lesssim \delta^{-\beta}$.
		Each $f_R^l$ is essentially supported in (a slight enlargement of) $R$ and
		$$\|f_R^l\|_{\infty}\lesssim M_R\delta^{\beta/2}.$$
		Thus
		$$\|f_R^l\|_2^2\lesssim M_R^2\delta^\beta|R|\sim M_R^2\delta^{1+\frac{\beta}{2}},$$
		and
		$$\sum_R\|f_R^l\|_2^2\lesssim\delta^{1+\frac\beta2}\max_{R} M_R\sum_RM_R\lesssim \delta^{1-\frac\beta2}|\cT|.$$
		Also,
		$$\sum_R\|\sum_{T\subset R}\phi_T\|_2^2\lesssim \sum_RM_R\|\sum_{T\subset R}\phi_T\|_1\lesssim \delta^{1-\beta}|\cT|.$$
		Combining these two inequalities with almost orthogonality we find
		$$\|f^h\|_2^2\lesssim \sum_R\|f_R^h\|_2^2\lesssim\sum_R\|f_R^l\|_2^2+\sum_R\|f_R\|_2^2 \lesssim \delta^{1-\beta}|\cT|.$$

		We have learned that $f(x)\le 2|f^l(x)|$ for $x$ in a subset of $\cP_{r_0}(\cT)$ with at least half its size. It follows that for at least half of the squares $q$ we have
		\[
		r_0 \lesssim \delta/\tilde{\delta}  \cdot \#\{ T\in \cT: T\cap B(c_q, \tilde{\delta}\delta^{-\upsilon})\neq \emptyset\}
		\]
		for $\tilde{\delta}=\delta^{1-\frac\beta2}$. This is because $\|\phi_T^l\|_\infty\lesssim \delta/\tilde{\delta}$, and $\phi_T^l$ is negligible outside the $\delta^{-\frac{\beta}2-\upsilon}$-thickening of $T$.
		
	\end{proof}

	\section{Two-ends reduction}
	\label{sec:new}
	We start with some notation. Consider a collection $\cT_p$ of $\delta$-tubes intersecting some $p\in\cD_\delta$. Given a scale $\delta\le \rho\le 1$, we let $\bp_\rho$ be the dyadic square with side-length $\rho$ containing $p$. We introduce the collection of $\delta\times \rho$-segments
	$$\cU_{\rho}(\cT_p)=\{U_\rho= T\cap\bp_\rho,\text{ for some }T\in\cT_p \}.$$
	We will use that
	\begin{equation}
		\label{ljkvkjcvkjkvjg}
		|\cU_{\rho}(\cT_p)|=|\cT_p^{\delta/\rho}|=|\cT_p|_{\delta/\rho}.
	\end{equation}
	Assume $\cT_p$ is uniform and assume $\tilde{\cT}_p\subset \cT_p$ satisfies $|\tilde{\cT}_p|\ge C^{-1}|\cT_p|$. Using \eqref{ljkvkjcvkjkvjg} it is immediate that
	\begin{equation}
		\label{iodiruiortug[igitop]}
		|\cU_{\rho}(\tilde{\cT}_p)|\gtrsim C^{-1}|\cU_{\rho}(\cT_p)|.
	\end{equation}
	The implicit constant in the inequality $\gtrsim$ depends on the uniformity constants.
	\smallskip
	
	The following theorem shows how to find subcollections of tube segments that are two-ends, see \eqref{jeciu4u5ut9i5t0o50to60-y9} and  preserve incidences, see \eqref{jeciu4u5ut9i5t0o50to60-y9hhytj7}. Moreover, \eqref{frkfreopgiprtogipotihpoyi} guarantees that the average number of incidences for tube segments is at least as large as the average for tubes. The proof combines graph theory methods (e.g. Lemma 2.8 in \cite{DvGo} and double counting for various subgraphs) with the  decomposition in Lemma~\ref{lem: goodinterval2}, that will help us identify the scale $\tilde{\rho}$. The reader eager to get to the main argument in Section \ref{sec7} may skip this proof.
	\begin{theorem}
		\label{twoendreduc}	
		Fix $\epsilon>0$ small enough.
		Consider a collection $\cP\subset \cD_\delta$. For each $p\in\cP$ we let $\cT_p$ be an $\epsilon$-uniform set of $\delta$-tubes intersecting $p$. We assume all $\cT_p$ have the same branching function.  Assume
		\begin{equation}
			\label{wef[porpogioigporti]}
			\sum_{p\in\cP}|\cT_p|\ge \delta^{-2\epsilon}|\bigcup_{p\in\cP}\cT_p|.
		\end{equation}
		Then there is a scale
		
		\begin{equation}
			\label{jirjfrueuegutoguortuig}
			\delta^{1-\epsilon}\le {\tilde{\rho}}\le 1
		\end{equation}
		such that the following holds. Call $\cU_{{\tilde{\rho}},p}=\cU_{\tilde{\rho}}(\cT_p)$ and  $\cU_{\tilde{\rho}}=\cup_{p\in\cP}\cU_{{\tilde{\rho}},p}$. For  each $p\in\cP$ there is $\bar{\cU}_{\tilde{\rho},p}\subset \cU_{{\tilde{\rho}},p}$ with the following properties. For each $U_{\tilde{\rho}}\in\bar{\cU}_{\tilde{\rho}}:=\cup_{p\in\cP}\bar{\cU}_{{\tilde{\rho}},p}$, letting
		$$\cP_{U_{\tilde{\rho}}}=\{p\in\cP:\;U_{\tilde{\rho}}\in \bar{\cU}_{{\tilde{\rho}},p}\},$$
		we have
		\begin{equation}
			\label{jeciu4u5ut9i5t0o50to60-y9}
			|\cP_{U_{\tilde{\rho}}}\cap B_{{\tilde{\rho}}'}|\lesssim (\frac{{\tilde{\rho}}'}{\tilde{\rho}})^{\epsilon}(\frac{\tilde{\rho}}\delta)^{5\epsilon^3}|\cP_{U_{\tilde{\rho}}}|,\;\;\forall\;{\tilde{\rho}}'\ge \delta,
		\end{equation}
		\begin{equation}
			\label{jeciu4u5ut9i5t0o50to60-y9hhytj7}
			\sum_{p\in\cP}|\bar{\cU}_{{\tilde{\rho}},p}|\gtrsim  (\delta/{\tilde{\rho}})^{\epsilon^3}\sum_{p\in\cP}|\cU_{{\tilde{\rho}},p}|,
		\end{equation}
		\begin{equation}
			\label{frkfreopgiprtogipotihpoyi}
			\delta^{o_\epsilon(1)}\frac{\sum_{p\in\cP}|\cT_p|}{|\bigcup_{p\in\cP}\cT_p|}\lesssim \frac{\sum_{p\in\cP}|{\cU}_{\tilde{\rho},p}|}{|\bar{\cU}_{\tilde{\rho}}|}.\end{equation}
	\end{theorem}
	\begin{proof}
		Step 1. Uniformization.
		
		Call $\cT=\cup_{p\in\cP}\cT_p$.
		For each $T\in \cT$, define $Y(T):=\{ p\in \cP: T\in \cT_p\}$. Find an $\epsilon^4$-uniform subset $Y'(T)\subset Y(T)$ such that $|Y'(T)|\gtrsim \delta^{\epsilon^4}|Y(T)|$. Let $\cT'\subset \cT$ be such that
		\\
		\\
		(1) for each $T\in \cT'$, $Y'(T)$ has the same branching function, in particular  the same size, which in the future will be denoted by $|Y'(T)|$.
		\\
		\\
		(2) $|\cT'|\cdot|Y'(T)|= \sum_{T\in \cT'}|Y'(T)|\gtrsim \delta^{\epsilon^4} |\cP|\cdot |\cT_p|.$
		\\
		\\
		Let $G$  be the incidence graph between $\cP$ and $\cT'$ defined as $(p,T)\in G$ if $p\in Y'(T)$. So $|G|\sim |\cT'||Y'(T)|$.
		We apply Lemma 2.8 in \cite{DvGo} to $G$ to obtain $\cP_1\subset\cP$, $\cT_1\subset\cT'$ and a subgraph $G'\subset (\cP_1\times\cT_1)\cap G$   such that
		\\
		\\
		(3) $|G'|\gtrsim |G|$
		\\
		\\
		(4) for each $p\in\cP_1$, $\cT_p'=\{ T'\in \cT_p\cap \cT_1: (p, T')\in G'\}$  satisfies $|\cT_p'|\gtrsim \frac{|G|}{|\cP|}\gtrsim \delta^{\epsilon^4} |\cT_p|$
		\\
		\\
		(5) for each $T\in\cT_1$,  $Y''(T)=\{p\in Y'(T)\cap \cP_1: (p,T)\in G'\}$ satisfies  $|Y''(T)|\gtrsim \frac{|G|}{|\cT'|} \sim |Y'(T)|.$

		It follows that
		\begin{equation}
			\label{jifjirirugiugitkgprtogiopti}
			|\cP_1|\gtrsim \delta^{\epsilon^4} |\cP|
		\end{equation}
		and
		$$
		|\cT_1|\gtrsim  |\cT'|.$$ Also, for each $(p,T)\in \cP_1\times\cT_1$ we have $T\in\cT_p'$ if and only if $p\in Y''(T)$.
		\\
		\\
		Step 2. The definition of ${\tilde{\rho}}$. Proof  of ${\tilde{\rho}}\ge \delta^{1-\epsilon}$.
		
		Consider the branching function $f$ of $Y'(T)$. Then using (2) followed by \eqref{wef[porpogioigporti]} we find
		\begin{equation}
			\label{weuioruiuruvhguithg}
			\delta^{f(0)-f(1)}=|Y'(T)|\gtrsim \frac{\delta^{\epsilon^4}|\cP||\cT_p|}{|\cT|}\ge \delta^{\epsilon^4-2\epsilon}>\delta^{-\epsilon-\epsilon^4}.
		\end{equation}

		Apply Lemma~\ref{lem: goodinterval2}  to $f$ to find a decomposition $\{[a_j, a_{j+1}]\}$ of $[0,1]$ and a sequence of $\sigma_j<\sigma_{j+1}$ such that for each  $x\in [a_j, a_{j+1}]$
		\begin{equation}\label{eq: interval1}
			f(x) \geq f(a_j)+ \sigma_j(x-a_j) - \epsilon^4 (a_{j+1}-a_j)
		\end{equation}
		and
		\begin{equation}\label{eq: interval2}
			f(a_{j+1})\leq f(a_j)+ (\sigma_j+\epsilon^4)(a_{j+1}-a_j).
		\end{equation}
		
		Let $b$ be the smallest $a_j$ such that $\sigma_j \geq \epsilon$. Such a $b$ exists because of\eqref {eq: interval2} and \eqref{weuioruiuruvhguithg}. It is possible that $b=0$, but this is harmless. We write ${\tilde{\rho}}=\delta^b$.
		
		To prove that ${\tilde{\rho}}\ge \delta^{1-\epsilon}$ we use three things. On the one hand, \eqref{eq: interval2} and the definition of $b$ imply that
		$f(b)-f(0)\le b(\epsilon+\epsilon^4).$
		On the other hand, since $f$ is 1-Lipschitz, $f(1)-f(b)\le 1-b.$ When combined with \eqref{weuioruiuruvhguithg} we get
		$$2\epsilon-\epsilon^4\le f(1)-f(0)\le 1+b(\epsilon+\epsilon^4-1),$$
		which gives $1-b\ge \frac{\epsilon-2\epsilon^4}{1-\epsilon-\epsilon^4}\ge \epsilon.$
		\\
		\\
		Step 3. The construction of $\cT_{1,p}$.

		For $p\in\cP_1$ define a maximal subset $\cT_{1,p}\subset \cT_p'$  such that for each $\bT\in \cT^{\delta/{\tilde{\rho}}}_{1,p}$ we have $\bT\cap \cT_{1, p} = \bT \cap \cT_p'$ and
		\begin{equation} \label{eq: T1p}
			|\bT\cap \cT_{1, p}|\ge \delta^{2\epsilon^4}|\cT_p|\cdot |\cT_p|_{\delta/{\tilde{\rho}}}^{-1}.
		\end{equation}Note that, due to (4)
		\begin{equation}
			\label{edirifuregoptrpogiy oh}
			|\cT_p'\setminus \cT_{1,p}|=\sum_{\bT\in{\cT_p'}^{\delta/{\tilde{\rho}}}\setminus \cT^{\delta/{\tilde{\rho}}}_{1,p}}|\bT\cap \cT'_p|\le \sum_{\bT\in\cT_p^{\delta/{\tilde{\rho}}}}\delta^{2\epsilon^4}|\cT_p|\cdot |\cT_p|_{\delta/{\tilde{\rho}}}^{-1}=\delta^{2\epsilon^4}|\cT_p|\lesssim \delta^{\epsilon^4}|\cT_p'|.\end{equation}
		\\
		\\
		Step 4. The construction of $Y_1(T)$
		
		For each  $U_{\tilde{\rho}}\in \cU_{\tilde{\rho}}$, define $\cT(U_{\tilde{\rho}})$ as the collection  of those $T\in\cT_1$  such that $U_{\tilde{\rho}}\subset T$ and  $U_{\tilde{\rho}}\cap Y''(T)\neq \emptyset$.
		
		For each $T\in\cT_1$ we prove the existence of  $Y_1(T)\subset Y''(T)$ such that for each  $U_{\tilde{\rho}}\in\cU_{\tilde{\rho}}$ with $U_{\tilde{\rho}}\subset T$ and $U_{\tilde{\rho}}\cap Y_1(T)\not=\emptyset$ we have
		$$U_{\tilde{\rho}}\cap Y_1(T)= \{ p\in U_{\tilde{\rho}}:  T\in \cT_{1, p}\},$$
		\begin{equation}
			\label{fjkrfj rji jgi oj}
			|U_{\tilde{\rho}}\cap Y_1(T)|\gtrsim\delta^{\epsilon^4}|U_{\tilde{\rho}}\cap Y'(T)|\sim  \delta^{\epsilon^4} \frac{|Y'(T)|_{\delta}}{|Y'(T)|_{{\tilde{\rho}}}},
		\end{equation}
		and defining  $\cT_1(U_{\tilde{\rho}})$ as the set of those $T\in\cT_1$ such that $U_{\tilde{\rho}}\subset T$ and $U_{\tilde{\rho}}\cap Y_1(T)\neq \emptyset$, then
		\begin{equation}\label{eq: T1U}
			|\cT_1(U_{\tilde{\rho}})|\ge \delta^{\epsilon^4} |\cT(U_{\tilde{\rho}})|.
		\end{equation}

		For each $U_{\tilde{\rho}}\subset T\in \cT_1$, define  $Y_0(U_{\tilde{\rho}}, T)=\{ p\in Y''(T)\cap U_{\tilde{\rho}}: T\in \cT_{1, p}\}$. Define $Y_0(T)=\cup_{U_{\tilde{\rho}}\subset T} Y_0(U_{\tilde{\rho}}, T)$, where the union is over all $U_{\tilde{\rho}} \subset T$ with $|Y_0(U_{\tilde{\rho}}, T)|\geq \delta^{\epsilon^4}\frac{|Y'(T)|_{\delta}}{|Y'(T)|_{{\tilde{\rho}}}}$.
		For each $U_{\tilde{\rho}}$, define $\cT_1(U_{\tilde{\rho}})$ as the set of those $T\in\cT_1$ such that $U_{\tilde{\rho}}\subset T$ and $U_{\tilde{\rho}}\cap Y_0(T)\neq \emptyset$. Define $Y_1(T) =\cup_{U_{\tilde{\rho}}\subset T} Y_0(U_{\tilde{\rho}}, T)$, where the union  is over all $U_{\tilde{\rho}}\subset T$ with $|\cT_1(U_{\tilde{\rho}})|\geq \delta^{\epsilon^4} |\cT(U_{\tilde{\rho}})|$.  Then $Y_1(T)$ and $\cT_1(U_{\tilde{\rho}})$ satisfy all the requirements.

		We prove that
		\begin{align*}
			\sum_{T\in \cT_1} |Y''(T)\setminus Y_1(T)| & \leq |\cP_1|\cdot \delta^{\epsilon^4}{|\cT_p'|} + |\cT_1|\cdot \delta^{\epsilon^4} \frac{|Y'(T)|_{\delta}}{|Y'(T)|_{{\tilde{\rho}}}}\cdot |Y'(T)|_{{\tilde{\rho}}} \\
			&+ \delta^{\epsilon^4} \sum_{U_{\tilde{\rho}}} |\cT(U_{\tilde{\rho}})| \cdot \frac{|Y'(T)|_{\delta}}{|Y'(T)|_{{\tilde{\rho}}}}.
		\end{align*}
		The three terms record the three sources of edge losses in the graph $G'$, whose size we recall to be $|G'|=\sum_{T\in\cT_1}|Y''(T)|$. The first term comes from \eqref{edirifuregoptrpogiy oh}, the second from \eqref{fjkrfj rji jgi oj}, while the third from \eqref{eq: T1U}. The verification is left to the reader.

		We next observe that each of the three terms is (see (5) for the $\sim$)
		$$
		\lessapprox \delta^{\epsilon^4} \sum_{T\in \cT_1} |Y'(T)|\sim \delta^{\epsilon^4} \sum_{T\in \cT_1} |Y''(T)|.$$
		It suffices to check this for the last term. Indeed, note that
		\[
		\sum_{U_{\tilde{\rho}}}|\cT(U_{\tilde{\rho}})|=\sum_{U_{\tilde{\rho}}}\sum_{T\in\cT_1:\atop{U_{\tilde{\rho}}\subset T\atop{U_{\tilde{\rho}}\cap Y''(T)\not=\emptyset}}}1=\sum_{T\in\cT_1}\sum_{U_{\tilde{\rho}}\subset T:\atop{U_{\tilde{\rho}}\cap Y''(T)\not=\emptyset}}1=\sum_{T\in\cT_1}|Y''(T)|_{\tilde{\rho}}\le
		\sum_{T\in \cT_1} |Y'(T)|_{\tilde{\rho}}.
		\]
		\\
		\\
		Step 5. Definition of $\bar{\cU}_{{\tilde{\rho}},p}$. Verification of \eqref{jeciu4u5ut9i5t0o50to60-y9hhytj7}.
		
		Let $\cT_{1,p}'=\{T\in\cT_1:\;p\in Y_1(T)\}$. Then Step 4  shows that
		$$\sum_{p\in\cP_1}|\cT_{1,p}'|=\sum_{T\in\cT_1}|Y_1(T)|\sim \sum_{T\in\cT_1}|Y''(T)|.$$
		Since the last term equals $|G'|\overset{(3)}{\sim} |G|\ge \sum_{p\in\cP_1}|\cT_p|$, we find that
		$$\sum_{p\in\cP_1}|\cT_{1,p}'|\gtrsim \sum_{p\in\cP_1}|\cT_{p}|.$$
		It follows that there is $\cP_1'\subset \cP_1$ with $|\cP_1'|\gtrsim |\cP_1|$ such that $|\cT_{1,p}'|\gtrsim |\cT_p|$ for each $p\in \cP_1'$. When combined with \eqref{iodiruiortug[igitop]}, this shows that for each $p\in\cP_1'$ we have $|\cU_{\tilde{\rho}}(\cT'_{1,p})|\gtrsim |\cU_{{\tilde{\rho}},p}|$ (recall our notation $\cU_{{\tilde{\rho}},p}=\cU_{\tilde{\rho}}(\cT_{\tilde{\rho}})$).
		
		We let $\bar{\cU}_{{\tilde{\rho}},p}$ equal $\cU_{{\tilde{\rho}}}(\cT_{1,p}')$ if $p\in\cP_1$ and $\bar{\cU}_{{\tilde{\rho}},p}=\emptyset$ if $p\in\cP\setminus\cP_1$. Using \eqref{jirjfrueuegutoguortuig} and  \eqref{jifjirirugiugitkgprtogiopti}, the inequality \eqref{jeciu4u5ut9i5t0o50to60-y9hhytj7} is now immediate, as $|\cU_{{\tilde{\rho}},p}|$ is constant for $p\in\cP$.
		\\
		\\
		Step 6. Verification of \eqref{jeciu4u5ut9i5t0o50to60-y9}.

		Fix a segment $U_{\tilde{\rho}}\in\bar{\cU}_{\tilde{\rho}}$, so  $\cP_{U_{\tilde{\rho}}}\not=\emptyset$. That means that there is $T\in\cT_1$ such that $U_{\tilde{\rho}}\subset T$ and $U_{\tilde{\rho}}\cap Y_1(T)\not=\emptyset$. Define $\cI(U_{\tilde{\rho}})= \cup_{T\in \cT_1(U_{\tilde{\rho}})} (Y_1(T)\cap U_{\tilde{\rho}})$. Note that $\cI(U_{\tilde{\rho}})=\cP_{U_{\tilde{\rho}}}$. Let
		\[
		X(U_{\tilde{\rho}})=\{ (p, T)\in \cI(U_{\tilde{\rho}})\times \cT_1(U_{\tilde{\rho}}): p\in Y_1(T)\}.
		\]
		Then, by \eqref{fjkrfj rji jgi oj}  $$|X(U_{\tilde{\rho}})|\ge \sum_{T\in\cT_1(U_{\tilde{\rho}})}|Y_1(T)\cap  U_{\tilde{\rho}}|\gtrsim \delta^{ \epsilon^4} \delta^{-f(1)+f(b)} |\cT_1(U_{\tilde{\rho}})|.$$ Since for each $p\in\cI(U_{\tilde{\rho}})$ the number of $T$ such that $(p, T)\in X(U_{\tilde{\rho}})$ is $\leq |\cT|_p \cdot |\cT_p|_{\delta/{\tilde{\rho}}}^{-1}$ (recall that $U_{\tilde{\rho}}\subset T$, so the upper bound comes from the non-concentration hypothesis on $\cT_p$), we have
		$$|X(U_{\tilde{\rho}})|\le |\cI(U_{\tilde{\rho}})||\cT|_p \cdot |\cT_p|_{\delta/{\tilde{\rho}}}^{-1}.$$
		Combining the upper and lower bounds we find
		\begin{equation}\label{eq: IU1}
			|\cI(U_{\tilde{\rho}})|\gtrsim \frac{ \delta^{-f(1)+f(b) + \epsilon^4} |\cT_1(U_{\tilde{\rho}})| }{ |\cT_p|\cdot |\cT_p|_{ \delta/{\tilde{\rho}}}^{-1}}.
		\end{equation}

		For any ${\tilde{\rho}}' =\delta^c <\delta^b$ and each square $B_{{\tilde{\rho}}'}$,  let
		\[
		X'(U_{\tilde{\rho}})=\{ (p, T)\in (\cI(U_{\tilde{\rho}})\cap B_{{\tilde{\rho}}'}) \times \cT(U_{\tilde{\rho}}): p\in Y''(T)\}.
		\]
		Then
		$$|X'(U_{\tilde{\rho}})|\le \sum_{T\in \cT(U_{\tilde{\rho}})}|Y''(T)\cap B_{{\tilde{\rho}}'}|\le \sum_{T\in \cT(U_{\tilde{\rho}})}|Y'(T)\cap B_{{\tilde{\rho}}'}|\lesssim \delta^{-f(1)+f(c)}\cdot |\cT(U_{\tilde{\rho}})|.$$
		Moreover,  for each $p\in \cI(U_{\tilde{\rho}})\cap B_{{\tilde{\rho}}'}$, by \eqref{eq: T1p},  the number of $T\in \cT(U_{\tilde{\rho}})$ such that $p\in Y''(T)$ is $\gtrsim \delta^{ 2\epsilon^4} |\cT_p|\cdot |\cT_p|_{\delta/{\tilde{\rho}}}^{-1}$. Therefore,
		\begin{equation}\label{eq: IUrho'}
			|\cI(U_{\tilde{\rho}})\cap B_{{\tilde{\rho}}'}| \lesssim \frac{\delta^{-f(1)+f(c) - 2\epsilon^4}\cdot |\cT(U_{\tilde{\rho}})|}{|\cT_p|\cdot |\cT_p|_{\delta/{\tilde{\rho}}}^{-1}}.
		\end{equation}
		By \eqref{eq: T1U},\eqref{eq: IU1} and \eqref{eq: IUrho'} we conclude that
		\[
		|\cI(U_{\tilde{\rho}})\cap B_{{\tilde{\rho}}'}|\lesssim \delta^{f(c)-f(b)-4\epsilon^4} |\cI(U_{\tilde{\rho}})|.
		\]
		Recall that $c>b$ and that  $\sigma_j\ge \epsilon$ on $[b,1]$. Combining this with \eqref{eq: interval1} it follows that
		$f(c) -f(b)\geq \epsilon(c-b) -\epsilon^4$. We thus have
		\begin{equation}\label{eq: two-ends2}
			|\cI(U_{\tilde{\rho}})\cap B_{{\tilde{\rho}}'}|\lesssim \delta^{\epsilon(c-b)-5\epsilon^4} |\cI(U_{\tilde{\rho}})|=\delta^{-5\epsilon^4}({\tilde{\rho}}'/{\tilde{\rho}})^{\epsilon}|\cI(U_{\tilde{\rho}})|.
		\end{equation}
		\eqref{jeciu4u5ut9i5t0o50to60-y9} is now a consequence of this and \eqref{jirjfrueuegutoguortuig}.
		\\
		\\
		Step 7. Verification of \eqref{frkfreopgiprtogipotihpoyi}.
		
		By \eqref{eq: interval2} and our choice of $b$ we infer that $f(b)\leq b(\epsilon+\epsilon^4)$. Thus,  for each $T\in \cT_1$,
		\begin{equation}
			\label{fkoprepfoirepogiro[gip]}
			|Y'(T)|_{{\tilde{\rho}}}=\delta^{-f(b)}\le \delta^{-b(\epsilon+\epsilon^4)}= (\frac{1}{{\tilde{\rho}}})^{\epsilon+\epsilon^4}.
		\end{equation}
		Fix any $U_{\tilde{\rho}}\in\bar{\cU}_{\tilde{\rho}}$. Then, as in Step 6, taking ${\tilde{\rho}}'={\tilde{\rho}}$
		$$|\cI(U_{\tilde{\rho}})| \delta^{ 2\epsilon^4} |\cT_p|\cdot |\cT_p|_{\delta/{\tilde{\rho}}}^{-1}\lesssim |X'(U_{\tilde{\rho}})|\lesssim |\cT(U_{\tilde{\rho}})|\frac{|Y'(T)|}{|Y'(T)|_{\tilde{\rho}}}.$$
		Since by \eqref{fjkrfj rji jgi oj}
		$$|\cI(U_{\tilde{\rho}})|\ge \delta^{\epsilon^4} \frac{|Y'(T)|}{|Y'(T)|_{{\tilde{\rho}}}},$$
		we find that
		\begin{equation}
			\label{jefjreiofjriejpiotegjpiort}
			|\cT(U_{\tilde{\rho}})|\gtrsim \delta^{ 3\epsilon^4} |\cT_p|\cdot |\cT_p|_{\delta/{\tilde{\rho}}}^{-1}.
		\end{equation}
		Another double counting shows that
		$$
		|\bar{\cU}_{\tilde{\rho}}|\min_{U_{\tilde{\rho}}\in\bar{\cU}_{\tilde{\rho}}}|\cT(U_{\tilde{\rho}})|\le |\{(T,U_{\tilde{\rho}})\in\cT_1\times \bar{\cU}_{\tilde{\rho}}:\;U_{\tilde{\rho}}\subset T\text{ and }U_{\tilde{\rho}}\cap Y''(T)\not=\emptyset\}|\le |\cT_1||Y'(T)|_{\tilde{\rho}},
		$$
		which when combined with \eqref{fkoprepfoirepogiro[gip]} and \eqref{jefjreiofjriejpiotegjpiort} shows that
		$$|\bar{\cU}_{\tilde{\rho}}||\cT_p|\lesssim |\cT_1||\cT_p|_{\delta/{\tilde{\rho}}}\delta^{-4\epsilon^4-\epsilon}.$$
		This gives \eqref{frkfreopgiprtogipotihpoyi}, as $|\cU_{{\tilde{\rho}},p}|=|\cT_p|_{\delta/{\tilde{\rho}}}.$
	\end{proof}
	
	\section{Proof of Theorem~\ref{thm: main}}
	\label{sec7}
	The parameter $s\in(0,\frac12]$ will be fixed throughout this section.
	
	The argument  will recycle ideas and estimates established earlier in the paper. Parts of the presentation will be kept slightly less formal. Here is one example. With the notation from \eqref{defst}, we need to prove that for each $\upsilon>0$ we have $\cS\cT(\delta,1,1)\lesssim_\upsilon \delta^{-\upsilon}$. The bootstrapping argument we will present in this section essentially establishes a bound of the form (we are oversimplifying a bit)
	$$\cS\cT(\delta,1,1)\le C_{\epsilon,1}(\frac1\delta)^{C_{\epsilon,2}}\cS\cT(\rho_\epsilon,C_{\epsilon,3},C_{\epsilon,3}),$$
	for each $\epsilon>0$, for some $\rho_\epsilon$ quantitatively larger than $\delta$. The constant $C_{\epsilon,3}$ (and part of $C_{\epsilon,2}$) arises due to the application of results in Section \ref {sec4}. The  relevant thing is that $\lim_{\epsilon\to 0}C_{\epsilon,2}=0$, which we continue to write as $C_{\epsilon,2}=o_{\epsilon}(1)$. Note however that this inequality trivially gets upgraded  to
	$$\cS\cT(\delta,C_{\epsilon,3},C_{\epsilon,3})\le C_{\epsilon,1}(\frac1\delta)^{C_{\epsilon,2}}\cS\cT(\rho_\epsilon,C_{\epsilon,3},C_{\epsilon,3}).$$
	This may be iterated to get 
	$$\cS\cT(\delta,C_{\epsilon,3},C_{\epsilon,3})\lesssim_{\epsilon}\delta^{o_\epsilon(1)},$$ which in turn implies the desired bound $\cS\cT(\delta,1,1)\lesssim_\upsilon \delta^{-\upsilon}$. For this reason,  we will ignore the distinction between $\cS\cT(\delta,1,1)$ and $\cS\cT(\delta,C_{\epsilon,3},C_{\epsilon,3})$ in the forthcoming argument. To this end, we denote $\cS\cT(\delta,1,1)$ by $\cS\cT(\delta)$. Moreover, for extra convenience, we may test inequality \eqref{defst} using only  $(1,1)$-sets $\cT$ (Definition \ref{KK}) that are maximal in size, that is, sets satisfying $|\Lambda|\sim \delta^{-s}$ and $|\cT_\theta|\sim \delta^{s-1}$ for each $\theta\in\Lambda$. This can be made precise using the methods from Section \ref{sec4}. This will be needed in order to facilitate the application of Proposition \ref{actthis}. 
	\smallskip
	
	In order to simplify the iteration in Step 3 below, we will use the fact that $\cS\cT(\rho)\lesssim \cS\cT(\delta)$ if $\delta<\rho$. We sketch a heuristic argument for this. Fix $r_\rho$ and a configuration $\cT(\rho)$ of $\rho$-tubes $T_\rho$ as in Theorem \ref{thm: main}. We call $\Lambda(\rho)\subset\mathbb{S}^1$ the associated collection of $\rho$-arcs. For each $I\in\Lambda(\rho)$ we consider a (rescaled) $(\delta/\rho,s)$-set $\Lambda_I$ of $\delta$-intervals inside $I$,  with cardinality $(\rho/\delta)^s$. The set $\Lambda(\delta)=\cup_{I\in\Lambda(\rho)}\Lambda_I$ is a $(\delta,s)$-set with cardinality $\sim \delta^{-s}$. For each $T_\rho\in \cT(\rho)$ in the direction $I$ and for each $J\in \Lambda_I$ we consider a (rescaled) $(\delta/\rho, 1-s)$-set $\cT(T_\rho,J)$ of $\delta$-tubes $T$ in the direction $J$, lying inside $T_\rho$ and with cardinality $(\rho/\delta)^{1-s}$. The union $\cT$ of these tubes satisfies the requirements of Theorem \ref{thm: main} at scale $\delta$. Fix a $\rho$-square $\bp_\rho$ intersecting some $T_\rho$. Since $T_\rho$ contains $\sim \rho/\delta$ tubes $T\in\cT$, the corresponding $\delta\times \rho$ segments $T\cap\bp_\rho$ have a total volume of $\sim \rho^2$, comparable to the area of $\bp_\rho$. A random choice of the tubes $T$ inside $T_\rho$ will allow us to assume that a large fraction of (the $\delta$-squares in) $\bp_\rho$ intersects at least one such $T$. Thus, for generic $\bp_\rho\in\cP_{r_\rho}(\cT(\rho))$, a large fraction of its $\delta$-squares will intersect $\sim r_\rho$ tubes $T\in\cT$. This leads to
	$$|\cP_{r_\rho}(\cT(\rho))|(\rho/\delta)^2\lesssim |\cP_{\sim r_\rho}(\cT)|\le \cS\cT(\delta)\frac{|\cT|^2}{(\sim r_\rho)^3}\sim \cS\cT(\delta)\frac{\delta^{-2}}{r_\rho^3},$$
	showing that $|\cP_{ r_\rho}(\cT(\rho))|\lesssim \cS\cT(\delta)\frac{|\cT(\rho)|^2}{r_\rho^3}$, as desired.
	\smallskip

	We need to verify that for each $\upsilon>0$
	\begin{equation}
		\label{desiredfinally}
		\cS\cT(\delta)\lesssim_\upsilon\delta^{-\upsilon}.\end{equation}To achieve this, we fix $r_0\ge 1$, $\delta>0$	
	and an arbitrary collection $\cT$ as in Theorem~\ref{thm: main}. We aim to  prove the upper bound  $$|\cP_{r_0}(\cT)|\lessapprox \frac{|\cT|^2}{r^3}\sim \frac{\delta^{-2}}{r_0^3}.$$
	There are several steps.
	\\
	\\
	Step 1. Uniformization:
	
	Fix $\epsilon>0$. Let $\epsilon_0$ satisfy \eqref{riogitgoitophiophi}.

	We will assume $\cT$ and $\Lambda$ are $\epsilon_0$-uniform. This will incur negligible $\delta^{-O(\epsilon_0)}$-losses. We will shortly see that these are  consistent with the bound \eqref{difjiougireugitugoudoiii} we aim to prove.
	
	By considering a $\delta^{\epsilon_0}$ smaller collection of squares  $\cP\subset \cP_{r_0}(\cT)$, we may assume as before (see the proof of Theorem \ref{thm: delta-s2}$\implies$Theorem \ref{thm: delta-s}) that the set of tubes through each $p\in\cP$ contains an $\epsilon_0$-uniform  subset $\cT_p$ with a universal branch function. In particular, call $r$ the common values of $|\cT_p|$. This produces another affordable  $\delta^{-\epsilon_0}$-loss, this time in the value of $r_0$. The best we can say is that
	\begin{equation}
		\label{ffjkrgprthpyojpopo}
		r\gtrsim r_0\delta^{\epsilon_0}.
	\end{equation}
	
	By an additional refinement we may also assume that $\cP$ is $\epsilon_0$-uniform, at the cost of yet another acceptable $\delta^{-\epsilon_0}$ loss. 
	Write  $P=|\cP|$. Since $\epsilon_0$ is a function of $\epsilon$, all $\epsilon_0$-uniform sets will also be approximate $\epsilon$-uniform, see Remark \ref{approxunifsets}. All implicit constants will depend only on $\epsilon$.
	\\
	\\
	Step 2: Scale inflation.

	We apply Theorem \ref{justonegood} to $\cT_p$. In this application of the theorem, we only need to work with the interval $[A_1,A_2]$. The numbers  $A_2$ and $t_1$ are independent of $p$. Recall that
	$$A_2\ge \epsilon^{-1}\epsilon_0,\;\;t_1\le \log_{1/\delta}r+\epsilon.$$

	We write $\rho_0=\delta^{1-A_2}$ and $\psi(\epsilon)=\epsilon^{-1}\epsilon_0=\epsilon^{2\epsilon^{-1}-1}$. We have that
	\begin{equation}
		\label{koig9igoi0i0}
		\delta^{-\epsilon_0}\le (\frac{\rho_0}{\delta})^{\epsilon\psi(\epsilon)}.\end{equation}
	\\
	\\
	Step 3: Induction on scales.

	We will prove that
	\begin{equation}
		\label{difjiougireugitugou}
		\begin{split}
			P&\le\frac{|\cT|^2}{r^3} C_{\epsilon}(\log\delta^{-1})^{O(1)}\big{\{}(\rho_0/\delta)^{a_\epsilon}\max_{\delta\le \rho\le \rho_0}(\rho_0/\rho)^{t_1-\frac12} [\cS\cT(\rho_0)\cS\cT(\rho)]^{1/2}
			\\
			&+\max_{  \delta (\rho_0/\delta)^{c_\epsilon}\lesssim \tilde{\delta}\le \rho_0}({\tilde{\delta}}/{\delta})^{b_\epsilon}\cS\cT(\tilde{\delta})\big{\}}.
		\end{split}
	\end{equation}
	The two terms on the right-hand side trace their origins to the two items in Theorem \ref{thm: delta-s}.
	The first comes from Item (1), the second from Item (2).
	
	The absolute  constants
	$a_\epsilon,b_\epsilon,c_\epsilon$ satisfy $\lim_{\epsilon\to 0}a_\epsilon,b_\epsilon=0$ and $c_\epsilon>0$. Moreover,
	$$\lim_{\epsilon\to0}\frac{\epsilon\psi(\epsilon)}{c_\epsilon}=0.$$
	See \eqref{inforrtrei} and \eqref{krgor[]oto5=-o3-=}.
	When combined with \eqref{koig9igoi0i0}, it shows that the $\delta^{-O(\epsilon_0)}$ losses are negligible. More precisely,
	$$\delta^{-O(\epsilon_0)}\lesssim (\frac{\tilde{\delta}}{\delta})^{o_\epsilon(1)},\text{ whenever } \delta (\rho_0/\delta)^{c_\epsilon}\lesssim \tilde{\delta}.$$
	This together with the estimate $|\cP_{r_0}(\cT)|\lesssim \delta^{-O(\epsilon_0)}P$,   \eqref{ffjkrgprthpyojpopo} and \eqref{difjiougireugitugou} leads to
	
	\begin{equation}
		\label{difjiougireugitugoudoi}
		\begin{split}
			|\cP_{r_0}(\cT)|&\le\frac{|\cT|^2}{r_0^3} C_{\epsilon}(\log\delta^{-1})^{O(1)}\big{\{}(\rho_0/\delta)^{a_\epsilon}\max_{\delta\le \rho\le \rho_0}(\rho_0/\rho)^{t_1-\frac12} [\cS\cT(\rho_0)\cS\cT(\rho)]^{1/2}
			\\
			&+\max_{  \delta (\rho_0/\delta)^{c_\epsilon}\lesssim \tilde{\delta}\le \rho_0}({\tilde{\delta}}/{\delta})^{b_\epsilon}\cS\cT(\tilde{\delta})\big{\}}.
		\end{split}
	\end{equation}
	The constants $C_\epsilon$, $a_\epsilon,b_\epsilon$ have changed slightly to incorporate the $\delta^{-O(\epsilon_0)}$ loss, but their properties remain unchanged.
	The parameters $\rho_0,t_1$ depend on $r$, thus on $r_0$,  but in a uniform way. More precisely, using our chief assumption that $s\le \frac12$, we find that $r\lesssim \delta^{-1/2}$, and thus  $t_1\le \frac12+o_\epsilon(1)$ (cf. (iv) in Theorem \ref{justonegood}). Since also  $\rho_0/\delta=\delta^{-A_2}\ge \delta^{-\psi(\epsilon)}$ with $\psi(\epsilon)$ independent of $r$,  inequality \eqref{difjiougireugitugoudoi} in fact gives the ($r_0$-independent) inequality (with a new $a_\epsilon$ that still goes to zero)
	$$\cS\cT(\delta)\le C_\epsilon(\log\delta^{-1})^{O(1)}\max_{\rho_0\ge \delta^{1-\psi(\epsilon)}} \big{\{} (\rho_0/\delta)^{a_\epsilon} [\cS\cT(\rho_0)\cS\cT(\delta)]^{1/2}$$
	$$+\max_{  \delta (\rho_0/\delta)^{c_\epsilon}\lesssim \tilde{\delta}\le \rho_0}({\tilde{\delta}}/{\delta})^{b_\epsilon}\cS\cT(\tilde{\delta})\big{\}}.$$
	To streamline a bit the inequality,  in the first line we have used that $\cS\cT(\rho)\lesssim \cS\cT(\delta)$.
	
	If the first line dominates, then we have $\cS\cT(\delta)\lesssim (\log\delta^{-1})^{O(1)} (\rho_0/\delta)^{o_\epsilon(1)}\cS\cT(\rho_0)$. Thus, the above inequality can be rewritten as
	\begin{align*}
		\cS\cT(\delta)&\le C_{\epsilon} (\log\delta^{-1})^{O(1)}\max_{\rho_0\ge \delta^{1-\psi(\epsilon)}}\big{\{} (\rho_0/\delta)^{a_\epsilon} \cS\cT(\rho_0)\\&
		+\max_{  \delta (\rho_0/\delta)^{c_\epsilon}\lesssim \tilde{\delta}\le \rho_0}({\tilde{\delta}}/{\delta})^{b_\epsilon}\cS\cT(\tilde{\delta})\big{\}}.
	\end{align*}
	Call $d_\epsilon=\psi(\epsilon)c_\epsilon>0$. Note that $\rho_0/\delta,\tilde{\delta}/\delta\gtrsim \delta^{-d_\epsilon}$. The above inequality implies that for each $\epsilon$ and $\delta$
	\begin{equation}
		\label{difjiougireugitugoudoiii}
		\cS\cT(\delta)\le C_{\epsilon} \max_{  \tilde{\delta}/\delta\gtrsim \delta^{-d_\epsilon}}({\tilde{\delta}}/{\delta})^{e_\epsilon}\cS\cT(\tilde{\delta}),
	\end{equation}
	for some $e_\epsilon\to0$.  The log terms have been absorbed into $({\tilde{\delta}}/{\delta})^{e_\epsilon}$.
	
	Note the delicate nature of this inequality.
	It contains no generic $\delta^{o_\epsilon(1)}$-losses, only the carefully calibrated term   $(\tilde{\delta}/\delta)^{e_\epsilon}$.
	This by itself is enough to  conclude \eqref{desiredfinally}  via a bootstrapping argument, as follows.
	
	Let $\upsilon_0$ be the infimum of the following  nonempty set (this is an interval)
	$$\{\upsilon>0:\;\exists\;C_\upsilon<\infty\text{ such that } \cS\cT(\delta)\le C_\upsilon\delta^{-\upsilon}\text{ for each }\delta\le 1\}.$$
	Assume for contradiction that $\upsilon_0>0$. This allows us to pick $\epsilon$ sufficiently small so that $e_\epsilon<\upsilon_0/2$. Then we pick $\upsilon$ larger than, but sufficiently close to $\upsilon_0$, so that $\upsilon(1-\frac{d_\epsilon}{2})<\upsilon_0$.
	We use $\cS\cT(\tilde{\delta})\le C_\upsilon\tilde{\delta}^{-\upsilon}$ in \eqref{difjiougireugitugoudoiii} with the $\epsilon$ picked above and get
	$$\cS\cT(\delta)\le C_\epsilon C_\upsilon(\delta^{-1})^{\upsilon-(\upsilon-e_\epsilon)d_\epsilon}.$$
	Since $\upsilon-(\upsilon-e_\epsilon)d_\epsilon<\upsilon(1-\frac{d_\epsilon}{2})<\upsilon_0$, this leads to a contradiction.
	\\
	\\
	Step 4: Proof of \eqref{difjiougireugitugou}. This occupies the remainder of this section.

	There are $\sim \delta^{-A_2(t_1+o_\epsilon(1))}$ tubes $\bT\in \cT_p^{\delta/\rho_0}\subset \cT^{\delta/\rho_0}$. We write $\cT_p\cap\bT$ for the collection $\cT_p[\bT]$ of tubes  $T\in \cT_p$ lying inside $\bT$. Due to uniformity, $|\cT_p\cap \bT|$ is either zero or $\sim \tilde{r}$. This $\tilde{r}$ is independent of $p$, but whether   $|\cT_p\cap \bT|$ is zero or $\sim \tilde{r}$ depends on $p$ (for each fixed $\bT$).
	We have that
	\begin{equation}
		\label{jewie9r93rr8948t958t9}
		\tilde{r}= |\cT_p^{\rho_0/\delta}|\sim  r(\delta/\rho_0)^{t_1+o_\epsilon(1)}
	\end{equation}
	and $|\cT_p\cap \bT|\sim \tilde{r}$ for  $\sim \frac{r}{\tilde{r}}$ of the tubes $\bT\in \cT^{\delta/\rho_0}$.
	
	Throughout the argument we may afford $(\log\delta^{-1})^{O(1)}$ losses, denoted by the symbol $\lessapprox$. These will arise in a few places, due to pigeonholing and the application of the auxiliary results from  Section \ref{sec4} (see also Remark \ref{logsareneeded}).
	
	We tile each $\bT\in \cT^{\delta/\rho_0}$ with  $\delta\times \rho_0$-segments $U$ having the same orientation as $\bT$. The length $\rho_0$ of these segments is informed by the fact that for each maximal cluster of tubes $T\subset\bT$
	intersecting a fixed square in the core $\frac12\bT$ of $\bT$, the intersection of the tubes in the cluster is a $\sim \delta\times \rho_0$-segment.
	
	For each $p\in\cP$ call $\cU_p=\cU_{\rho_0,p}$ the collection of all such segments containing $p$, one for each $\bT\in\cT_p^{\delta/\rho_0}$.
	Call $\cU=\cup_{p\in\cP}\cU_p$. Recall that
	\begin{equation}
		\label{fkrofipriprigpigp}
		|\cU_p|\sim r/\tilde{r}.
	\end{equation}
	
	We partition $[0,1]^2$ into squares $\bp_{\rho_0}$ with side length $\rho_0$, and call $\cU[\bp_{\rho_0}]$ ($\cP[\bp_{\rho_0}]$) those segments from $\cU$ (squares in $\cP$) that lie inside $\bp_{\rho_0}$. Write
	$$V_{\bp_{\rho_0}}=\frac{\sum_{p\in \cP[\bp_{\rho_0}]}|\cU_p|}{|\cU[\bp_{\rho_0}]|}.$$
	Recall the function $\varphi$ from Theorem \ref{justonegood} and the function $\psi$ from \eqref{koig9igoi0i0}.	We let  $h(\epsilon)$ be a positive quantity converging to zero slowly enough as $\epsilon\to 0$, such that
	\begin{equation}
		\label{inforr}
		\lim_{\epsilon\to 0}\frac{\varphi(\epsilon)}{h(\epsilon)^{7/4}}=0,
	\end{equation}
	and
	\begin{equation}
		\label{inforrtrei}
		\lim_{\epsilon\to 0}\frac{\epsilon\psi(\epsilon)}{h(\epsilon)^{7/4}}=0.
	\end{equation}
	We let $\cP_{1}$ be those $p\in\cP$ that lie inside a $\bp_{\rho_0}$ with $V_{\bp_{\rho_0}}\lesssim  (\rho_0/\delta) ^{2h(\epsilon)}$. Then $\cP_2=\cP\setminus\cP_1$.
	\\
	\\
	\noindent \textbf{Case 1.} Estimates for $|\cP_1|$.

	Each $U\in\cU$ is contained inside at least $\tilde{r}$ tubes $T$,  coming from exactly one $\bT\in\cT^{\delta/\rho_0}$.

	Observe that $\cT[\bT]$, after parabolic rescaling by $\rho_0/\delta$, becomes a $(O(1),O(1))$-set (recall Definition \ref{KK}), with scale $\delta$ replaced by $\rho_0$. This rescaling maps $U$ to a $\rho_0$-square and the collection $\cT[\bT]$ to the collection of $\rho_0\times 1$ tubes $\cT_{resc}(\bT)$.
	Proposition \ref{cor: mainKT} implies that
	\begin{equation}
		\label{iiiihhhhh2}
		|\cP_{\ge \tilde{r}}(\cT_{resc}(\bT))|\lessapprox \cS\cT(\rho_0) \rho_0^{-1}\frac{|\cT[\bT]|_{\delta}}{\tilde{r}^3}.
	\end{equation}
	On the other hand, using \eqref{fkrofipriprigpigp} we find
	\begin{equation}
		\label{iiiihhhhh1}
		\frac{r}{\tilde{r}}|\cP_1|\sim  \sum_{V_{\bp_{\rho_0}}\lesssim  (\rho_0/\delta) ^{2h(\epsilon)}} V_{\bp_{\rho_0}}\cdot |\cU[\bp_{\rho_0}]|\le (\rho_0/\delta) ^{o_\epsilon(1)}   \sum_{\bT\in \cT^{\delta/\rho_0}} |\cP_{\ge \tilde{r}}(\cT_{resc}(\bT))|.
	\end{equation}

	Using  \eqref{jewie9r93rr8948t958t9},  \eqref{iiiihhhhh2} and \eqref{iiiihhhhh1} we find after summation in $\bT\in\cT^{\delta/\rho_0}$
	\begin{align*}
		|\cP_1|& \lessapprox \cS\cT(\rho_0) \rho_0^{-1}(\rho_0/\delta)^{o_\epsilon(1)}  \frac{|\cT|}{r\tilde{r}^2}\\
		&\sim\cS\cT(\rho_0) (\frac{\rho_0}\delta)^{-1+o_\epsilon(1)}\frac{|\cT|^2}{r\tilde{r}^2} \text{, since   }|\cT|\sim \delta^{-1}\\
		&\sim \cS\cT(\rho_0) (\frac{\rho_0}{\delta})^{2t_1-1+o_\epsilon(1)}\frac{|\cT|^2}{r^3}.
	\end{align*}
	This verifies \eqref{difjiougireugitugou} in this case, as $|\cP_r(\cT)|$ is dominated by the first term in  \eqref{difjiougireugitugou}, with the maximum evaluated at $\rho=\delta$. Indeed, this follows by taking the geometric average of the estimate from above, with the trivial estimate $P\le \cS\cT(\delta)\frac{|\cT|^2}{r^3}.$
	\\
	\\
	
	\noindent \textbf{Case 2}.
	Estimates for $|\cP_2|$ .
	
	Here is a brief overview of the argument in this case.
	We first pursue the two-ends reduction from Section \ref{sec:new}.
	This will create the first intermediate scale, called $\rho$. At this scale Theorem \ref{thm: delta-s} becomes applicable. It leads to a dichotomy that we analyze in Case 2(a) and Case 2(b). In the first case we get another intermediate scale $\Delta$. The application of Lemma \ref{lem: high-low} at scale $\Delta$ creates the third intermediate scale $\tilde{\delta}$. It will be important to verify that all these scales are quantitatively larger than the initial scale $\delta$.
	\medskip
	
	Here are the details. We apply the following procedure to each  $\bp_{\rho_0}$ with  $V_{\bp_{\rho_0}}\gtrsim  (\rho_0/\delta) ^{2h(\epsilon)}$. We rescale $\bp_{\rho_0}$ by $1/\rho_0$ and map it to $[0,1]^2$. The segments $U\in\cU[\bp_{\rho_0}]$ will become $\delta\rho_0^{-1}\times 1$ tubes $T$. We apply Theorem \ref{twoendreduc} to these tubes, and the collection of rescaled squares in $\cP[\bp_{\rho_0}]$. The scale $\delta$ in Theorem \ref{twoendreduc} is replaced with $\delta\rho_0^{-1}$. The hypothesis \eqref{wef[porpogioigporti]} is satisfied with $\epsilon$ replaced by $h(\epsilon)$. Then Theorem \ref{twoendreduc} delivers a scale $\tilde{\rho}$ such that, according to \eqref{jirjfrueuegutoguortuig},
	$$\tilde{\rho}\ge (\frac{\delta}{\rho_0})^{1-h(\epsilon)}.$$
	Subject to a logarithmic loss we may and will assume that this scale is the same for each $\bp_{\rho_0}$.
	When we rescale back, this scale becomes $\rho:=\tilde{\rho}\rho_0$.
	Note that
	\begin{equation}
		\label{twoendindisg2oweifruiirugidoi}
		\frac{\rho}{\delta}\ge (\frac{\rho_0}\delta)^{h(\epsilon)}.
	\end{equation}
	For each $p\in\cP$, let $\cU_{1,p}$ be the collection of all $\delta\times \rho$ mini-segments $U_1$ containing $p$, one for each $\bT_1\in\cT_p^{\delta/\rho}$.

	Partition $\bp_{\rho_0}$ into $\rho$-squares $\bp_\rho$. Since $\cP$ is uniform, each $\bp_\rho$ intersecting $\cP$ will contain roughly the same number of squares $\cP[\bp_\rho]$ from $\cP$. For each such $\bp_\rho$, we call $\cU_1[\bp_\rho]=\cup_{p\in\bp_\rho}\cU_{1,p}$. We apply Theorem \ref{thm: delta-s} to the $\times \rho^{-1}$ rescaled version of $\bp_\rho$. The scale $\delta$ there is replaced with $\delta/\rho$. The collections $\cT_p$ in Theorem \ref{thm: delta-s} are the rescaled copies of the mini-segments $\cU_{1,p}[\bp_\rho]$. The $\epsilon$ in that theorem is replaced with $\epsilon':=h(\epsilon)$. Let us note the following.
	\\
	\\
	$\bullet$ The two-ends hypothesis \eqref{ ifj iojiojio} is satisfied due to \eqref{jeciu4u5ut9i5t0o50to60-y9}.
	\\
	\\
	$\bullet$  Due to \eqref{jeciu4u5ut9i5t0o50to60-y9hhytj7}, the hypothesis \eqref{o.k} is satisfied with $\eta=(\epsilon')^3$, for at least a $(\delta/\rho)^\eta$ fraction of the squares $\bp_\rho$. This is good enough for us, as each $\bp_\rho$ contains the same number of $p\in\cP_2$, and we will see that $(\delta/\rho)^\eta$  is a negligible loss.
	\\
	\\
	$\bullet$ $\cU_{1,p}$ can be identified with $\cT_p^{\delta/\rho}$. Theorem \ref{justonegood} (iii) together with Lemma \ref{tradim} imply that its $1/\rho$-rescaled copy  is a $(\delta/\rho,t_1, (\rho_0/\delta)^{\varphi(\epsilon)})$-set.
	Combining \eqref{inforr}   and \eqref{twoendindisg2oweifruiirugidoi}
	we conclude that this is   a $(\delta/\rho,t_1, (\rho/\delta)^{o_\epsilon(1)})$-set. Thus, it is also a $(\delta/\rho,t_1, (\rho/\delta)^{o_{\epsilon'}(1)})$-set.
	\\
	\\
	$\bullet$ We let $\bar{\cU}_1[\bp_\rho]$ be the subset of $\cU_1[\bp_\rho]$ obtained by rescaling back to $\bp_\rho$ the collection $\bar{\cU}_{\tilde{\rho}}$ provided by  Theorem \ref{twoendreduc}. Then \eqref{frkfreopgiprtogipotihpoyi} implies that for a $\gtrsim 1$ fraction of the contributing  $\bp_\rho\subset\bp_{\rho_0}$ we have
	\begin{equation}
		\label{kjrgirigutiguitruh}
		(\delta/\rho)^{o_\epsilon(1)}V_{\bp_{\rho_0}}\lesssim \frac{\sum_{p\in\cP[\bp_\rho]}|\cU_{1,p}|}{|\bar{\cU}_1[\bp_\rho]|}.\end{equation}
	This is because the numerator on the right-hand side has the same value for all contributing $\bp_\rho$ (since $\cP$ is uniform).
	\\
	\\
	Call $\cP_{\rho,gen}$ the collection of all $\bp_\rho$ with $\bp_\rho\cap\cP_2\not=\emptyset$, that satisfy \eqref{o.k} (after rescaling) and \eqref{kjrgirigutiguitruh}. With another logarithmic loss, we may also assume that
	\begin{equation}
		\label{crifjiregjitgirt hioujryio}
		\frac{\sum_{p\in\cP[\bp_\rho]}|\cU_{1,p}|}{|\bar{\cU}_1[\bp_\rho]|}\sim W\end{equation}
	for each $\bp_\rho\in\cP_{\rho,gen}$.
	
	We have
	\begin{equation}
		\label{bigchunk}
		|\cP_2|\lessapprox (\rho/\delta)^{h(\epsilon)^3}\sum_{\bp_\rho\in\cP_{\rho,gen}}|\bp_\rho\cap\cP_2|.
	\end{equation}
	
	Theorem~\ref{thm: delta-s} provides a dichotomy for each $\bp_\rho\in\cP_{\rho,gen}$, and we split the analysis in two subcases.
	\\
	\\
	\textbf{Case 2(a)}: Assume  Item $(2)$   of Theorem~\ref{thm: delta-s} happens for at least half of $\bp_\rho$. This means that there is a scale
	\begin{equation}
		\label{ifiofriufrirr09-02} \delta(\frac{\rho}{\delta})^{\sqrt{h(\epsilon)}}\le \Delta\le \rho
	\end{equation}
	such that at least a $(\delta/\rho)^{ O(h(\epsilon))}$-fraction $\cP_{\Delta,\bp_\rho}$ of $\cP_2\cap \bp_\rho$
	is covered by squares $B_\Delta$ satisfying
	\begin{equation}
		\label{pocirfireigig0}
		|\cP_2\cap B_\Delta|\gtrsim (\frac{\Delta}{\delta})^{2-t_1+h(\epsilon)^{1/4}}.\end{equation}
	We have, using \eqref{ifiofriufrirr09-02}
	\begin{equation}
		\label{thirdloss}
		|\cP_{\Delta,\bp_\rho}|\gtrsim (\delta/\Delta)^{\sqrt{h(\epsilon)}}|\cP_2\cap \bp_\rho|.
	\end{equation}

	We now focus our analysis inside a fixed $B_\Delta$, in preparation for the application of Lemma \ref{lem: high-low}. There are $r_3\sim |\cT_p^{\delta/\Delta}|$ distinct $\delta\times \Delta$-segments that intersect each $p\in\cP'\cap B_\Delta$. Recall that $\cT_p^{\delta/\rho_0}$ is a $(\delta/\rho_0,t_1,(\rho_0/\delta)^{\varphi(\epsilon)})$-set. By Lemma \ref{tradim}    $r_3\gtrsim (\delta/\rho_0)^{\varphi(\epsilon)}(\frac{\Delta}{\delta})^{t_1}.$ Using \eqref{twoendindisg2oweifruiirugidoi}, \eqref{ifiofriufrirr09-02}, then \eqref{inforr}  we find
	$$({\rho_0}/\delta)^{\varphi(\epsilon)}\le (\Delta/\delta)^{\frac{\varphi(\epsilon)}{h(\epsilon)^{3/2}}}\le (\Delta/\delta)^{\frac12h(\epsilon)^{1/4}},$$
	and thus $r_3\gtrsim (\Delta/\delta)^{t_1-\frac12h(\epsilon)^{1/4}}$.
	When combined with \eqref{pocirfireigig0} it shows that
	$$|\cP_2\cap B_\Delta|\gtrsim\frac1{r_3} (\frac{\Delta}{\delta})^{2+\frac12h(\epsilon)^{1/4}}.$$
	We may apply Lemma \ref{lem: high-low} to the $1/\Delta$-rescaled copy of $B_\Delta$, with  $\beta=\frac12h(\epsilon)^{1/4}$. We find a scale
	\begin{equation}
		\label{ifiofriufrirr09-03}
		\delta(\Delta/\delta)^{\beta/2}\lesssim \tilde{\delta}\le \Delta
	\end{equation}
	such that for at least a half of $p\in \cP_2\cap B_\Delta$, the heavy ball $B(c_p,\tilde{\delta})$ (we ignore the extra $(\Delta/\delta)^{\upsilon}$ in the definition of the radius, as it is negligible) intersects $\gtrsim r_3\tilde{\delta}/\delta$ of the $\delta\times \Delta$-segments. Each of these segments intersects  $\sim r/r_3$ tubes $T\in\cT$, since each $p$ intersects $\sim r$ tubes in $T\in\cT$, and the tubes corresponding to distinct segments are themselves distinct. We thus have
	$$\gtrsim r/r_3\times r_3\tilde{\delta}/\delta=r\tilde{\delta}/\delta$$
	tubes $T\in\cT$ intersecting $B(c_p,\tilde{\delta})$.
	
	We summarize \eqref{twoendindisg2oweifruiirugidoi}, \eqref{ifiofriufrirr09-02} and \eqref{ifiofriufrirr09-03}
	\begin{equation}
		\label{summary}
		\rho_0/\delta\lesssim (\rho/\delta)^{\frac{1}{h(\epsilon)}}\lesssim (\Delta/\delta)^{\frac{1}{h(\epsilon)^{3/2}}}\lesssim (\tilde{\delta}/\delta)^{\frac{4}{h(\epsilon)^{7/4}}}.
	\end{equation}

	Recall $\cT$ is (approximate) $\epsilon$-uniform. Let $N=|\cT[\bT_{\tilde{\delta}}]|$ be the number of tubes $\cT$ inside a nonempty $\bT_{\tilde{\delta}}\in\cT^{\tilde{\delta}}$ (this number is roughly the same, within $O_\epsilon(1)$-losses). Then the number $\br$ of $\bT_{\tilde{\delta}}\in\cT^{\tilde{\delta}}$ that intersect such a heavy $B(c_p,\tilde{\delta})$ is $\gtrsim \frac{r\tilde{\delta}}{N\delta}$. Also, there are $\sim |\cT|/N$ thick tubes $\bT_{\tilde{\delta}}\in\cT^{\tilde{\delta}}$.

	By Proposition~\ref{actthis} applied at scale $\tilde{\delta}$ with $\br\gtrsim \frac{r\tilde{\delta}}{N\delta}$ we find that the number of essentially distinct balls $B(c_p,\tilde{\delta})$ is
	$$\lessapprox \cS\cT(\tilde{\delta})\tilde{\delta}\frac{|\cT^{\tilde{\delta}}|^3}{(r\tilde{\delta}/N\delta)^3}\sim \cS\cT(\tilde{\delta})\tilde{\delta}(\delta/\tilde{\delta})^3\frac{|\cT|^3}{r^3}.$$
	Since, trivially, each $B(c_p,\tilde{\delta})$ contains at most $(\tilde{\delta}/\delta)^2$ $\delta$-squares, it follows that (recalling that $|\cT|\lesssim \delta^{-1}$)
	\begin{equation}
		\label{stillfinee}
		|\bigcup_{\bp_\rho}\cP_{\Delta,\bp_\rho}|\lessapprox (\tilde{\delta}/\delta)^2\cS\cT(\tilde{\delta})\tilde{\delta}(\delta/\tilde{\delta})^3\frac{|\cT|^3}{r^3}\sim \cS\cT(\tilde{\delta})\frac{|\cT|^2}{r^3}.
	\end{equation}
	Also, combining this with \eqref{bigchunk}, \eqref{thirdloss}, and then with  \eqref{summary}  we get
	$$|\cP_2|\lessapprox (\rho/\delta)^{h(\epsilon)^3}(\Delta/\delta)^{\sqrt{h(\epsilon)}}\cS\cT(\tilde{\delta})\frac{|\cT|^2}{r^3}\lessapprox(\tilde{\delta}/\delta)^{o_\epsilon(1)}\cS\cT(\tilde{\delta})\frac{|\cT|^2}{r^3}.$$
	From the three inequalities in \eqref{summary} we deduce that
	\begin{equation}
		\label{krgor[]oto5=-o3-=}
		\tilde{\delta}/\delta\gtrsim (\rho_0/\delta)^{\frac14h(\epsilon)^{7/4}}.
	\end{equation}
	To summarize, if we are in Case 2(a), we get
	$$|\cP_2|\lessapprox\max_{\tilde{\delta}\gtrsim \delta (\rho_0/\delta)^{c_\epsilon}}\left((\tilde{\delta}/\delta)^{b_\epsilon}\cS\cT(\tilde{\delta})\right)\frac{|\cT|^2}{r^3},$$
	for some $b_\epsilon\to0$, as $\epsilon\to0$, and with $c_\epsilon=\frac14h(\epsilon)^{7/4}$. This verifies \eqref{difjiougireugitugou}.
	\\
	\\
	\textbf{Case 2(b):}
	Suppose Item $(1)$  of Theorem~\ref{thm: delta-s}  happens for at least half of (and we will casually assume this holds for each)  $\bp_\rho\in \cP_{\rho,gen}$. We no longer need to be precise with quantifying the $o_\epsilon(1)$ exponents. As a general rule, $(\rho_0/\delta)^{o_\epsilon(1)}$-losses are acceptable in this case.

	We will derive two estimates for $|\cP_2|,$ and then we derive an upper bound for $W$.
	\\
	\\
	\textbf{First estimate:}

	Using \eqref{kjrgirigutiguitruh},
	\eqref{crifjiregjitgirt hioujryio}
	and
	\eqref{bigchunk} we get
	\begin{equation}
		\label{ehfurueigyuigytuiy}
		|\cP_2|\frac{r}{\tilde{r}}\lessapprox(\rho/\delta)^{o_\epsilon(1)}W|\cU|\lessapprox (\rho/\delta)^{o_\epsilon(1)}W\rho_0^{-1}\cS\cT(\rho_0) \frac{|\cT|}{\tilde{r}^3},
	\end{equation}
	where the last estimate is as in Case 1. Thus
	\begin{equation}\label{eq: up1}
		|\cP_2|\lessapprox \cS\cT(\rho_0)(\rho_0/\delta)^{o_\epsilon(1)}W \rho_0^{-1} \frac{|\cT|}{{r}\tilde{r}^2}.
	\end{equation}
	\\
	\\
	\textbf{Second estimate:}
	
	We double count incidences between mini-segments  $\cU_1=\cup_{\bp_\rho\in\cP_{\rho,gen}}\cU_{1}[\bp_{\rho}]$  and the squares in $\cup_{\bp_\rho\in\cP_{\rho,gen}}\cP[\bp_\rho]$. Each $p$ is incident to each $U_1\in\cU_{1,p}$.
	
	Recall that $|\cU_{1,p}|=|\cT_p^{\delta/\rho}|$, and this number is independent of $p$. Let $\tilde{r}_1$ be such that
	$$|\cU_{1,p}|=r/\tilde{r}_1.$$
	It follows from \eqref{bigchunk} that
	\begin{equation}\label{eq: narrow}
		|\cP_2|{r}/{\tilde{r}_1} \lesssim (\rho_0/\delta)^{o_\epsilon(1)} W|\cU_{1}|.
	\end{equation}

	Recall our notation  $\cT[\bp]:=\{ T\in \cT: T\cap \bp\neq \emptyset\}$.
	Subject to only a logarithmic loss, we may assume $ |\cT[\bp_\rho]|_{\rho}\sim \br$ for each $\bp_\rho \in \cP_{\rho,gen}$ and some fixed $\br$. More precisely, the value $\br$ corresponds to a subcollection of $ \cP_{\rho,gen}$  that contains a logarithmic fraction of $\cP_2$.
	
	By Proposition \ref{actthis}
	\[
	|\cP_{\rho,gen}| \lessapprox \cS\cT(\rho)\rho \frac{|\cT|_{\rho}^3}{\br^3}.
	\]
	We get that
	$$
	|\{(\bp_\rho,\bT_\rho)\in\cP_{\rho,gen}\times \cT^\rho:\;\bT_\rho\cap\bp_\rho\not=\emptyset\}|\lessapprox \cS\cT(\rho)\rho \frac{|\cT|_{\rho}^3}{\br^2} \sim \cS\cT(\rho) \rho \frac{|\cT|_{\delta}^2\cdot|\cT|_\rho}{|\cT[\bp_{\rho}]|_{\delta}^2}.
	$$

	Let us explain the $\sim$ from above. Recall we assumed $\cT$ is uniform. The collection of tubes $T$ in $\cT[\bp_{\rho}]$ is the disjoint union of the tubes  $\cT[\bT_{\rho}]$ inside the  $\sim\br$ fat tubes $\bT_{\rho}$ intersecting $\bp_{\rho}$. On the dual side (where uniformization is performed) the tubes in any given $\cT[\bT_{\rho}]$ coincide with the $\delta$-squares inside a fixed $\rho$-square. Their number is either zero or $|\cT|_\delta/|\cT|_{\rho}$. Thus
	$$ |\cT[\bp_{\rho}]|_{\delta}\sim \br|\cT|_\delta/|\cT|_{\rho}.$$
	We conclude that
	\begin{equation}
		\label{oeieif9if9rif9ri}
		|\{(\bp_\rho,T)\in\cP_{\rho,gen}\times \cT:\;T\cap\bp_\rho\not=\emptyset\}|\lessapprox  \cS\cT(\rho) \rho \frac{|\cT|_{\delta}^3}{|\cT[\bp_{\rho}]|_{\delta}^2}.
	\end{equation}
	We now double count the intersections between the collection $\cU_1$ and the tubes in $\cT$. Each $U_1\in\cU_1$ must be in some $\cU_{1,p}$, and thus it is contained in $|\cT_p|/|\cT_p^{\delta/\rho}|=\tilde{r}_1$ tubes $T\in\cT_p$. Thus
	$$\tilde{r}_1|\cU_1|\lesssim |\{(U_1,T)\in\cU_1\times \cT:\;U_1\subset T\}|\lessapprox  \cS\cT(\rho) \rho \frac{|\cT|_{\delta}^3}{|\cT[\bp_{\rho}]|_{\delta}^2}.$$
	This gives an upper bound for $|\cU_1|$. Combining this with \eqref{eq: narrow} we find
	\begin{equation}\label{eq: up2}
		|\cP_2|\lessapprox (\rho_0/\delta)^{o_\epsilon(1)} \cS\cT(\rho)\frac{W\rho}{r}\frac{|\cT|_{\delta}^3}{|\cT[\bp_{\rho}]|_{\delta}^2}.
	\end{equation}
	Taking the geometric average of \ref{eq: up1} and \ref{eq: up2} we find
	\begin{equation}\label{eq: up}
		|\cP_2|\lessapprox[\cS\cT(\rho_0)\cS\cT(\rho)]^{1/2}(\rho_0/\delta)^{o_\epsilon(1)}(\rho/\rho_0)^{1/2} W\frac{|\cT|_{\delta}^2}{r\tilde{r}|\cT[\bp_{\rho}]|_{\delta}}.
	\end{equation}
	\\
	\\
	\textbf{Upper bound for $W$:}
	
	Since Item (1)  of Theorem~\ref{thm: delta-s}  holds for $\bp_\rho\in\cP_{\rho,gen}$ we have
	$$|\cP\cap\bp_\rho|\lessapprox (\delta/\rho)^{t_1+o_\epsilon(1)}\frac{|\bar{\cU}_{1}[\bp_\rho]|^2}{(r/\tilde{r}_1)^2}.$$
	Then double counting incidences between squares and mini-segments shows that
	\begin{equation}\label{eq: A}
		W\lessapprox (\frac{\delta}{\rho})^{t_1+o_\epsilon(1)}  \frac{|\bar{\cU}_{1}[\bp_{\rho}]| }{r/\tilde{r}_1} \lessapprox (\frac{\delta}{\rho})^{t_1+o_\epsilon(1)} \frac{|\cT[\bp_{\rho}]|_{\delta}}{r}.
	\end{equation}

	The second inequality uses again the fact that each mini-segment is contained in at least $\tilde{r}_1$ tubes $T$. Tubes corresponding to distinct mini-segments lying inside  $\bp_\rho$ must  themselves be distinct.
	
	Combining  \eqref{eq: up} with \eqref{eq: A} we find
	\begin{align*}
		|\cP_2|&\lessapprox[\cS\cT(\rho_0)\cS\cT(\rho)]^{1/2} \frac{|\cT|^2}{r^3}\frac{r}{\tilde{r}}(\frac{\delta}{\rho})^{t_1}(\frac\rho{\rho_0})^{1/2+o_\epsilon(1)}\\&\lessapprox[\cS\cT(\rho_0)\cS\cT(\rho)]^{1/2}(\rho_0/\rho)^{t_1-\frac12} (\rho_0/\delta)^{o_\epsilon(1)}\frac{|\cT|^2}{r^3}.
	\end{align*}
	The last inequality is due to \eqref{jewie9r93rr8948t958t9}. This  verifies  \eqref{difjiougireugitugou} in Case 2(b).
	
	\section{Proof of Theorem \ref{t1}}
	\label{sec9}
	
	We will rely on the framework introduced in \cite{DD}, which we briefly recall below.
	
	Fix a smooth $\psi:\R^2\to \R$ satisfying $0\le \psi\le 1_{B(0,1)}$. Write $\psi_\delta(\xi)=\delta^{-2}\psi(\frac{\xi}{\delta})$ and $\mu_\delta=\mu*\psi_\delta$. Then \eqref{e1} can  be reformulated equivalently as
	\begin{equation}
		\label{e2}
		\|\widehat{\mu_{1/R}}\|_{L^6(\R^2)}^6\lesssim_\epsilon R^{2-2s-\frac{s}{4}+\epsilon}.
	\end{equation}
	Since $\mu_{1/R}$ is supported on the $1/R$-neighborhood of $\Gamma$,
	we partition this neighborhood into essentially rectangular regions $\theta$ with dimensions roughly $R^{-1/2}$ and $1/R$. We also partition the $1/R^{1/2}$-neighborhood into  essentially rectangular regions $\tau$ with dimensions roughly $R^{-1/4}$ and $R^{-1/2}$. Each $\theta$ sits inside some $\tau$.
	\smallskip
	
	Let $D,M,P$ be dyadic parameters, with $D$ and $P$  integers. We decompose $\mu$
	$$\mu=\sum_{D,M,P}\mu_{D,M,P},$$
	with each $\mu_{D,M,P}$ satisfying the following properties:
	\\
	\\
	(P1) for each $\theta$, either $\mu_{D,M,P}(\theta)=\mu(\theta)\sim MR^{-s}$, or  $\mu_{D,M,P}(\theta)=0$. We call $\theta$ {\em active} if it falls into the first category.
	\\
	\\
	(P2) each $\tau$ contains either $\sim P$ or no active $\theta$. We call $\tau$ {\em active} if it fits into the first category.
	\\
	\\
	(P3) there are $\sim D$ active $\theta's$, or equivalently, there are $\sim D/P$ active $\tau's$.
	\medskip
	
	We note that these properties together with \eqref{e4} force the following inequalities
	\begin{equation}
		\label{e5}
		M\lesssim R^{s/2}
	\end{equation}
	\begin{equation}
		\label{e6}
		DM\lesssim R^s
	\end{equation}
	\begin{equation}
		\label{e7}
		MP\lesssim R^{3s/4}.
	\end{equation}
	\smallskip

	For the rest of the argument we fix $D,M,P$ and let $F=R^{s-2}\widehat{\mu_{D,M,P}*\psi_{1/R}}$ be the $L^\infty$ upper normalized version, $\|F\|_\infty\lesssim 1$. We now recall the relevant estimates from \cite{DD}. Invoking the triangle inequality, it suffices to prove that
	\begin{equation}
		\label{e3}
		\|F\|_{L^6}^6\lesssim_\epsilon R^{4s-10-\frac{s}4+\epsilon}.
	\end{equation}
	Write $\mu_\theta$ for the restriction of $\mu$ to an active  $\theta$.
	Let $F_\theta=R^{s-2}\widehat{\mu_\theta*\psi_{1/R}}$ be the Fourier restriction of $F$ to $\theta$, so that
	$F=\sum_\theta F_\theta$.
	Then
	$$F_\theta=\sum_{T\in\cT_\theta}\langle F_\theta,W_T\rangle W_T+O(R^{-100}).$$
	Each $T$ is a rectangle (referred to as tube) with dimensions $R^{1/2}$ and $R$, with the long side pointing in the direction normal to $\theta$. The term $O(R^{-100})$ is negligible, and can be dismissed. Its role is to ensure that all $T$ sit inside, say,  $[-R,R]^2$. The wave packet $W_T$ is $L^2$ normalized, has spectrum inside (a slight enlargement of) $\theta$, and is essentially concentrated spatially in (a slight enlargement of) $T$. Thus, with $\chi_T$ being a smooth approximation of $1_T$, we have
	$$|W_T|\lesssim R^{-3/4}\chi_T.$$

	Given the dyadic parameter $\lambda$ we write
	$$\cT_{\lambda,\theta}=\{T\in\cT_\theta:\;|\langle F_\theta,W_T\rangle|\sim \lambda\}.$$
	We have
	\begin{equation}
		\label{e15}
		\lambda\le \lambda_{max}\sim MR^{-\frac54}.
	\end{equation}
	\eqref{e3} boils down to proving
	\begin{equation}
		\label{e3gtprgptphyphyt==}
		\|\sum_{T\in \cup_\theta\cT_{\lambda,\theta}}\langle F_\theta,W_T\rangle W_T\|_{L^6}^6\lesssim_\epsilon R^{4s-10-\frac{s}4+\epsilon}.
	\end{equation}
	For each rectangle $B$ with dimensions $\Delta$ and $R$, and with orientation identical to that of the tubes in $\cT_\theta$, we have  the estimate
	\begin{equation}
		\label{e12}
		|\{T\in\cT_{\lambda,\theta}:\; T\subset B\}|\lesssim (\frac{\Delta}{\sqrt{R}})^{1-s}\frac{MR^{\frac{s-5}{2}}}{\lambda^2}.
	\end{equation}
	In particular,
	\begin{equation}
		\label{e11}
		|\cup_{\theta}\cT_{\lambda,\theta}|\lesssim \frac{MD}{\lambda^2R^2}.
	\end{equation}
	Each arc of length $r\lesssim 1$ on $\Gamma$ intersects at most
	\begin{equation}
		\label{rifururgurtgrthi90hi90}
		\lesssim (r\sqrt{R})^s\frac{R^{s/2}}{M}\end{equation}
	many $\theta$.
	
	We derive two estimates.
	\\
	\\
	\textbf{First estimate:} We first use decoupling (\cite{BD}) into caps $\theta$ (combined with property (P3)), then \eqref{e11} to get
	\begin{align*}
		\|\sum_{\theta\text{ active}}\sum_{T\in\cT_{\lambda,\theta}}\langle F_\theta,W_T\rangle W_T\|_6^6&\lesssim_\epsilon R^\epsilon D^2\sum_{\theta\text{ active}}\|\sum_{T\in\cT_{\lambda,\theta}}\langle F_\theta,W_T\rangle W_T\|_6^6\\&\lesssim_\epsilon R^\epsilon D^2R^{-9/2}\lambda^6\sum_{\theta\text{ active}}\|\sum_{T\in\cT_{\lambda,\theta}}\chi_T\|_6^6\\&\lesssim_\epsilon R^\epsilon D^2R^{-9/2}\lambda^6R^{3/2}|\cup_{\theta}\cT_{\lambda,\theta}|\\&\lesssim_\epsilon R^\epsilon\frac{MD^3\lambda^4}{R^{5}}.
	\end{align*}Using  \eqref{e15} we conclude with
	\begin{equation}
		\label{oirpofiigtigihoyt0ho0}
		\|\sum_{\theta\text{ active}}\sum_{T\in\cT_{\lambda,\theta}}\langle F_\theta,W_T\rangle W_T\|_6^6\lesssim_\epsilon R^\epsilon\frac{M^5D^3}{R^{10}}.\end{equation}
	\textbf{Second estimate:}
	\begin{defn}Consider a collection consisting of $\sim r'/P'$ many caps $\tau$, each of which contains $\sim P'$ many $\theta$. Call $\cD$ the collection of all these $\theta$, and note that $|\cD|\sim r'$.  	
		Let $\cB\subset \cup_{\theta\in\cD}\cT_{\lambda,\theta}$ be a collection of tubes $T$  intersecting a fixed  $\sqrt{R}$-square $q$. Note that there are $O(1)$ many possible $T$ for each $\theta\in\cD$.
		
		We call $\cB$ a bush with data $(P',r',q)$.
	\end{defn}	
	We recall the following bush estimate proved in \cite{DD}, that uses decoupling into the larger caps $\tau$.

	\begin{prop}
		\label{p5}
		Let $\cB$ be a bush with data $(P',r',q)$.  Then
		\begin{equation}
			\label{dokoregotrgo[p]oh}
			\|\sum_{T\in \cB}\langle F_\theta,W_T\rangle W_T\|_{L^6(q)}^6\lesssim_\epsilon R^{\epsilon-7/2}(r')^3(P')^2\lambda^6.\end{equation}
	\end{prop}
	Now, for each dyadic $r\ge 1$, let
	$$\cQ_r=\{q:\;q\text { intersects }\sim r\text{ tubes in } \cup_\theta\cT_{\lambda,\theta}\}.$$
	To estimate $|\cQ_r|$, we note that, due to \eqref{e12} and \eqref{rifururgurtgrthi90hi90}, the rescaled copies of the tubes $T\in \cup_\theta\cT_{\lambda,\theta}$ satisfy the hypotheses of Theorem \ref{thm: mainbettt} with $\delta=R^{-1/2}$, $K_1\sim \frac{R^{s/2}}{M}$ and $K_2\sim \frac{MR^{\frac{s-5}{2}}}{\lambda^2}.$ Combining this proposition with \eqref{e11} shows that
	\begin{equation}
		\label{dokoregotrgo[p]ohhpoihopyih}
		|\cQ_r|\lesssim_\epsilon \frac{R^{\frac12+\epsilon}}{r^3}(\frac{R^{s-\frac52}}{\lambda^2})^2\frac{MD}{\lambda^2R^2}.
	\end{equation}
	For each $q\in\cQ_r$, the collection $\cB(q)=\{T\in \cup_\theta\cT_{\lambda,\theta}:\;q\cap T\not=\emptyset\}$ can be partitioned into $\lessapprox 1$ many bushes with data $(P',r',q)$ satisfying $P'\le P$, $r'\le r$. Thus, combining \eqref{dokoregotrgo[p]oh} with \eqref{dokoregotrgo[p]ohhpoihopyih} we find
	$$\|\sum_{T\in \cup_\theta\cT_{\lambda,\theta}}\langle F_\theta,W_T\rangle W_T\|_{L^6(\cup_{q\in\cQ_r q})}^6\lesssim_\epsilon R^{\epsilon-7/2}r^3P^2\lambda^6R^{\frac12}(\frac{R^{s-\frac52}}{\lambda^2})^2\frac{MD}{\lambda^2R^2}\frac1{r^3}=R^{2s+\epsilon}\frac{MDP^2}{R^{10}}.$$
	Since there are $\lessapprox 1$ dyadic values of $r$, we get our second estimate
	\begin{equation}
		\label{lkporkforigoitpgoitopgi[tpo]}
		\|\sum_{T\in \cup_\theta\cT_{\lambda,\theta}}\langle F_\theta,W_T\rangle W_T\|_{L^6}^6\lesssim_\epsilon R^{2s+\epsilon}\frac{MDP^2}{R^{10}}.
	\end{equation}
	When $M\sim R^{s/2}$, \eqref{lkporkforigoitpgoitopgi[tpo]} together with \eqref{e6} and \eqref{e7} imply that
	$MDP^2\lesssim R^{\frac{3s}{2}}$, leading to the extra saving $R^{s/4}$ (that means, the exponent $\frac{7s}{2}-10+\epsilon$) in \eqref{e3gtprgptphyphyt==}. But when $M$ is small, \eqref{lkporkforigoitpgoitopgi[tpo]} by itself gives no useful upper bound. Instead, taking the geometric average of \eqref{oirpofiigtigihoyt0ho0} and \eqref{lkporkforigoitpgoitopgi[tpo]} we find
	$$\|\sum_{T\in \cup_\theta\cT_{\lambda,\theta}}\langle F_\theta,W_T\rangle W_T\|_{L^6}^6\lesssim_\epsilon R^{s+\epsilon-10}(MD)^2MP\lesssim_\epsilon R^{4s-10-\frac{s}{4}+\epsilon}.$$
	The last inequality follows from \eqref{e6} and \eqref{e7}. The extreme case for which our argument is tight is when $M\sim R^{\frac{3s}{8}}$, $D\sim R^{\frac{5s}{8}}$, $P\sim R^{\frac{3s}{8}}$.
	\begin{remark}
		\label{jhhdhfuihugururgurguiy}
		When the measure is $AD$-regular, $M$ is forced to be $\sim R^{s/2}$. Moreover, as explained in \cite{DD}, there is an extra gain in the bush inequality \eqref{dokoregotrgo[p]oh} that comes from  estimating non trivially the energy of flat AD-regular sets. As a result, the exponent in \eqref{e1} for such measures can be lowered to  $2-2s-\frac{s}{2}-\beta$, for some unspecified $\beta>0$. Similarly, in the context of Theorem \ref{wejfiuriofuriurioeug} we get
		$$\mathbb{E}_{3,\delta}(S)\lesssim \delta^{-\frac{7s}{2}+\beta}.$$
	\end{remark}


\begin{thebibliography}{99}	
		\bibitem{BB} E. Bombieri and J. Bourgain,  {\em A problem on sums of two squares}, Int. Math. Res. Not. IMRN 2015, no. 11, 3343-3407
		\bibitem{BD} J. Bourgain and C. Demeter, {\em The proof of the $l^2$ decoupling conjecture} Ann. of Math. (2) 182 (2015), no. 1, 351-389	
		\bibitem{DD} S. Dasu and C. Demeter {\em Fourier decay for curved Frostman measure}, Proc. Amer. Math. Soc. 152 (2024), no. 1, 267-280
		\bibitem{Dembook} C. Demeter {\em Fourier restriction, decoupling and applications},  Cambridge Stud. Adv. Math., 184 Cambridge University Press, Cambridge, 2020, xvi+331 pp.
		\bibitem{DvGo} Z. Dvir and S. Gopi. {\em On the number of rich lines in truly high dimensional sets}, Proc. of 31st
		International Symposium on Computational Geometry (SoCG 2015). Vol 34:584-598, 2015
		\bibitem{FOP} K. F\"assler, T. Orponen and A. Pinamonti, {\em Planar incidences and geometric inequalities in the Heisenberg group}, arXiv:2003.05862
		\bibitem{FuRe} Yuqiu Fu and Kevin Ren, {\em Incidence estimates for $\alpha$-dimensional tubes and $\beta$-dimensional balls in $\R^2$.}
		J. Fractal Geom. (to appear)
		\bibitem{GSW}L. Guth, N. Solomon and H. Wang,  {\em Incidence estimates for well spaced tubes},
		Geom. Funct. Anal.29(2019), no.6, 1844-1863
		\bibitem{Gu} L. Guth, A. Iosevich, O. Yumeng and H. Wang {\em On Falconer's distance set problem in the plane.} Invent. Math. 219 (2020), no. 3, 779-830
		\bibitem{KT} N. Katz and T. Tao {\em Some connections between Falconer's distance set conjecture and sets of Furstenburg type},  New York J. Math. 7 (2001), 149-187.
		\bibitem{O} T. Orponen, {\em Additive properties of fractal sets on the parabola}, Ann. Fenn. Math.48(2023), no.1, 113-139
		\bibitem{OS} T. Orponen and P. Shmerkin, {\em On the Hausdorff dimension of Furstenberg sets and orthogonal
			projections in the plane}, arXiv:2106.03338
		\bibitem{OS2} T. Orponen and P. Shmerkin, {\em Projections, Furstenberg sets, and the abc sum-product problem}, arXiv:2301.10199
		\bibitem{O2} T. Orponen, C. Puliatti and A.  Py\"or\"al\"a {\em On Fourier transforms of fractal measures on the parabola}, arXiv:2401.17867
		\bibitem{RW} K. Ren and H. Wang {\em Furstenberg sets estimate in the plane}, arXiv:2308.08819
		\bibitem{Sh} P. Shmerkin {\em A non-linear version of Bourgain's projection theorem},
		J. Eur. Math. Soc. 25(2023), no.10, 4155-4204
		\bibitem{SW} P. Shmerkin and H. Wang, {\em On the distance sets spanned by sets of dimension $d/2$ in $\R^d$}, arXiv:2112.09044
		\bibitem{SzTr} E. Szemer\'edi and W. Trotter, {\em Extremal problems in discrete geometry}, Combinatorica 3 (1983), no. 3-4, 381-392	
	\end{thebibliography}
\end{document}